\documentclass[10pt]{article}
\usepackage{amsmath,amsthm,amssymb}
\usepackage{hyperref}
\usepackage{mathdots}

\numberwithin{equation}{section}
\numberwithin{figure}{section}

\newtheorem{theorem}{Theorem}[section]
\newtheorem{corollary}[theorem]{Corollary}

\newtheorem{lemma}[theorem]{Lemma}
\newtheorem{proposition}[theorem]{Proposition}

\newtheorem{definition}[theorem]{Definition}
\let\olddefinition\definition
\renewcommand{\definition}{\olddefinition\normalfont}
\newtheorem{remark}[theorem]{Remark}
\newtheorem{example}[theorem]{Example}

\newcommand{\refpart}[1]{{\it (#1)}} 
\renewcommand\ge\geqslant
\renewcommand\geq\geqslant
\renewcommand\le\leqslant
\renewcommand\leq\leqslant

\newcommand{\RR}{{\ensuremath{\mathbb{R}}}}

\newcommand{\ma}{{\ensuremath{\beta}}} 
\newcommand{\mb}{{\ensuremath{\gamma}}}

\newcommand{\cS}{{\ensuremath{\mathcal{S}}}}
\newcommand{\cM}{{\ensuremath{\mathcal{M}}}}

\newcommand{\cQ}{{\ensuremath{\mathcal{Q}}}}
\newcommand{\cR}{{\ensuremath{\mathcal{R}}}}

\newcommand{\cE}{{\ensuremath{\mathcal{E}}}} 

\newcommand{\cH}{{\ensuremath{\mathcal{H}}}}

\newcommand{\cZ}{{\ensuremath{\mathcal{Z}}}} 
\newcommand{\cB}{{\ensuremath{\mathcal{B}}}}
\newcommand{\cP}{{\ensuremath{\mathcal{P}}}} 

\newcommand{\img}{\mathrm{im\ }}

\newcommand{\sgma}{\Omega}

\def\vb{\mathbf{v}}

\def\crossing{e} 
\newcommand{\hh}{{\hspace{7pt}}}
\newcommand{\hg}{{\hspace{6pt}}}

\newcommand{\fc}[1]{\mbox{\em #1}}
\newcommand{\tdeg}[1]{\mbox{tdeg} \;#1}
\newcommand{\ii}{K}

\title{Building geometrically continuous splines}

\author{Raimundas Vidunas\footnote{Department of Mathematical Informatics, University of Tokyo, Bunkyo-ku 113-8656, Japan. E-mail: {\sf rvidunas@gmail.com}}} 

\date{}

\begin{document}

\maketitle

\begin{abstract}
With the renewed and growing interest in geometric continuity in mind,
this article gives a general definition of geometrically continuous polygonal surfaces and 
geometrically continuous spline functions on them. 
Polynomial splines defined by $G^1$ gluing data in terms of rational functions
are analyzed further. A general structure for a spline basis is defined,
and a dimension formula is proved for spline spaces of bounded degree 
on polygonal surfaces made up of rectangles and triangles.
Lastly, a comprehensive example is presented, 
and practical perspectives of geometric continuity are discussed. 
The whole objective of the paper is to put forward a modernized, practicable framework 
of modeling with geometric continuity.\\

\noindent{\bf Keywords.} Geometric continuity, Polynomial splines, Dimension.
\end{abstract}

\section{Introduction}

There is growing interest in geometric continuity,
especially in Computer Aided Geometric Design \cite{flex}, \cite{beccari}, \cite{Bercovier}
and isogeometric analysis \cite{Kapl2014}, \cite{Peters14}. 
Nevertheless, apprehension of its technical details, applications,
potential 
appears to be variable and inconsistent
among active researchers, perceivably impeding communication and 
faster development of the subject. Substantially updating, improving 
the presentation in author's PhD thesis \cite[\S 6]{raimundas}, 
this article aims to facilitate a more uniform understanding of practical requisites 
of geometric continuity.

Geometric continuity \cite{Boehm87}, \cite{derose}, 
\cite{PetersHandbook}  is a 
general technique to produce visually smooth surfaces in Computer Aided Geometric Design (CAGD).
A resulting 
surface is typically made of parametric patches in $\RR^3$,
with each patch defined by a (usually) polynomial map from a polygon in  $\RR^2$,
in such a way that the patches fit each other continuously along the edges.
{\em Geometric continuity} requires continuousness of tangent planes 
(and possibly, of curvature, higher osculating spaces) along the glued edges and around the formed vertices. 
The special case of {\em parametric continuity} occurs 
when the parametrizing polygons are situated next to each other in $\RR^2$,
forming a polygonal mesh. A whole surface in $\RR^3$ is then parametrized
by a single map from the polygonal mesh, and the single map is required to be
$C^r$ continuous (with chosen $r\ge 1$ differential continuity).
Parametric continuity is not adequate for modeling surfaces
of arbitrary topology \cite{LoopDerose}.

Direct geometric continuity conditions in $\RR^3$ lead to complex 
models that are not simple to use and modify \cite{DeRose90}, \cite{Peters96i}, 
\cite{Che05}. 
The other approach \cite{derose}, \cite{Hahn87}
is to define a geometric data structure of glueing a set of polygons
abstractly, in manifold theoretic terms. 
A concrete surface in $\RR^3$ is then viewed as a {\em realization} 
of an abstract 
polygonal manifold $\cM$ by a map $\cM\to\RR^3$
defined by a triple $(f_1,f_2,f_3)$ of suitably 
defined {\em geometrically continuous $G^r$ functions} on $\cM$. 
The actual goal is to define the spaces of $G^r$ functions,
sufficiently rich so that they could 
generate satisfactorily smooth surfaces in $\RR^3$.
Of particular interest are $G^r$ {\em polynomial splines}, 
that is, $G^r$ functions that restrict to polynomial functions on each polygon.
 
Geometric continuity of abstract polygonal surfaces is customarily defined
with reference to {\em differential manifolds} \cite{Wiki}  
so that the $G^r$ functions on the polygonal surfaces would correspond
to the $C^r$ functions on a differential surface.
The gluing data is then the transition maps between open neighborhoods 
of the glued polygonal edges \cite{derose}, or the jet germs of the transition maps \cite{PetersHandbook},
or the induced transformations of tangent spaces \cite{Hahn87}. 
In terms of direct $G^r$ gluing in $\RR^3$, this corresponds to fixing
the {\em shape parameters} \cite{piegl87} 
and working with induced linear relations 
for the coefficients of the component functions $(f_1,f_2,f_3)$.

The geometric continuity conditions and theoretical grounding in differential geometry
are well understood in principle \cite{Gregory89}, \cite{PetersHandbook}, 
\cite{Zheng95}.
Sufficiency of the known constrains for smooth realizations has not been proved yet.
In particular, Peters and Fan \cite{Peters2010} noticed a topological balancing restriction 
around {\em crossing vertices} (see Definition \ref{def:crossing} and \S \ref{sec:crossing} here)
in a context of $G^1$ gluing of rectangles. The first new result 
of this paper is a generalization of this restriction with Theorem \ref{cond:comp2}. 
The restriction depends on the assumption of polynomial (or more generally, $C^2$)
specializations of the $G^1$ functions on the polygons.  

Other contributions of this paper are: 
a general definition of $G^1$ polygonal surfaces and $G^1$ functions on them (in \S \ref{sec:gsplines}),
a dimension formula for certain spaces of $G^1$ polynomial splines of bounded degree (in \S \ref{eq:spdim}),
and an extensive, illuminating example in \S \ref{sec:poctah}.
A general strategy of building a basis $G^1$ splines of bounded degree is presented in \S \ref{sec:generate},
and demonstrated on $G^1$ polygonal surfaces made up of rectangles and triangles
(in \S \ref{sec:boxdelta}, \S \ref{sec:poctah}). 
In addition, \S \ref{sec:practical} 
discusses sufficiency of defined spline spaces for smooth generalization in $\RR^3$.

The purpose of the whole paper is to encourage a uniform level of 
familiarity and communication in the accelerating field of geometric continuity.
There is no shortage of general definitions \cite{GuHeQin06}, \cite{Tosun:2011}, 
frameworks \cite{beccari}, \cite{Vecchia:2008}, 
special constructions \cite{Prautzsch:1997}, \cite{Reif:1995} 
deploying geometric continuity. But it becomes harder to distinguish their relative merits,
level of generality, addressed issues, triviality or prominence of technical details.
There is a lack of mutual comparisons and extended examples in the literature.
Recent communication on the subject revealed odd patterns of superficial understanding
of basic routines, incongruent focus on purportedly important questions or practicality.
My updated experience of writing PhD thesis \cite{raimundas} and \cite{VidunasAMS}
is worth sharing, apparently.

We start in \S \ref{sec:geoglue} with a review of $G^r$ continuity conditions,
updated with the restriction on the crossing vertices generalizing \cite{Peters2010}.
This serves as a preparation for our general definitions of $G^1$ polygonal surfaces 
and splines in \S \ref{sec:gsplines}. 
The focus is not on the most comfortable theory building, but on 
a transparent overview of technical details and on
grasping a full potential of geometric continuity.
One underlying suggestion is that grounding in differential geometry should not be taken too seriously.
Ultimately, a designer does not need to be deeply aware of conformity with differential geometry.
The key objective for expansive applications is defining a spline space with preferably minimal gluing data,
and estimating the quality of realizations in $\RR^3$ by those splines.

Section \ref{sec:freedom} presents essential technical details for understanding
a basis structure of spline spaces and computing their dimension.
Section \ref {sec:boxdelta} proves general formulas for the dimension of spline spaces
on $G^1$ polygonal surfaces made up exclusively of rectangles and triangles.
Section \ref{sec:poctah} gives an extensive, illustrative  
example of a $G^1$ polygonal surface, and demonstrates computation of spline bases on it.
Finally, Section \ref{sec:practical} summarizes the results and offers a few practical perspectives  
of using geometric continuity.\\

\noindent
{\bf Acknowledgment. } This article originated as an alternative version to the work \cite{MVV15}
with Bernard Mourrain and Nelly Villamizar. 
The author appreciates rich discussions, shared literature and revived interest 
in geometric continuity, provoking this work in full.

\section{Geometrically continuous gluing}
\label{sec:geoglue}

This section is devoted to defining the $G^r$ gluing data and restrictions on it, 
with more attention to the $r=1$ case. 
This is a preparation for defining the {\em $G^1$ polygonal surfaces}
and {\em $G^1$ spline functions} on them in \S \ref{sec:gsplines}.
The $G^1$ gluing data is generally clear from differential geometry and CAGD practice, 
though a comprehensive set of restrictions to ensure possibility of $G^r$ gluing around vertices 
was still not formalized.
We spell out different representations of geometric continuity conditions and full terminology, 
preparing for inclusive definitions of \S \ref{sec:gsplines}. 

Building geometric continuous surfaces by defining at first 
some minimal gluing data  (or an abstract surface, polygonal complex) %
to patch a collection of polygons has a definite motivation from differential geometry. 
We recall a standard definition \cite{Warner83} of differential surfaces.
\begin{definition}
Let $r$ denote a positive integer, and let $\ii$ denote a finite set.
A $C^r$ {\em differential surface}  is defined
as a connected topological Hausdorff \cite{Wiki} manifold $\cS$ with 
a collection $\{(V_k,\psi_k)\}_{k\in\ii}$ such that
\begin{enumerate}
\item $\{V_k\}_{k\in\ii}$ is an open covering of $\cS$.
\item Each $\psi_k$ is a homeomorphism $\psi_k \!:\! U_k\!\to V_k$, 
where $U_k$ is an open set in $\RR^2$.
\item For distinct $k,\ell\in\ii$ such that $V_{k,\ell}:=V_k\cap V_{\ell}$ 
is not an empty set, let $U_{k,\ell}:=\psi_k^{-1}(V_{k,\ell})$,
$U_{\ell,k}:=\psi_\ell^{-1}(V_{k,\ell})$.
Then the map \mbox{$\psi_\ell^{-1}\circ\psi_k:U_{k,\ell}\to U_{\ell,k}$} 
is required to be a $C^r$-diffeomorphism. 
\end{enumerate}
The collection $\{(V_k,\psi_k)\}_{k\in I}$ is a $C^r$ {\em atlas} on $\cS$,
and the maps $\psi_\ell^{-1}\circ\psi_k$ are called {\em transition maps}.
A function $f:\cS\to\RR$ is {\em $C^r$ continuous} 
on $\cal S$  if for any $k\in\ii$ 
the function $f \circ \psi_{k}^{-1}: U_{k} \to \RR$ is $C^r$ continuous.
\end{definition}
In contrast to differential geometry, CAGD aims to build surfaces from {\em closed polygons}
rather than from open sets. It is easy to underestimate the technical difficulty of properly
defining the gluing data, $G^r$ functions and relating them to differential manifolds, $C^r$ functions. 

Given a collection of polygons and homeomorphisms between their edges,
a topological surface and continuous functions on it are defined easily.
\begin{definition} \label{def:g0complex}
Let $\ii_0,\ii_1$ denote finite sets.
A {\em $G^0$ polygonal surface} $\cal M$ is a pair 
$(\{\sgma_k\}_{k\in \ii_0},\{\lambda_k\}_{k\in \ii_1})$ such that
\begin{enumerate}
\item $\{\sgma_k\}_{k\in \ii_0}$ is a collection of (possibly coinciding) convex 
closed polygons in $\RR^2$. 
\item $\{\lambda_k\}_{k\in \ii_1}$ is a collection of homeomorphism 
$\lambda_k:\tau_{k}\to{\tau}'_{k}$ between pairs of polygonal edges 
$\tau_{k}\subset\sgma_{\ell_k}$, ${\tau}'_{k}\subset{\sgma}_{\ell'_k}$.
\item Each polygonal edge can be paired with at most one other edge, 
and it cannot be paired with itself. 
\item The equivalence relation on the polygons generated
by the incidences $\lambda_k$ of glued edges has exactly one orbit. 
\end{enumerate}
A $G^0$ polygonal surface $\cM$ is called {\em linear} if all $\lambda_k$ 
are affine-linear homeomorphisms. 
\end{definition}
\begin{definition}
Any $G^0$ polygonal surface $\cM$ has a structure of a 
topological surface, as the disjoint union of the polygons with some points identified 
to equivalence classes by the homeomorphisms $\lambda_k$.
Condition \refpart{iv} means that this topological surface is connected.
The identifications of polygonal edges and vertices give a combinatorial complex of 
{\em polygons}, {\em edges} and {\em vertices} (as {\em faces} of the complex \cite{Ziegler00}) 
on the topological surface.
A common edge on $\cM$ defined by $\lambda_k$ in \refpart{ii}
will be denoted by $\tau_{k}\sim {\tau}'_{k}$. These edges are called an {\em interior edges} of $\cM$.
The other edges are {\em boundary edges}. 
We use the set notation $\{P_1,\ldots,P_n\}$
to denote an equivalence class of identified polygonal vertices. 
A vertex is a {\em boundary vertex} is on a boundary edge, and it is an {\em interior vertex} otherwise.
\end{definition}
\begin{definition}
A {\em  continuous function} on a $G^0$ polygonal surface $\cM$ 
is defined by assigning a continuous
function $f_k$ to each polygon $\sgma_k$ ($k\in N_0$) 
so that their restrictions to the polygonal
edges are compatible with the homeomorphisms $\lambda_k$ ($k\in N_1$).
The set of continuous functions on $\cM$ is denoted by $G^0(\cM)$.
\end{definition}
Our intention is to define sets $G^r(\cM)$ of $G^r$ {\em geometrically continuos functions}
on a polygonal surface $\cM$, such that a generic triple $(f_1,f_2,f_3)$
of functions from the same $G^r(\cM)$ would give a map $\cM\to\RR^3$ whose image is
a $G^r$ continuous surface in $\RR^3$ as defined in CAGD \cite{Hahn87}, \cite{Peters96i}.  
In particular, the $G^1$ continuous surfaces have {\em tangent plane continuity}
of polygonal patches along the glued edges. 
We will refer to a map $\cM\to\RR^3$ with a $G^r$ continuous image as 
a {\em $G^r$ smooth realization} of $\cM$.
The described topological surface $\cM$ should remain the underlying topological space.
To keep technical details simpler, we consider only linear polygonal surfaces
and throughly analyze only the $r=1$ case.  
We will use the following technical definitions.
\begin{definition} \label{def:c1cl}
 By a {\em $C^r$ function} on a closed polygon $\sgma\subset\RR^2$
 we mean a continuous function on $\sgma$ that can be extended to a $C^r$
 on an open set containing $\sgma$. We denote the space of these functions by $C^r(\sgma)$.
We could require weaker condition of being $C^r$ functions on the interior of 
$\sgma$ and continuous functions on $\sgma$. However,
in CAGD applications one typically allows the polygonal restrictions $g_1,g_2$
to be in a specific class of $C^\infty$ functions (on open neighborhoods 
of $\sgma_1$ or $\sgma_2$), such as polynomial, rational or trigonometric functions.
\end{definition}
\begin{definition} \label{def:stcoor}
Let $\sgma$ denote a convex closed polygon in $\RR^2$,
and let $\tau$ denote an edge of $\sgma$.
By a {\em coordinate system attached to $\tau$} we mean
a pair $(u,v)$ of 
linear functions on $\RR^2$ such that:
\begin{itemize}
\item $u$ attains the values $0$ and $1$ at the endpoints of $\tau$;
\item $v=0$ on the edge $\tau$, and $v>0$ on the interior of $\sgma$.
\end{itemize}
If additionally $(v,u)$ is a coordinate system attached to the other edge
of $\sgma$ at the endpoint $u=0$, the pair $(u,v)$ is called
a {\em standard coordinate system} attached to $\tau$ or to the end vertex $u=0$.
There are exactly two standard coordinate systems attached to $\tau$.
\end{definition}
\begin{definition} \label{def:inout}
Let $\sgma$ denote a convex closed polygon in $\RR^2$.
Let $\tau$ denote an edge of $\sgma$, 
and let $L\subset \RR^2$ denote the line containing $\tau$.
The {\em inward side (from $\tau$, towards $\Omega$)}
is the open half-plane from $L$ that contains the interior of $\Omega$.
The {\em outward side (from $\tau$, away from $\Omega$)}
is the other open half-plane from $L$.
\end{definition}

\subsection{Gluing two polygons}
\label{sec:edgeglue}

There are two basic problems in constructing geometrically continuous surfaces:
gluing of two patches along an edge,  and gluing several patches around a vertex. 
The former 
is a relatively straightforward routine \cite{derose}, \cite{Zheng95}.
The analogy with differential geometry can be followed closely using transition maps.
Simple pictures (from \cite{Hahn87}, \cite{Boehm87}, for example) can be useful
in supplementing this discussion.

Let $\sgma_1,\sgma_2$ denote two convex closed polygons in $\RR^2$, and let
$\sgma^0_{1}\subset\sgma_{1}$, $\sgma^0_{2}\subset{\sgma}_{2}$ denote their interiors.
Let $\lambda:\tau_{1}\to{\tau}_{2}$ denote a linear homeomorphism 
between their edges $\tau_{1}\subset\sgma_{1}$, ${\tau}_{2}\subset{\sgma}_{2}$,
and let $\cM_0$ denote the $G^0$ polygonal surface $(\{\Omega_1,\Omega_2\},\{\lambda\})$. 
Let $U_1\supset \tau_1$, $U_2\supset \tau_2$ denote open neighborhoods of the two edges.
Suppose we have a $C^r$ diffeomorphism $\psi_0:U_1\to U_2$ such that 
\begin{itemize}
\item[\refpart{A1}] It specializes to $\lambda:\tau_{1}\to{\tau}_{2}$ 
when restricted to the edge $\tau_1$;
\item[\refpart{A2}] It maps the interior $\sgma^0_{1}\cup U_1$ to 
the exterior $U_2\setminus \sgma_{2}$ of the polygon $\sgma_2$.
\end{itemize}
We have a $C^r$ differential surface $\cM_r$ by considering the disjoint union of 
$\sgma^0_{1}\cup U_1$ and $\sgma^0_{2}\cup U_2$ modulo the equivalence relation
defined by the transition map $\psi_0$. The disjoint union of 
$\sgma^0_{1}$ and $\sgma^0_{2}$ is a proper subset of $\cM_r$.
The convexity assumption makes condition \refpart{A2} possible.

A $C^r$ function on $\cM_r$ is given by a pair $(f_1,f_2)$ of functions
$f_1\in C^r(\sgma_{1})$, 
$f_2\in C^r(\sgma_{2})$ 
such that 
\begin{equation} \label{eq:crcont}
f_1=f_2\circ\psi_0 \qquad \mbox{on} \quad U_1.
\end{equation}
We can define the {\em geometrically continuous $G^r$ functions} on $\cM_0$
as 
pairs $(g_1,g_2)$ of functions $g_1\in C^0(\sgma_1)$,  $g_2\in C^0(\sgma_2)$ such that 
$(g_1|_{\sgma_1},g_2|_{\sgma_2})$ is a restriction of a $C^r$ function on $\cM_r$.
We seek a more constructive 
definition.

Let $(u_1,v_1)$, $(u_2,v_2)$ denote coordinate systems attached to $\tau_1$, $\tau_2$,
respectively. We assume that $\lambda$ identifies the endpoints $u_1=0$ and $u_2=0$;
then it identifies the endpoints $u_1=1$ and $u_2=1$.  
Let $\ell_1$ denote the line in $\RR^2$ containing $\tau_1$.
The transition map $\psi_0$ has Taylor expansions \cite{Wiki} in \mbox{$u_1-u_1(P)$}, 
$v_1$  of order $r$ at each point $P\in\tau_1$.
In particular, we can write the action of $\psi_0$ as
\begin{equation} \label{eq:thetat}
{u_1\choose v_1} \mapsto  {u_2\choose v_2}={u_1\choose 0}+
\sum_{k=1}^r \theta_k(u_1)v_1^k+\psi_r(u_1,v_1),
\end{equation}
where each $\theta_k(u_1)$ is a $C^{r-k}$ map to $\RR^2$ from the open 
subset $\ell_1\cap U_1$ of $\ell_1$, and $\psi_r(u_1,v_1)$ 
is a $C^r$ map $U_1\to \RR^2$ that vanishes at each point of $\tau_1$ 
together with all derivatives of order $\le r$. 
If $(f_1,f_2)$ is a $C^r$ function on $\cM_r$, then $f_1$ has Taylor expansions 
in \mbox{$u_1-u_1(P)$}, $v_1$ of order $r$ at each point $P\in\tau_1$, and similarly,
$f_2$ has Taylor expansions in \mbox{$u_2-u_2(Q)$}, $v_2$ of order $r$ 
at each point $Q\in\tau_2$. Condition (\ref{eq:crcont}) gives
invertible linear transformations of the Taylor coefficients of $f_1$ at each $P\in\tau_1$ 
to the Taylor coefficients of $f_2$ at $Q=\lambda(P)$. 
The linear transformations are determined 
by the $r$-th order Taylor expansion of $\psi_0$ in  (\ref{eq:thetat}) at each $P\in\tau_1$.

The order $r$ Taylor expansions 
at a single point can be interpreted as
{\em $r$-th order jets} \cite{Wiki} of $\psi_0$, $f_1$ or $f_2$. The jets are equivalence classes
of $C^r$ functions or maps, with the equivalence defined as having the same Taylor expansion
of order $r$. The jets can be represented by polynomials (or polynomial vectors)
in formal variables $\tilde{u},\tilde{v}$  of degree $\le r$. 
The linear spaces of jets are denoted by $J^r_P(U_1,\RR^2)$ 
for the $C^r$ functions $U_1\to\RR^2$ at $P\in U_1$, etc.
Multiplication and composition of jets is defined by corresponding
polynomial operations modulo the monomials in $\tilde{u},\tilde{v}$ of degree $r+1$.
The relevant jets are denoted by $J^r_P(\psi_0)$, $J^r_P(f_1)$, $J^r_Q(f_2)$ for
$P\in\tau_1$, $Q\in\tau_2$. Morevoer, the jet $J^r_P(g)$ can be defined for a function 
$g\in C^0(\sgma_1)$ at a boundary point $P\in\sgma_1$ if $g$ is extendible to
a $C^r$ function in a neighborhood of $P$, because all Taylor coefficients are determined
by the behavior of $g$ on $\sgma_1$.
Now we can give a general definition of $G^r$ functions,
more computable than (\ref{eq:crcont}). 
\begin{definition} \label{eq:grfdef}
Suppose we have functions $g_1\in C^1(\sgma_1)$,  $g_2\in C^1(\sgma_2)$. 
The pair $(g_1,g_2)$ 
is called a {\em geometrically continuous $G^r$ function} on $\cM_0$ if
$g_1|_{\tau_1}=g_2\circ \lambda$ and 
\begin{equation} \label{eq:grjet}
J^r_P(g_1)=J^r_Q(g_2)\circ J^r_P(\psi_0)
\end{equation}
for all points $P\in\tau_1$ and $Q=\lambda(P)$.
We denote the linear space of $G^r$ functions by $G^r(\cM_0,\psi_0)$.
The space is independent from the term $\psi_r$ in (\ref{eq:thetat}).
\end{definition}
To define the $G^r$ functions, 
it is thus enough to specify continuously varying $J^r_P(\psi_0)$
for each $P\in\tau_1$, 
following (\ref{eq:thetat}) and the continuity restrictions on all $\theta_k(u_1)$ 
but ignoring $\psi_r(u_1,v_1)$. The technical term for this continuously 
varying family of jets is the {\em jet bundle} $J^r_{\tau_1}(\psi_0)$.
Since CAGD typically uses $C^\infty$ restrictions $g_1,g_2$ 
(as remarked in Definition \ref{def:c1cl}), 
it is natural to define the jet bundle $J^r_{\tau_1}(\psi_0)$ by equation (\ref{eq:thetat}) 
with all $\theta_k(u_1)$ in the same  or a similar class of $C^\infty$ functions. 
For example, if we want to define {\em polynomial $G^r$ splines},
we may choose all $\theta_k(u_1)$ to be continuous rational functions of $u_1$.
We can always take $\psi_r(u_1,v_1)=0$.
Condition \refpart{A2} leads to additional constrains on $J^r_P(\psi_0)$;
see Proposition \ref{def:g1vf} soon. 

Dualization leads to a coordinate-independent (and at the same time, 
more computable) formulation of the geometric continuity condition. 
For that, we concentrate on the linear transformation of the Taylor coefficients of $g_1,g_2$
rather than on the transformation (\ref{eq:grjet}) of whole jets.
Up to a constant multiple, a Taylor coefficient is a differentiation functional
$\partial^{j+k}/\partial u_i^j\partial v_i^k$ (of order $j+k\le r$, with $i\in\{1,2\}$). 
Applying these differentiations with $i=1$ to (\ref{eq:crcont}) 
we get expressions of $\partial^{j+k}/\partial u_1^j\partial v_1^k$ (of $f_1$)
in terms of differentiations in $u_2,v_2$ (of $f_2$) of order \mbox{$\le j+k$}.
The same transformation of differential operators of order $\le r$ is induced by (\ref{eq:grjet}).
Rather than transforming the space of functions, we seek to use the
corresponding transformation of differentiation operators.
We work out this correspondence for $r=1$.

The {\em tangent space} $T_P$ of $\RR^2$ at a point $P\in\RR^2$ can be defined \cite{Wiki} 
as the linear space of {\em derivations}, or {\em directional derivatives}. 
For a vector $\,\overrightarrow{\!AB}$, the direction derivative $\partial_{AB}$ is defined as
\begin{equation} \label{eq:dirder}
\partial_{AB} f(P) = \lim_{t\to 0} \frac{f(P+t\;\overrightarrow{\!AB})-f(P)}{t}
\end{equation}
If $A,B$ are the endpoints $u_1=0$, $u_1=1$ of $\tau_1$, then $\partial/\partial u_1=\partial_{AB}$.
We have then \mbox{$\partial/\partial v_1=\partial_{AC}$}
where $C$ is the point $u_1=0$, $v_1=1$.

With $r=1$, the $G^1$-structure defining jet bundle $J^1_{\tau_1}(\psi_0)$ can be written as
\begin{equation} \label{eq:jetb1}
{u_1\choose v_1} \mapsto  {u_2\choose v_2}={u_1+\ma(u_1)v_1\choose \mb(u_1) v_1}
\end{equation}
for some $\ma, \mb\in C^0(\tau_1)$. 
This expression can be taken as a representative transition map
in the equivalence class $J^1_{\tau_1}(\psi_0)$ of transition maps. 
If $(g_1(u_1,v_1)$, $g_2(u_2,v_2))$ is a $G^1$ function on $\cM_0$, then we have
$\partial g_1/\partial u_1(u_1,0)\!=\partial g_2/\partial u_2(u_1,0)$ for all $u_1\in[0,1]$
as a consequence of \mbox{$g_1(u_1,0)=g_2(u_1,0)$}, and
\begin{equation} \label{eq:g1func0}
\frac{\partial g_1}{\partial v_1}(u_1,0) = 
\ma(u_1)\frac{\partial g_2}{\partial u_2}(u_1,0) 
+\mb(u_1)\frac{\partial g_2}{\partial v_2}(u_1,0).
\end{equation}
Defining a space of $G^1$ functions on $\cM_0$ is thereby equivalent to giving a continuous 
family $\Theta$ 
of linear isomorphisms of the tangent spaces $T_P$ and $T_{\lambda(P)}$ for all $P\in\tau_1$,
defined by $\partial/\partial u_1\mapsto\partial/\partial u_2$ and
\begin{equation} \label{eq:tbiso}
\frac{\partial}{\partial v_1} \mapsto \ma(u_1)\frac{\partial}{\partial u_2} 
+\mb(u_1)\frac{\partial}{\partial v_2}
\end{equation}
with specified functions $\ma(u_1), \mb(u_1)$. 
Each tangent space is isomorphic to $\RR^2$ as a linear space.
The vectors are directly identified as directional derivatives by (\ref{eq:dirder}).
We say that a non-zero derivative is {\em parallel} to an edge $\tau'$ 
if its differentiation direction  is parallel $\tau'$. Otherwise the derivative
is called {\em transversal} \cite{Wiki} to $\tau'$. 

The tangent spaces form {\em trivial tangent bundles} \cite{Wiki}
$\tau_1\times\RR^2$, $\tau_2\times\RR^2$,
and $\Theta$ is an isomorphism of the tangent bundles 
(compatible with \mbox{$\lambda:\tau_1\to\tau_2$}).
The isomorphism $\Theta$ can be specified without a coordinate system,
using only derivative bases and the linear function $u_1|_{\tau_1}$.
Since the derivatives along $\tau_1\sim\tau_2$ are identified,
it is enough to identify pairs of transversal derivatives at each $P\in\tau_1$
and $Q=\lambda(P)$. This leads to the characterization of geometric continuity
in terms of {\em transversal vector fields} \cite{Wiki} 
along the edges, as in \cite[Corollary 3.3]{Hahn87}.
Condition \refpart{A2} is then easier to reflect as well 
(with reference to Definition \ref{def:inout}),
leading to the requirement $\mb(u_1)<0$ on $\tau_1$.
A thorough discussion of condition \refpart{A2} and 
topological degenerations to avoid in $G^1$ gluing 
is presented in \S \ref{sec:topology}.

The following statement is the special case $r=1$ of \cite[Lemma 3.2]{Hahn87}.
The vector field $D_1$ can be chosen to be constant, like $\partial/\partial v_1$ in (\ref{eq:tbiso}).
\begin{proposition} \label{def:g1vf} 
Let $D_1$ denote a transversal $C^0$ vector field on $\tau_1$ 
pointing to the inward side (towards $\Omega_1$), 
and let $D_2$ denote a transversal $C^0$ vector field on $\tau_2$ 
pointing  to the outward side (away from $\Omega_2$).
Then the space of functions $(g_1,g_2)$ satisfying 
\begin{equation} \label{eq:vfdat}
g_1(P)=g_2(Q), \qquad D_1g_1(P)=D_2g_2(Q)
\end{equation}
for any $P\in\tau_1$, $Q=\lambda(P)$ is the space $G^1(\cM_0,\varphi_0)$ 
for some diffeomorphism $\varphi_0$ between open neighborhoods 
of $\tau_1,\tau_2$.
\end{proposition}
\begin{proof}
The vectors in $D_1$, $D_2$ are viewed as directional derivatives. 
The data in (\ref{eq:vfdat}) defines isomorphisms $T_{P}\to T_{Q}$
of tangent spaces, continuously varying with $P$ and compatible 
with $\lambda:\tau_1\to\tau_2$. After attaching coordinate systems 
$(u_1,v_1)$, $(u_2,v_2)$ to $\tau_1$, $\tau_2$, the tangent bundle 
isomorphism can be characterized by transformation (\ref{eq:tbiso}) 
with some continuous functions $\ma(u_1)$, $\mb(u_1)$. 
We define $\varphi_0$ by (\ref{eq:jetb1}).
\end{proof}
\begin{example} \label{ex:joining} \rm
Suppose that the two polygons of $\cM_0$ 
are triangles \mbox{$\sgma_1=A_1B_1C_1$}, $\sgma_2=A_2B_2C_2$,
that \mbox{$\tau_1=A_1B_1$}, $\tau_2=A_2B_2$, and that $\lambda$ maps $A_1$ to $A_2$. 
We attach {\em standard} coordinate systems $(u_1,v_1)$, $(u_2,v_2)$
to $\tau_1$, $\tau_2$, such that $u_1=0$, $u_2=0$ define the edges 
$A_1C_1$, $A_2C_2$, respectively. Consider a tangent bundle isomorphism $\Theta$ with 
\begin{equation} \label{eq:exdev}
\ma(u_1)=2u_1, \qquad \mb(u_1)=-1
\end{equation}
in (\ref{eq:tbiso}). 
Since $u_1,\partial/\partial u_1$ are synonymous to $u_2,\partial/\partial u_2$
for $G^1$ functions, identification (\ref{eq:tbiso}) can be rewritten as
\begin{equation}
\left( \frac{\partial}{\partial v_1}-u_1\,\frac{\partial}{\partial u_1} \right) \mapsto  
-\left( \frac{\partial}{\partial v_2}-u_2\,\frac{\partial}{\partial u_2} \right).
\end{equation}
In other words, for any $P\in A_1B_1$ the isomorphism $T_P\to T_{\lambda(P)}$ is defined by
\begin{equation}
\partial_{PC_1} \mapsto -\,\partial_{QC_2}
\end{equation}
(together with the trivial $\partial_{A_1B_1} \mapsto \partial_{A_2B_2}$).
The vector fields $\partial_{PC_1}$, $\partial_{QC_2}$ pointing ``exactly" 
towards the third vertices $C_1$, $C_2$ are identified with the minus sign.
This is a nice characterization of a tangent bundle isomorphism. 
In particular, the two end-points of $\tau_1\sim\tau_2$ are symmetric.
If the coordinate systems $(u_1,v_1)$, $(u_2,v_2)$ are changed 
to the standard alternatives with $u_1=0$, $u_2=0$ 
on the edges $B_1C_1$, $B_2C_2$ (respectively), 
the coordinate change is
\begin{equation} \label{eq:trcoor}
(u_k,v_k)\mapsto (1-u_k-v_k,v_k) 
\end{equation}
for $k \in\{1,2\}$, and the derivative changes are
\begin{equation} \label{eq:reders}
\left( \frac{\partial}{\partial u_k},\frac{\partial}{\partial v_k} \right)
\mapsto \left( -\frac{\partial}{\partial u_k}, \,
\frac{\partial}{\partial v_k}-\frac{\partial}{\partial u_k}\right).
\end{equation}
The $\Theta$-defining functions $\ma(u_1),\mb(u_1)$ 
have the same expressions 
(\ref{eq:exdev}) in the alternative coordinates and derivative bases.
\end{example}
\begin{definition} \label{def:join}
Let $P_1\in \sgma_1$ 
denote an endpoint of $\tau_1$, and let $P_2=\lambda(P_1)\in\tau_2$.
Let $\widehat{\tau}_1\subset \sgma_1$ , $\widehat{\tau}_2\subset \sgma_2$ denote
the other polygonal edges incident to $P_1$, $P_2$, respectively. 
The edge $\tau_1\sim\tau_2$ of $\cM_0$ is called a {\em joining edge} 
at the vertex $\{P_1, P_2\}$ if the tangent space isomorphism
$T_{P_1}\to T_{P_2}$ maps a derivative parallel to $\widehat{\tau}_1$
to a derivative parallel to $\widehat{\tau}_2$.
If we choose the standard coordinates $(u_1,v_1)$, $(u_2,v_2)$ 
with $u_1=0$, $u_2=0$ at $P_1,P_2$ (respectively),
the joining edge is characterized by $\ma(P_1)=0$ in (\ref{eq:tbiso}).
We call the edges $\widehat{\tau}_1$, $\widehat{\tau}_2$ {\em opposite to each other} at $\{P_1,P_2\}$.
They are forced to transverse into each other 
across the common vertex in smooth realizations of $\cM_0$.
In Example \ref{ex:joining}, the gluing edge is a joining edge at both of its endpoints.
\end{definition}

\subsection{Glueing around a vertex}
\label{sec:g1vertex}

Another basic situation to consider is $G^r$ gluing of several polygons 
$\sgma_1,\ldots,\sgma_n$ around a common vertex.
Let $P_1,\ldots,P_n$ denote their vertices, respectively, 
to be glued to the common vertex. For each polygon $\sgma_k$,
let $\tau_k,\tau'_k$ denote the two edges incident to $P_k$.
We denote $\ii=\{1,\ldots,n\}$, and 
label cyclically $P_{0}=P_n$,  $\tau'_{0}=\tau'_n$.
Assume that for each $k\in\ii$ the edge $\tau_k$ is glued to $\tau'_{k-1}$
so that the vertices $P_k$, $P_{k-1}$ are thereby identified. 
More explicitly, we assume linear homeomorphisms $\lambda_k:\tau_k\to\tau'_{k-1}$
such that $\lambda_k(P_k)=P_{k-1}$, and $C^r$ transition maps $\varphi_k$
between open neighborhoods of $\tau_k$ and $\tau'_{k-1}$
specializing to $\lambda_k$. Once local coordinates around each $P_k$ are fixed,
we may replace the transition maps $\varphi_k$ by their  jet bundles $J^r_{\tau_k}(\varphi_k)$
as a more concise required data.

Let $\cM^*_0$ denote the $G^0$ polygonal surface defined by the polygons $\Omega_k$ 
and the homeomorphisms $\lambda_k$ with $k\in\ii$.
Let $\cP_0$ denote the common vertex $\{P_1,\ldots,P_n\}$.
\begin{definition} \label{def:g1v}
Suppose that for each $k\in\ii$ we have a function $g_k\in C^0(\sgma_k)$  
extendible to a $C^r$ function 
on an open neighborhoods of $\sgma_k$. Set $g_0=g_n$.
The sequence $(g_1,\ldots,g_n)$ is called a {\em geometrically continuous $G^r$ function} 
on $\cM^*_0$ if for each $k\in\ii$ the pair $(g_k,g_{k-1})$ is a $G^r$ function
on the polygonal surface $(\{\sgma_k,\sgma_{k-1}\},\{\lambda_k\})$ as in \S \ref{sec:edgeglue},
glued by $\varphi_k$ as well.
In particular, (\ref{eq:grjet}) becomes 
\begin{equation} \label{eq:grvjet}
J^r_{P_k}(g_k)=J^r_{P_{k-1}}(g_{k-1})\circ J^r_{P_k}(\varphi_k).
\end{equation}
\end{definition}

A natural restriction on the gluing data is
\begin{equation} \label{eq:grcjet}
J^r_{P_n}(\varphi_1\circ\ldots\circ\varphi_n)(\tilde{u},\tilde{v})=(\tilde{u},\tilde{v})
\end{equation}
in some (or any) local coordinates around $P_n$. 
This is a necessary condition for existence of satisfactorily many $G^r$ functions,
because the equations (\ref{eq:grvjet}) imply 
\begin{equation} \label{eq:jetcomp}
J^r_{P_n}(g_n)=J^r_{P_n}(g_n)\circ J^r_{P_n}(\varphi_1\circ\ldots\circ\varphi_n).
\end{equation}
This is a restrictive functional equation on $g_n$ unless (\ref{eq:grcjet}) is satisfied.
Remark \ref{rm:fulltg} below clarifies more explicitly.

For each $k\in \ii$, let us choose the standard coordinate system $(u_k,v_k)$ 
attached to $\tau_k$ with $u_k=0$ at $P_k$. 
Then $(v_k,u_k)$ is a standard coordinate system attached to $\widehat{\tau}_k$. 
With these coordinates, it is straightforward to rewrite  (\ref{eq:grcjet})
as a composition of jets:
\begin{equation} \label{eq:grcjett}
J^r_{P_1}(\varphi_1)\circ\ldots\circ J^r_{P_n}(\varphi_n)(u_n,v_n)=(u_n,v_n).
\end{equation}
This form has two advantages. It involves composition of computationally more definite objects,
that is, jets at single points rather than possibly transcendental transition maps. 
Secondly, it allows to define the gluing data in terms of jet bundles rather than
transition maps.

For $r=1$, each jet bundle $J^1_{\tau_k}(\varphi_k)$ with $k\in \ii$
is determined by continuous functions $\ma_k(u_k)$, $\mb_k(u_k)$,
so that it acts like in (\ref{eq:jetb1}):
\begin{equation} \label{eq:jetb2}
{u_k\choose v_k} \mapsto  {v_{k-1}\choose u_{k-1}}
={u_k+\ma_k(u_k)v_k\choose \mb_k(u_k) v_k}.
\end{equation}
Specialization to $P_k$ gives this transformation $J^1_{P_k}(\varphi_k)$
of $J^1$-jets in the standard coordinates:
\begin{equation} \label{eq:jetbp}
\left( \begin{array}{c} 1 \\ u_{k-1} \\ v_{k-1} \end{array} \right)
=\left( \begin{array}{ccc} 1 & 0 & 0 \\
0 & 0 & \mb_k(0) \\ 0 & 1 & \ma_k(0) \end{array} \right) 
\left( \begin{array}{c} 1 \\ u_k \\ v_k \end{array} \right).
\end{equation}
Ignoring the trivial transformation of the constant terms,
condition (\ref{eq:grcjett}) becomes
\begin{equation} \label{eq:jetbp1}
\left( \begin{array}{cc} 0 & \mb_1(0) \\ 1 & \ma_1(0) \end{array} \right) \,\cdots\,
\left( \begin{array}{cc} 0 & \mb_n(0) \\ 1 & \ma_n(0) \end{array} \right) =
\left( \begin{array}{cc} 1 & 0 \\  0 & 1 \end{array} \right).
\end{equation}
Note that the matrix product is non-commutative, 
hence we do not use the $\prod$-notation. 
The dual transformations of partial derivatives are similar:
\begin{equation} \label{eq:tgp}
{\partial/\partial u_{k}\choose \partial/\partial v_{k}}
=\left( \begin{array}{cc} 0 & 1 \\ \mb_k(0) & \ma_k(0) \end{array} \right) 
{\partial/\partial u_{k-1} \choose \partial/\partial v_{k-1}}.
\end{equation}
This is an expression of the induced isomorphism $T_{P_n}\to T_{P_{n-1}}$
of the tangent spaces.  
The transposed version of (\ref{eq:jetbp1})
\begin{equation} \label{eq:tsbp1}
\left( \begin{array}{cc} 0 & 1 \\ \mb_n(0) & \ma_n(0) \end{array} \right) \,\cdots\,
\left( \begin{array}{cc} 0 & 1 \\ \mb_1(0)  & \ma_1(0) \end{array} \right) =
\left( \begin{array}{cc} 1 & 0 \\  0 & 1 \end{array} \right)
\end{equation}
expresses the dual to (\ref{eq:grcjett}) requirement that the composition
\mbox{$T_{P_n} \!\!\to \ldots\to T_{P_1} \!\to T_{P_{n}}$} 
must be the identity map. (See the same Remark \ref{rm:fulltg}.) 
The partial matrix products
\begin{equation} \label{eq:jetbpk}
\left( \begin{array}{cc} 0 & 1 \\ \mb_k(0) & \ma_k(0) \end{array} \right) \,\cdots\,
\left( \begin{array}{cc} 0 & 1 \\ \mb_1(0)  & \ma_1(0) \end{array} \right)
\end{equation}
with $k\in \ii\setminus \{n\}$ express the derivatives $\partial/\partial u_{k}$, $\partial/\partial v_{k}$ 
(of the component $g_k$ of presumed $G^1$ functions on $\cM_0^*$) 
in terms of $\partial/\partial u_{n}$, $\partial/\partial v_{n}$ (of the component $g_n$).
This allows us to represent the involved partial derivatives graphically 
as directional derivatives in the tangent space $T_{P_n}$; see Figure \ref{fig:g1v}\refpart{a}.  
The derivatives $\partial/\partial u_{k}$ and $\partial/\partial v_{k-1}$ 
are automatically identified, and the convex sectors $\widehat{\sgma}_k$ 
bounded by consecutive vectors $\vb_{k-1}$, $\vb_k$ represent derivatives in the polygonal corners $P_k$. 
It is not automatic that the fan \cite{Wiki} 
of tangent sectors will cover $\RR^2$ exactly once, as we discuss in \S \ref{sec:topology}.
This schematic picture can visually reveal features and quality of the gluing data 
for CAGD. As explained in \S \ref{rm:gconstr}, we can start constructing 
good gluing data on a settled topological manifold $\cM$ 
by choosing a sector partition as in Figure \ref{fig:g1v}\refpart{a} around each vertex,
drawing directional derivatives on the sector boundaries, and copying relations 
between derivative vectors to the values $\ma_k(0)$, $\mb_k(0)$
to be interpolated to a global $G^1$ gluing structure.


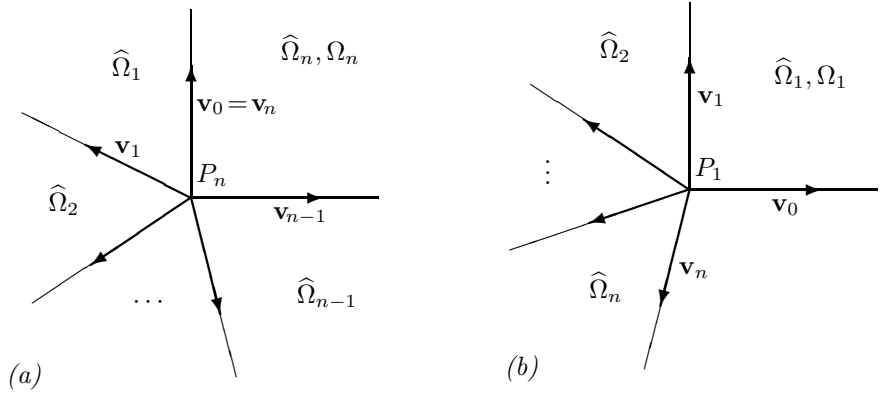
\begin{figure} \[ 
\begin{picture}(160,144)(-72,-72)
\put(0,0){\line(1,0){71}}  \put(0,0){\line(0,1){71}} 
\put(0,0){\line(-2,1){64}} \put(0,0){\line(-3,-2){60}}
\put(0,0){\line(1,-4){17}} 
\thicklines \put(0,0){\vector(1,0){50}}  \put(0,0){\vector(0,1){50}} 
\put(0,0){\vector(-2,1){40}} \put(0,0){\vector(-3,-2){38}}
\put(0,0){\vector(1,-4){11}}  \put(2,6){$P_n$}  
\put(-30,46){$\widehat{\sgma}_1$} 
\put(34,52){$\widehat{\sgma}_n,\sgma_n$}  
\put(40,-40){$\widehat{\sgma}_{n-1}$}     \put(-54,-5){$\widehat{\sgma}_2$}  
\put(-22,-42){$\cdots$}  \put(-70,-70){\refpart{a}}
\put(2,33){$\vb_0\!=\!\vb_{\!n}$} \put(-29,17){$\vb_1$} \put(31,-8){$\vb_{\!n-1}$} 
\end{picture} \qquad\;\;
\begin{picture}(160,150)(-75,-75)
\put(0,0){\line(1,0){71}}  \put(0,0){\line(0,1){71}} 
\put(0,0){\line(-3,2){60}} \put(0,0){\line(-3,-1){68}}
\put(0,0){\line(-1,-4){17}} 
\thicklines \put(0,0){\vector(1,0){50}}  \put(0,0){\vector(0,1){50}} 
\put(0,0){\vector(-3,2){40}} \put(0,0){\vector(-3,-1){38}}
\put(0,0){\vector(-1,-4){11}}  \put(2,5){$P_1$}  
\put(-34,50){$\widehat{\sgma}_2$}  
\put(32,40){$\widehat{\sgma}_1,\sgma_1$}   
\put(-38,-41){$\widehat{\sgma}_n$} 
 \put(-55,2){$\vdots$}  \put(-70,-70){\refpart{b}}
\put(3,33){$\vb_1$} \put(-4,-32){$\vb_n$} \put(31,-8){$\vb_0$} 
\end{picture} \]
\caption{Tangent sectors around interior and boundary vertices}  
\label{fig:g1v}
\end{figure}

\begin{example} \rm \label{ex:hahnsym}
A simple way to choose the relations between derivatives along polygonal edges
at the vertex $\cP_0=\{P_1,\ldots,P_n\}$ is to choose $n$ vectors of the same length,
and evenly space them at the angle $2\pi/n$. The gluing data then has
\begin{equation} \label{eq:hsym}
\ma_k(0) = 2\cos{\small \frac{2\pi}n}, \qquad \mb_k(0)=-1
\end{equation}
for all $k\in N$. 
In the context of gluing rectangles, Hahn \cite[\S 8]{Hahn87} proposed to choose
this symmetric $G^1$ (or even $G^r$) gluing structure at all vertices of 
a linear polygonal surface $\cM\supset\cM_0^*$,  and straightforwardly interpolate 
the restricted values $\ma_k(0)$. In particular, if our $\tau_k\sim\tau'_{k-1}$ 
connects $\cP_0$ 
to an interior vertex on $\cM$ of valency $m$,
Hahn's proposal is to  take $\mb_k(u_1)=-1$ and let 
$\ma_k(u_1)\in C^r(\tau_1)$ interpolate
\begin{equation}
\ma_k(0)= 2\cos{\small \frac{2\pi}n}, \qquad
\ma_k(1)= -2\cos{\small \frac{2\pi}m},
\end{equation}
with all derivatives of order $\le r$ at both end-points set to zero. 
The minus sign in $\ma_k(1)$ appears because the transformation
of standard coordinates and derivatives is
\begin{equation} \label{eq:rectuv}
(u_k,v_k)\mapsto (1-u_k,v_k), \qquad
\left( \frac{\partial}{\partial u_k},\frac{\partial}{\partial v_k} \right)
\mapsto \left( -\frac{\partial}{\partial u_k}, \,
\frac{\partial}{\partial v_k}\right).
\end{equation}
For comparison, if $\sgma_k$ is a triangle (and $\sgma_{k-1}$ is a rectangle)
then the change of standard coordinates and derivatives is 
(\ref{eq:trcoor})--(\ref{eq:reders}). Then we have to interpolate to 
\begin{equation}
\ma_k(1)= 1-2\cos{\small \frac{2\pi}m}.
\end{equation}
If $\sgma_{k-1}$ is a triangle as well,  then we have to transform
$(v_{k-1},u_{k-1})$ by (\ref{eq:trcoor})--(\ref{eq:reders}) also.
The interpolation is then to 
\begin{equation}
\ma_k(1)= 2-2\cos{\small \frac{2\pi}m},
\end{equation}
like in Example \ref{ex:joining}.
\end{example}
\begin{definition}
In the context of $G^1$ gluing, linear interpolation of the requisite values
$\ma_k(0)$, $\ma_k(1)$ of the symmetric gluing around interior vertices 
(of a linear polygonal surface $\cM$) is often used in CAGD. 
Let us refer to this kind of gluing data as  {\em linear Hahn gluing}.
\end{definition}
\begin{definition} \label{def:crossing}
Suppose that $n=4$ polygons are being glued around the common vertex
$\cP_0\in\cM_0^*$. 
The vertex is called {\em a crossing vertex} if all four edges $\tau_1\sim\tau'_4$,
$\tau_2\sim\tau'_1$, $\tau_3\sim\tau'_2$, $\tau_4\sim\tau'_3$
are joining edges (respectively) at $\{P_1,P_4\}$, $\{P_1,P_2\}$, $\{P_2,P_3\}$, $\{P_3,P_4\}$,
as in Definition \ref{def:join}.
Since we use standard coordinate systems, the characterizing condition
is $\ma_1(0)=\ma_2(0)=\ma_3(0)=\ma_4(0)=0$.
Then condition (\ref{eq:tsbp1}) leads to
\begin{equation} \label{eq:qbetas}
\mb_1(0)\mb_3(0)=1,\qquad \mb_2(0)\mb_4(0)=1.
\end{equation}
The special case $\mb_1(0)=\mb_2(0)=\mb_3(0)=\mb_4(0)=-1$ 
gives Hahn's symmetric gluing (\ref{eq:hsym}).
\end{definition}
\begin{remark} \label{rm:fulltg} \rm
If conditions (\ref{eq:jetcomp}), (\ref{eq:tsbp1}) are not satisfied,
there will be directional derivatives 
in $T_{P_{n}}$ that always evaluate the component $g_n$ of any $G^1$ 
function of Definition \ref{def:g1v} to zero. 
This is clearly the case for the directions corresponding to
the eigenvalues $\neq 1$ of the matrix product in (\ref{eq:jetbp1}).
If the matrix product 
is $1\ 1\choose 0\ 1$,
then the eigenvector direction will give zero derivatives.
Then we do not have a "full" tangent space, that is, enough functionals on the Taylor coefficients.
There should be two linearly independent derivatives at any vertex 
of our  polygonal surfaces 
acting non-trivially on {considered} spaces of $G^1$ functions.
The next subsection formalizes an additional condition caused by this guiding principle,
generalizing a relevant obstruction in \cite{Peters2010}.
\end{remark}

\subsection{Restriction on crossing vertices}
\label{sec:crossing}

It may seem that linear Hahn gluing always gives spline spaces $G^r(\cM)$
that allow smooth realizations of the surface pieces such as $\cM^*_0\subset\cM$ 
around their common vertices. Rather unexpectedly,
Peters and Fan \cite{Peters2010} noticed that even if
the gluing data satisfies (\ref{eq:jetbp1}) there are feasibility issues 
of smooth realization around crossing vertices.
Their context was linear Hahn gluing of rectangles.
The noticed necessary restriction appears to be topological,
but we generalize it to a local obstruction.
\begin{definition} \label{def:balance}
Let $\cP=\{P_1,P_2,P_3,P_4\}$ denote a crossing vertex of a polygonal surface $\cM$.
A pair of opposite edges $\tau^*_1$, $\tau^*_2$ at $\cP$ is called {\em balanced} if
(at least) one of the following conditions holds:
\begin{itemize}
\item the other endpoints of $\tau^*_1$, $\tau^*_2$ are vertices in $\cM$ of the same order;
\item the other endpoint of $\tau^*_1$ or $\tau^*_2$ is on the boundary of $\cM$.
\end{itemize}
The crossing vertex $\cP$ is called {\em balanced} if both pairs of opposite edges are
balanced.
\end{definition}
The observation in \cite{Peters2010} is this:
if $\cM$ is a linear polygonal surface of rectangles, and the linear Hahn gluing is used, 
then a smooth realization of $\cM$ by polynomial splines exits only if all its crossing vertices
are balanced. It turns out that the obstruction is essentially local after all. It arrises 
because of differentiability of the gluing data and assumed polynomial restrictions.
Here is a generalization of the Peters--Fan obstruction, 
in the local setting of \S \ref{sec:g1vertex}.
\begin{theorem}\label{cond:comp2}
Let $\cM^*_0$ denote a polygonal surface as in \S $\ref{sec:g1vertex}$ with $n=4$.
Suppose that the interior vertex $\cP_0$ is a crossing vertex,
and that the gluing data $(\ma_k,\mb_k)$ for $k\in N=\{1,2,3,4\}$ 
are differentiable functions on the edges $\tau_k$. 
Let $G^1_2(\cM^*_0)$ denote the linear space of $G^1$ functions 
$(g_1,g_2,g_3,g_4)$ on $\cM_0^*$ such that $g_k\in C^2(\sgma_k)$
for $k\in N$. 
A smooth realization of $\cM^*_0$ by functions in $G^1_2(\cM^*_0)$ exists if and only if
\begin{eqnarray} \label{eq:ddv1}
\ma'_1(0)+\frac{\mb'_2(0)}{\mb_2(0)} &=& -\mb_1(0) 
\left(\ma'_3(0)+\frac{\mb'_4(0)}{\mb_4(0)}\right),\\ 
\label{eq:ddv2}
\ma'_2(0)+\frac{\mb'_3(0)}{\mb_3(0)} &=& -\mb_2(0) 
\left(\ma'_4(0)+\frac{\mb'_1(0)}{\mb_1(0)}\right). 
\end{eqnarray}
\end{theorem}
\begin{proof}
We differentiate (\ref{eq:tbiso}) with respect to $u_1$ and adjust the index notation,
to get this transformation of the mixed derivative for the presumed $g_k$:
\[
\frac{\partial^2}{\partial u_k\partial v_k} \mapsto  \ma'(u_k)\frac{\partial}{\partial u_{k-1}} 
+\ma(u_k)\frac{\partial^2}{\partial u_{k-1}^2} +\mb'(u_k)\frac{\partial}{\partial v_{k-1}}
+\mb(u_k)\frac{\partial^2}{\partial u_{k-1}\partial v_{k-1}}.
\]
Transformation (\ref{eq:tgp}) of derivatives at $u_1=0$ is extended to 
\begin{equation} \label{eq:trE3}
\left( \begin{array}{c} \partial/\partial u_{k}\\ \partial/\partial v_{k} \\ 
\!\partial^2/\partial u_{k}\partial v_{k}\!\! \end{array} \right)
=\left( \begin{array}{ccc} 0 & 1 & 0 \\ \mb_k(0) & 0 & 0 \\
\mb'_k(0) & \ma'_k(0) & \mb_k(0) \end{array} \right) 
\left( \begin{array}{c} \partial/\partial u_{k-1}\\ \partial/\partial v_{k-1} \\ 
\!\partial^2\!/\partial u_{k-1}\partial v_{k-1}\!\!\! \end{array} \right) \!.
\end{equation}
Following (\ref{eq:tsbp1}) and keeping in mind (\ref{eq:qbetas}),
we compute the matrix product 
\[
\left( \begin{array}{ccc} 0 & 1 & 0 \\ \mb_4(0) & 0 & 0 \\
\mb'_4(0) & \ma'_4(0) & \mb_4(0) \end{array} \right) \cdots
\left( \begin{array}{ccc} 0 & 1 & 0 \\ \mb_1(0) & 0 & 0 \\
\mb'_1(0) & \ma'_1(0) & \mb_1(0) \end{array} \right)
=\left( \begin{array}{ccc} 1 & 0 & 0 \\ 0 & 1 & 0 \\
H_1 & H_2 & 1 \end{array} \right) \,
\]
with
\begin{align*}
H_1 & = \ma_4'(0)+\mb_3(0)\mb'_1(0)+\mb_4(0)\big(\ma_2'(0)+\mb_1(0)\mb'_3(0)\big),\\
H_2 & = \ma_3'(0)+\mb_2(0)\mb'_4(0)+\mb_3(0)\big(\ma_1'(0)+\mb_4(0)\mb'_2(0)\big).
\end{align*}
Since the derivative $\partial^2g_n/\partial u_{n}\partial v_{n}$ is well defined,
we have 
\begin{equation} \label{eq:singproof}
\left( H_1\frac{\partial}{\partial u_{n}}+H_2\frac{\partial}{\partial u_{n}}\right) g_n=0.
\end{equation}
If the differential operator is not zero, it gives a directional derivative
(in terms of Figure \ref{fig:g1v} \refpart{a}, towards the interior of one of the polygons)
that annihilates any function in $G^1_2(\cM^*_0)$. Any realization of $\cM^*_0$
would be singular at $\cP_0$ in that direction. Therefore we must have
$H_1=H_2=0$, leading to (\ref{eq:ddv1})--(\ref{eq:ddv2}) after applying (\ref{eq:qbetas}).
\end{proof}
In the Peters-Fan setting \cite{Peters2010}, we have all $\mb_k(0)=-1$, $\mb_k'(0)=0$.
The restrictions (\ref{eq:ddv1})--(\ref{eq:ddv2}) are then simply 
\begin{equation} \label{eq:quadrsym}
\ma'_1(0)=\ma'_3(0), \qquad \ma'_2(0)=\ma'_4(0).
\end{equation}
With the linear Hahn gluing of rectangles, the derivatives $\ma'_k(0)$ are determined
by the order of the neighboring vertices. The mentioned balancing condition follows.
\begin{example} \label{ex:pfcontra} \rm
This examples shows that a smooth realization may exist with the equations (\ref{eq:ddv1})--(\ref{eq:ddv2})
not fulfilled, if the $C^2$-differentiability condition is dropped. 
Let us assume that the four polygons of $\cM^*_0$ are the rectangles \mbox{$\Omega_1=[0,1]^2$},
$\Omega_2=[-1,0]\times [0,1]$, $\Omega_3=[-1,0]^2$, $\Omega_4=[0,1]\times [-1,0]$,  and set
\begin{align*}
\textstyle \ma_1(u_1)=\frac12 u_1, \hspace{25pt} 
& \quad \ma_2(u_2)=\ma_3(u_3)=\ma_4(u_4)=0,\\
\textstyle \mb_1(u_1)=-1+\frac12 u_1, 
& \quad \mb_2(u_2)=\mb_3(u_3)=\mb_4(u_4)=-1.
\end{align*}
Conditions (\ref{eq:ddv1})--(\ref{eq:ddv2}) are not satisfied. 
Using the global coordinates \mbox{$x=u_1=\!-v_2=\!-u_3=v_4$}, $y=v_1=u_2=-v_3=-u_4$,
consider these functions on $\cM^*_0$:
\begin{align*} \textstyle
G_1 = \big(x-\frac14 x^2,x,x,x-\frac14 h(x,y)\big),\qquad
G_2  = \big(y+\frac14 x^2,y,y,y+\frac14 h(x,y)\big),
\end{align*}
with $h(x,y)=(x+y)^2$ for $x+y\ge 0$ and $h(x,y)=0$ for $x+y\le 0$.
We have \mbox{$h(x,y)\in C^1(\Omega_4)$} and $G_1,G_2\in G^1(\cM^*_0)$.
The functions $(G_1,G_2,1)$ define a smooth embedding of $\cM^*_0$ into $\RR^3$. 
Lemma \ref{cond:comp2} does not apply because \mbox{$h(x,y)\not\in C^2(\Omega_4)$}.
On the level of polynomial specializations of $G_1,G_2$, 
the common vertex is not a crossing vertex 
as $\Omega_4$ is split along $u_4=v_4$ for them. 
\end{example}

\subsection{Topological restrictions}
\label{sec:topology}

Degenerations of geometrically continuous gluing reveal subtle
differences between the differential geometry setting and CAGD objectives.
Particularly condition \refpart{A2} is not reflected in the transformations 
(\ref{eq:tbiso}), (\ref{eq:tgp}) and constrains (\ref{eq:grcjett}), (\ref{eq:tsbp1}).
As we mentioned before Proposition \ref{def:g1vf}, condition \refpart{A2}
leads to the requirement that $\mb(u_1)$ in (\ref{eq:jetb1}), (\ref{eq:tbiso})
must be  a negative function on $\tau_1$. 
The particular condition $\mb(u_1)<0$ depends on our consistent use
of positive coordinates on the polygons. 
In loose intuitive terms, the coordinate function $v_1$ on $\sgma_1$ 
{\em should be} negative on $\sgma^0_2$, so it {\em should ``project"} to $v_2$ 
with a negative coefficient. Equivalently, the standard derivatives such as 
$\partial/\partial v_1$, $\partial/\partial v_2$ always give transversal vector fields 
pointing towards the interior sides, in contrast to Proposition \ref{def:g1vf}.

If one takes $\mb(u_1)>0$ then the implied transition map $\psi_0:U_1\to U_2$ 
identifies interior points in $\Omega_1\cup U_1$ and $\Omega_2\cup U_2$ of 
the two polygons of \S \ref{sec:edgeglue}. This is not desirable in CAGD, 
as we want the underlying topological surface to be as in Definition \ref{def:g0complex}. 
One can formally define the space $G^1(\cM_0)$ with $\mb(u_1)>0$
and use it to model surfaces with a sharp edge, so that ``smooth" realizations
in $\RR^3$ would have polygonal patches meeting at the angle $0$ rather
than the properly continuous angle $\pi$. But then $G^1(\cM_0)$ should not be
identified with the $C^1$ functions on $\cM_r$ of \S \ref{sec:edgeglue} (with $r=1$). 
For example, the first component of a $C^1$ function $(f_1,f_2)$ on $\cM_r$
then determines $f_2$ uniquely on $\Omega_2\cup U_2$.

Most strikingly, the algebraic conditions (\ref{eq:grcjett})--(\ref{eq:tsbp1}) 
of geometric continuity around vertices 
allow the tangent sectors in Figure \ref{fig:g1v}\refpart{a} to overlap
in a "winding up" fashion,  beyond the angle $2\pi$. 
If this happens, a 
realization of $\cM^*_0$ (of \S \ref{sec:g1vertex}) in $\RR^3$
by $G^1$ functions will have the interior vertex $\cP_0$ as a similarly branching point,
with the polygonal patches winding up and forming  
a self-intersecting surface in $\RR^3$ around the image of $\cP_0$.

This "winding up" degeneracy can even happen at a boundary vertex (of a large polygonal surface $\cM$),
where a sequence of polygons $\Omega_1,\ldots,\Omega_n$ is glued
around a common vertex $\{P_1,\ldots,P_n\}$ as in Figure \ref{fig:g1v}\refpart{b}, 
without  $\lambda_1:\tau_1\to\tau'_n$. All tangent spaces $T_{P_k}$
can be transformed to $T_{P_1}$ by following the partial matrix products
\begin{equation} \label{eq:jetbpk2}
\left( \begin{array}{cc} 0 & 1 \\ \mb_k(0) & \ma_k(0) \end{array} \right) \,\cdots\,
\left( \begin{array}{cc} 0 & 1 \\ \mb_2(0)  & \ma_2(0) \end{array} \right)
\qquad \mbox{with }  k\in \{2,\ldots,n\}
\end{equation}
like in (\ref{eq:jetbpk}), and draw a similar set of sectors in $T_{P_1}$.
Normally, the union of sectors should be a proper subset of $\RR^2$,
but the algebraic conditions do not disallow them to overlap beyond the angle $2\pi$.

For both interior and boundary vertices, it is enough to require that no other sector 
intersects an initial sector $\widehat{\Omega}_1\subset T_{P_1}$ or $\widehat{\Omega}_n\subset T_{P_n}$.
By an affine linear transformation, we can have $(1\ 0)$ and $(0\ 1)$ 
as the generating vectors of the initial sector. Then the initial sector coincides with the first quadrant.
The condition of not intersecting the first quadrant 
is formulated as follows.
\begin{lemma} \label{th:sectors}
Let $M$ denote a $2\times 2$ matrix with real entries.
Let $\sgma$ denote a convex sector in $\RR^2$ from the origin,  
generated by the two row vectors of $M$.
Then $\sgma$ intersects the first quadrant if and only if
there are two positive entries in a row of $M$, or on the main diagonal.
\end{lemma}
\begin{proof}
Let $\sgma_0$ denote the first quadrant. 
If $\sgma$ and $\sgma_0$ intersect, 
then either (at least) one of the boundaries of $\sgma$ lies inside $\sgma_0$, 
or $\sgma_0$ lies completely inside $\sgma_1$.
In the former case, we have a row of positive entries. 
In the latter case, the diagonal entries are positive by the convexity assumption. 
\end{proof}
Here is a formulation of restrictions to prevent the ``winding up" tangent spaces 
and locally self-intersecting realizations in $\RR^3$:
\begin{itemize}
\item[\refpart{B1}] For an interior vertex, we require 
that every partial matrix product (\ref{eq:jetbpk}) has 
a non-positive entry on each row and on the main diagonal.
\item[\refpart{B2}] For an boundary vertex of valency $n$, we require 
that every matrix product (\ref{eq:jetbpk2})
has a non-positive entry on each row and on the main diagonal,
and the bottom row of the product with $k=n$ should not be $(1\ 0)$.
\end{itemize}
In some applications, one my wish the stronger condition on boundary vertices
of the total angle of the tangents sectors to be less than $\pi$. 
Then the matrix products in (\ref{eq:jetbpk2}) should not have negative entries
in the second column.

As we see, the unstated CAGD-oriented requirement to fit the polygons tightly
leads to significant deviations from the differential geometry setting.
Even the common sense requirement \refpart{iii} 
of Definition \ref{def:g0complex} 
can be viewed in this light. 
This type of conditions are absent in differential geometry,
where all gluing is determined by the equivalence relation 
(of points on the coordinate open sets in $\RR^2$) defined by
the transition maps. Only with \refpart{B1}, \refpart{B2} 
satisfied we have the intended topological space $\cM$ of the corresponding differential surface.
Only then we can talk sensibly about ``full"  tangent spaces at the common vertices
such as $\cP$ in $\cM_0$, defined by the equivalence isomorphisms $T_P\to T_{\lambda(P)}$. 

The topological constraints 
could be dropped in some applications, for example, 
when modeling analytical surfaces with branching points, 
or surfaces with sharp wing-like ``interior" edges,
or with winding-up boundary. In these specific applications,
one should extend Definition \ref{def:crossing},
modify Theorem \ref{cond:comp2} 
so to allow winding up of $8, 12, \ldots$ polygons along joining edges,
or a combination (with any even number of polygons) 
of winding up and ``reflection" back from contra-\refpart{A2} sharp edges.
Apart from this consideration, 
the topological constrains do not 
affect our algebraic dimension count in \S \ref{sec:dformula}.

\section{Geometrically continuous splines}
\label{sec:gsplines}

We are ready to define a $G^1$ version of polygonal surfaces of Definition \ref{def:g0complex}.
The following data structure should be useful in most applications.
Allowing T-splines \cite{He:2006} elegantly would be an important adjustment.
\begin{definition} \label{def:g1complex}
Let $\ii_0,\ii_1$ denote finite sets.
A {\em (geometrically continuous) $G^1$ polygonal surface} is a collection
$(\{\sgma_k\}_{k\in \ii_0},\{(\lambda_k,\Theta_k)\}_{k\in \ii_1})$ such that
\begin{enumerate}
\item The pair $(\{\sgma_k\}_{k\in \ii_0},\{(\lambda_k)\}_{k\in \ii_1})$ 
is a linear $G^0$ 
polygonal surface of Definition \ref{def:g0complex}.
\item Each $\Theta_k$ 
is an isomorphism of tangent (or jet) bundles
on the polygonal edges glued by $\lambda_k:\tau_{k}\to{\tau}'_{k}$,
compatible with $\lambda_k$ and condition \refpart{A2}.
Concretely, $\Theta_k$ can be given by one of the following objects
(with adaptation of their local context in \S \ref{sec:geoglue}):
\begin{itemize}
\item a transition map satisfying \refpart{A1} and  \refpart{A2};
\item a jet bundle isomorphism as in (\ref{eq:jetb1}), with $\mb(u_1)<0$;
\item a tangent bundle isomorphism as in (\ref{eq:tbiso}), with $\mb(u_1)<0$;
\item an identification of transversal vector fields as in Proposition \ref{def:g1vf}.
\end{itemize}
\item At each interior vertex, 
the following restrictions must be satisfied:
\begin{itemize}
\item one of the equivalent conditions (\ref{eq:grcjet}), (\ref{eq:grcjett}), (\ref{eq:jetbp1}), (\ref{eq:tsbp1});
\item conditions (\ref{eq:ddv1})--(\ref{eq:ddv2});
\item condition \refpart{B1}.
\end{itemize}
\item At each boundary vertex, 
condition \refpart{B2} must be satisfied.
\end{enumerate}
\end{definition}
\begin{definition} 
A {\em $G^1$ function} on the $G^1$ polygonal surface $\cM$ of Definition \ref{def:g1complex}
is a collection $(g_\ell)_{\ell\in \ii_0}$ with \mbox{$g_\ell\in C^1(\sgma_\ell)$}, 
such that for any $k\in \ii_1$ and the defined gluing 
$(\lambda_k:\tau_1\to\tau_2,\Theta_k)$
of polygonal edges $\tau_1\subset\Omega_i$, $\tau_2\subset\Omega_j$ (with $i,j\in \ii_0$)
the functions $g_i,g_j$ 
are related as 
\begin{equation}
g_i|_{\tau_1}=g_j\circ \lambda_k, \qquad
(D_1g_i)|_{\tau_1} = (D_2g_j)  \circ \lambda_k.
\end{equation}
Here we assume the gluing data $\Theta_k$ is given in terms 
of transversal vector fields $D_1:\tau_1\to\RR^2$, $D_2:\tau_2\to\RR^2$ as in Proposition \ref{def:g1vf}.

Alternatively, a {\em $G^1$ function} on $\cM$
is a continuous function on the underlying $G^0$ polygonal surface,
such that the corresponding conditions (\ref{eq:g1func0}) hold on all interior edges $\tau_1\sim\tau_2$,
where $(u_1,v_1)$,  $(u_2,v_2)$ are coordinate systems attached to $\tau_1,\tau_2$ 
and matched by $\lambda_k:\tau_1\to\tau_2$,
and $\ma(u_1),\mb(u_1)$ represent the gluing data in those coordinates.
Note that the compatibility conditions of \S \ref{sec:g1vertex}, \S \ref{sec:crossing}
apply to the gluing data, not to the splines.
\end{definition}
\begin{definition} \label{def:g1splines}
A $G^1$ polygonal surface is called {\em rational} if for any $k\in\ii_1$
the gluing data $\Theta_k$ is defined by rational functions $\ma_k(u_k),\mb_k(u_k)\in \RR(u_k)$
in  (\ref{eq:jetb1}) or (\ref{eq:tbiso})
with respect to some (or any) coordinate systems $(u_k,v_k)$ attached to the glued edges.
In particular, we have
\begin{equation} \label{eq:edgeabc}
\ma_k(u_k)=\frac{b_k(u_k)}{a_k(u_k)}, \qquad
\mb_k(u_k)=\frac{c_k(u_k)}{a_k(u_k)}
\end{equation}
for some polynomials $a_k(u_k),b_k(u_k),c_k(u_k)\in\RR[u_k]$ without a common factor.
A $G^1$ function on a rational $G^1$ polygonal surface is a {\em (polynomial) spline}
if all restrictions $g_\ell$ ($\ell\in \ii_0$) to the polygons are polynomial functions.
\end{definition}
\begin{remark} \rm \label{rm:parametric}
The special case of constant $\ma_k(u_k),\mb_k(u_k)$ 
is equivalent  to {\em $C^1$ parametric continuity}, as termed in CAGD. 
In the context of \S \ref{sec:edgeglue}, formula (\ref{eq:jetb1}) then 
gives an affine-linear map that transforms $\sgma_1$
to a polygon $\sgma^*_1$ next to $\sgma_2$, with the edge $\tau_1$
mapped onto $\tau_2$. The $G^1$ functions on $\cM_0$ then correspond to
the $C^1$ functions on $\sgma^*_1\cup\sgma_2$.
\end{remark}
\begin{remark} \rm
Definitions \ref{def:g0complex}, \ref{def:g1complex} allow surfaces of arbitrary topology,
including non-orientable surfaces. To introduce an orientation, one can fix an ordering
on the standard coordinate systems (of Definition \ref{def:stcoor}) on each edge
and keep the orderings of coordinates compatible 
by properly relating them on the edges of each common polygon and around every common vertex.
If $\cM$ is an orientable surface with boundary, an orientation of $\cM$ 
gives orientations of the boundary components as well.

In \cite[Chapter 3]{Peters:2008}, orientation and {\em rigid embeddings} of polygons into $\RR^2$ 
allow to resolve the topological restriction \refpart{A2} automatically.
By itself, orientation is not relevant to the topological issues of \S \ref{sec:topology}.
In particular, all triangles (or rectangles) are linearly equivalent in any orientation
via the barycentric (or tensor product) coordinates presented in \S \ref{sec:polyfaces} here.
As discused in \S \ref{sec:topology},  the negative sign of $\mb_k(u)$ 
is determined by our choice of positive coordinates on each polygon. 
\end{remark}
\begin{remark} \rm
Definitions \ref{def:g0complex}, \ref{def:g1complex} allow edges to connect a vertex of $\cM$ to itself,
via sequences of glued polygonal vertices. For example, one can define $G^1$ surfaces without a boundary
from a single rectangle, by identifying their opposite edges in classical ways \cite{cltopology} to form a torus,
a Klein bottle or a projective space. Their $G^1$ gluing data can be even defined as in Remark \ref{rm:parametric},
utilizing parametric continuity along both obtained interior edges up to affine translations and reflections.
Similar surfaces with a single vertex and made up of two triangles are presented 
in \cite[Examples 6.16--6.18, 6.23, 6.31]{raimundas}.
The edges in the simple polygonal surfaces $\cM_0$, $\cM^*_0$ of \S \ref{sec:edgeglue}, \S \ref{sec:g1vertex}
connect different topological vertices.
\end{remark}
\begin{example} \rm \label{ex:bercovier} 
In \cite{Bercovier}, the $C^1$ structure of a mesh of quadrilaterals in $\RR^2$
is translated to a $G^1$ gluing structure on a set of rectangles (or copies of the square $[0,1]^2$)
via bilinear parametrizations of the quadrilaterals. 
The bilinear parametrizations are unique up to the symmetries 
of the tensor product coordinates of the rectangles (see \S \ref{sec:polyfaces} here).
Given two adjacent quadrilaterals $O_1P_1Q_1R_1$, $O_1P_1Q_2R_2$,
the identification of $C^1$ functions on their union with $G^1$ functions on
two rectangles (glued along the pre-images of $O_1P_1$) 
gives rational $G^1$ gluing data with polynomials 
$a_k$, $b_k$, $c_k$ in (\ref{eq:edgeabc}) of degree $\le 1, \le 2, \le 1$.
Analysis of the explicit B\'ezier coefficients in \cite[(29)]{Bercovier} reveals
that the polynomial $b_k$ is linear iff the {\em twist vectors}
\mbox{$\,\overrightarrow{\!O_1P_1}+\,\overrightarrow{\!Q_1R_1}$}, 
$\,\overrightarrow{\!O_1P_1}+\,\overrightarrow{\!Q_2R_2}$
are linearly dependent. 
The polynomial $a_k$ is a constant iff $\,\overrightarrow{\!R_2Q_2}$ 
is parallel to $\,\overrightarrow{\!O_1P_1}$.
In particular, we cannot have constant $a_k,c_k$ and quadratic $b_k$ in
this construction.  The functions $\ma_k$, $\mb_k$ are polynomials of degree 1
iff the twist $\,\overrightarrow{\!O_1P_1}+\,\overrightarrow{\!Q_2R_2}$ is zero.
If both twist vectors are zero, then $\mb_k$ is a constant and $\ma_k$ is a linear polynomial. 
It is commonly agreed (in \cite{flex} as well)
that constructions with zero twists lead to unsatisfactory spline spaces. 
\end{example}
\begin{example} \rm \label{eq:octahed}
The octahedral example in \cite{VidunasAMS} is constructed 
from 8 triangles, combinatorially glued in the same way as the Platonic octahedral.
The polygonal surface has 12 edges and 6 vertices. 
Topologically, it is a compact, orientable surface of genus $0$, homeomorphic to a sphere.
All vertices are crossing vertices, and the gluing data on each edge 
is the same symmetric gluing of Example \ref{ex:joining}.
A corresponding structure of a $C^1$ differential surface is given in \cite[\S 4]{VidunasAMS},
with all transition maps (around the edges and vertices)
being fractional-linear rational functions. 
A dimension formula for spline spaces of bounded degree is recalled 
in Example \ref{ex:platonic} here. 

Generally,  the rational gluing data with constant $\gamma(u)=\gamma$
and linear $\beta(u)=\lambda u+\eta$ is realized by the fractional-linear transition map
\begin{equation} \label{eq:frlintr}
(u,v)\mapsto \left( \frac{u+\eta\,v}{1-\lambda \,v}, \frac{\gamma\,v}{1-\lambda\,v} \right).
\end{equation}
In fact, any rational gluing data can be realized by a rational transition map.
A corresponding transition map for (\ref{eq:edgeabc}) is
\begin{equation}
(u,v)\mapsto \left( \frac{a_k(u)\,u+\eta(u)\,v}{a_k(u)-\lambda(u)\,v},
\frac{c_k(u)\,v}{a_k(u)-\lambda(u)\,v} \right)
\end{equation}
for any expression $b_k(u)=\lambda(u)\,u+\eta(u)$.
One can add $O(v^2)$ terms in the numerators or the denominator.
The composition of these transition maps around a vertex 
is the identity only if they are all fractional-linear maps as in (\ref{eq:frlintr}).
\end{example}

\subsection{Constructing a polygonal surface}
 \label{rm:gconstr}
 
The initial step of constructing a $G^1$ polygonal surface is choosing  the topology 
and incidence complex \cite{Ziegler00}
of the underlying $G^0$ polygonal surface $\cM$. 
For a surface of genus $0$ with boundary, one can simply use parametric continuity
of a polygonal mesh in $\RR^2$, or start with a polygonal mesh.
For a closed surface of genus $0$, one can start with a convex polyhedron in $\RR^3$ 
and copy the incidence relations of its edges and facets. For surfaces of genus $g$,
one can start with a polygonal subdivision of a regular $4g$-gon, 
copy its edge identifications and append parallel identifications of the opposite edges \cite{cltopology}.
For non-orientable surfaces, one can similarly start with an appropriate regular polygon
and append defining identifications of the opposite edges.

With the polygonal $G^0$ structure defined,
constructing a flexible rational $G^1$ gluing data on it is not a trivial task.
The main flow of the article is independent from this discussion.
For most polygonal structures, we can be proceed as follows:
\begin{itemize}
\item[\refpart{C1}] 
For each inner vertex, 
choose a fan of tangent sectors as in Figure \ref{fig:g1v}\refpart{a},
and choose vectors $\vb_1,\ldots,\vb_n=\vb_0$ on each boundary between the sectors. 
Here $n$ is the valency of a vertex under consideration.
For $k\in\{1,\ldots,n\}$, the linear relation between three consecutive vectors
\begin{equation} \label{eq:vects}
\vb_{k-1} =  \ma'_{k} \vb_k + \mb'_k \vb_{k+1} 
\qquad \mbox{or}\qquad
\vb_{k+1} =  -\frac{\ma'_{k}}{\mb'_k} \, \vb_k + \frac1{\mb'_k} \,\vb_{k+1} 
\end{equation}
will determine the end values (such as $\ma_k(0)=\ma'_k$, $\mb_k(0)=\mb'_k$
or $\ma_k(0)=-\ma'_k/\mb'_k$, \mbox{$\mb_k(0)=1/\mb'_k$})
of eventual glueing data along the edges.
The symmetric gluing of Example \ref{ex:hahnsym} is a possibility.
The specific end-values refer to particular standard coordinate systems 
attached to the polygonal vertices.
The set of crossing vertices is defined by this step.
\item[\refpart{C2}]  For each boundary vertex, choose a partial subdivision of $\RR^2$ into sectors
around the origin as in Figure \ref{fig:g1v}\refpart{b} keeping in mind condition \refpart{B2}.
The linear relations as in (\ref{eq:vects}) complete specification
of end values of eventual glueing data along all interior edges.
The subdivisions of Figure \ref{fig:g1v}\refpart{b} can be chosen keeping in mind 
the adjacent vertices, especially noting expected interpolation values at the adjacent crossing vertices.
\item[\refpart{C3}]  For each interior edge, 
adjust the decided end-values of the respective $\ma_k$, $\mb_k$ 
to a common coordinate system attached to the edge.
Examples of the adjustments are given in (\ref{eq:trcoor})--(\ref{eq:reders}) and (\ref{eq:rectuv}).
Generally, the applicable transformations 
of standard coordinates at both end-points 
of $\tau_1\sim\tau_2$ have the 
form (for $i\in\{1,2\}$)
\begin{equation}
(u_i,v_i)\mapsto (1-u_i-r_iv_i,q_iv_i),    \qquad \mbox{with } q_i>0.
\end{equation}
We have $q_i=1$ if the two adjacent edges to $\tau_i$ end up at the same ``height" $v_i$ from $\tau_i$.
This is automatic for rectangles and triangles.
The general adjustment is 
\begin{equation} \label{eq:genadj}
\big(\ma(u_1),\mb(u_1)\big) \mapsto 
\left( \frac{r_1}{q_1}-\frac{r_2}{q_1}\, \mb(1-u_1)-\frac{\ma(1-u_i)}{q_1}, \,
\frac{q_2}{q_1} \,\mb(1-u_1) \right).
\end{equation}
\item[\refpart{C4}]  For each interior edge, 
interpolate the adjusted values of the 
$\mb_k$-functions  at their end-vertices by 
fractional-linear functions. 
Linear interpolations are possible,
though generally $1/\mb_k$ has as many reasons to be a polynomial as $\mb_k$.
\item[\refpart{C5}]  For each interior edge connecting non-crossing vertices,
interpolate the adjusted values of the $\ma_k$-functions by fractional-linear functions. 
If $\mb_k$ has a denominator, then the respective $\ma_k$ should have the same denominator.
\end{itemize}
If there are no crossing vertices,
this defines a $G^1$ gluing data with linear or fractional-linear functions $\ma_k,\mb_k$.
If there are crossing vertices, it is tricky to cope with the conditions of Theorem \ref{cond:comp2},
especially if there are edges connecting crossing vertices. It is hard to avoid quadratic interpolation for $\ma_k$ then.
\begin{definition}
By a {\em crossing rim} on a $G^1$ polygonal surface we mean a sequence $\cE_0,\cE_1,\ldots,\cE_n$ 
of edges and a sequence $\cP_1,\ldots,\cP_n$ of crossing vertices 
such that $\cE_{j-1}$, $\cE_j$  are {\em opposite} edges at $\cP_j$ (as in Definition \ref{def:join})
 for $j\in\{1,\ldots,n\}$. It is a {\em maximal crossing rim} if the edges $\cE_0$, $\cE_n$ 
both have a non-crossing end-point. It is a {\em crossing equator} if $\cE_0=\cE_n$.
\end{definition}
The procedure of constructing a complete $G^1$ polygonal surface can be continued as follows:
\begin{itemize}
\item[\refpart{C6}] For each maximal crossing rim,
pick an edge $\cE_j$ on it and define its $\ma_k$
by linear or fractional-linear interpolation as in \refpart{C5}. 
Then subsequently (and possibly in both directions from $\cE_j$),
choose respective $\ma_k$ on the next adjacent edge by 
interpolating the two adjusted end-values and a derivative value determined by (\ref{eq:ddv1}).
Each $\ma_k$ can be a rational function with the same denominator as the respective $\mb_k$
and a quadratic (at worst) numerator.
\item[\refpart{C7}] For every edge on a crossing equator,
choose the derivative values $\ma'_k(0)$, $\ma'_k(1)$ 
and construct $\ma_k$ by cubic interpolation of the numerator 
(while keeping the denominator as in the respective $\mb_k$).
On a crossing equator with more than one edge, 
we may start with a fractional-linear interpolation on one edge $\cE_j$ as in \refpart{C6},
apply quadratic interpolations on the subsequent edges, and use the cubic interpolation
only on the last edge before returning to $\cE_j$.
\end{itemize}
To analyze feasibility of at most quadratic interpolations $\ma_k$ in \refpart{C7}, 
note that a quadratic function $h(x)$ satisfies 
\begin{equation} \label{eq:quadh}
h'(0)+h'(1)=2h(1)-2h(0).
\end{equation}
We have a fractional-linear $\mb_k=c_k/a_k$ on some edge of the equator. 
We use the same denominator for $\ma_k$.
Starting with a quadratic polynomial $h=\ma_ka_k$ interpolating the right values $h(0)=0$, $h(1)=r_1a_k(1)-r_2c_k(1)$  
and freely undetermined $h'(0)$, we find the value $h'(1)$ by (\ref{eq:quadh}),
transform it by (\ref{eq:genadj}) and use (\ref{eq:ddv1}) to determine the subsequent 
interpolation value $h'(0)$ for the next edge. Repeating this routine along the edges of the equator,
we obtain a linear equation for the original undetermined $h'(0)$. 
The linear equation might be degenerate, leading either to a free choice of this $h'(0)$ 
or no possibility for completing the quadratic $G^1$ data.
It is thus even unclear whether cubic $\beta_k$ on crossing equators can be avoided.

A choice of quadratic $\ma_k$ is easily possible if all $q_i=1$, $\mb_k=-1$. 
Then all interpolation values $\ma_k'(0)$ are transformed and related by (\ref{eq:ddv1}) invariantly
on all edges of an equator, hence the initial value $h'(0)$ can be chosen freely.
The target $\mb_k=-1$ restricts Step \refpart{C1}. 
We are led to solving (\ref{eq:tsbp1}) with all $\mb_i(0)=-1$.
Since the determinant is right automatically, 
we get $2\times2-1=3$ equations for the values $\ma_i(0)$ at each vertex.
Let us denote them  by $\ma_i$ for shorthand.
If $n=3$, there is only Hahn's symmetric choice.
If $n=4$, we must have a pair of joining edges:
\begin{equation} \label{bt1n4}
\ma_1=\ma_3=0,\quad \ma_4=-\ma_2, \quad\mbox{or}\quad 
\ma_2=\ma_4=0,\quad \ma_3=-\ma_1.
\end{equation}
For $n=5$, a Gr\"obner basis \cite{Wiki} gives the equations
\begin{equation} \label{bt1n5}
\ma_4=1+\ma_1\ma_2,  \qquad \ma_5=1+\ma_2\ma_3, \qquad
\ma_1+\ma_3=1+\ma_1\ma_2\ma_3.
\end{equation}
The equation system for $n=6$ is
\begin{align} \label{bt1n6} 
\ma_5=\ma_1+\ma_3-\ma_1\ma_2\ma_3, \  
\ma_6=\ma_2+\ma_4-\ma_2\ma_3\ma_4, \  
\ma_1\ma_2+\ma_1\ma_4+\ma_3\ma_4=2+\ma_1\ma_2\ma_3\ma_4. 
\end{align}
The solutions gives convenient alternatives to (\ref{eq:hsym}).

The relations (\ref{eq:quadh}), (\ref{eq:genadj}), (\ref{eq:ddv1}),  (\ref{eq:tsbp1}) 
appear to be manageable  with any fixed constant $\gamma_k$, $q_i$. 
In Step \refpart{C1}, one can prescribe constant $-\gamma_k>0$ as ratios of two numbers
attached to the polygonal sides of the edge. The product of the ratios $-\gamma_k$
around each edge must necessarily evaluate to $1$ as the determinant in (\ref{eq:tsbp1}).

It is worth mentioning here that $G^1$ gluing data can be defined 
by specifying two independent splines ``around" each edge,
via the corresponding syzygies, $M^1$-spaces of forthcoming \S \ref{eq:edgespace}
and Lemma \ref {lm:syzygy} \refpart{iii}. 
One can even choose freely two global independent splines 
(or more, if locally supported) 
and define $G^1$ glueing data to have them in the spline space.
Example \ref{ex:pfcontra} was constructed this way.


\section{Degrees of freedom}
\label{sec:freedom}

The main mathematical result of this article is the dimension formula (in Theorem \ref{th:dimform})
for spline spaces on rational $G^1$ polygonal surfacees made up of rectangles and triangles.
We prepare for this result by taking a closer look at rational $G^1$ gluing 
along edges and around vertices. 
This section adds the assumption of {\em rational} $G^1$ gluing 
to the basic setting of \S \ref{sec:edgeglue}--\S \ref{sec:crossing}. 
Further refinement to the context of rectangles and triangles is done in \S \ref{sec:boxdelta}.
The general structure of a spline basis in \S \ref{sec:generate} must be a good guidance
for counting the dimension of spaces of (polynomial or rational) splines 
on any rational $G^1$ surface.

The big example in \S \ref{sec:poctah} can be preliminarily read after \S \ref{sec:vertexv},
if the Bernstein-B\'ezier bases of polynomials on rectangles and triangles are familiar
(described here in  \S \ref{sec:polyfaces}).

\subsection{Jet evaluation maps}
\label{sec:jeteval}

We view $\RR[u,v]$ as the space of polynomial functions on $\RR^2$ in some coordinates $u,v$.
At a polygonal vertex $P$, let $u_P,v_P$ denote linear functions such that \mbox{$u_P=0$} and \mbox{$v_P=0$}
define the edges incident to $P$. Let $M(P)$ denote the linear space of bilinear functions in $u_P,v_P$.
We naturally identify $M(P)\cong\RR[u_P,v_P]/(u_P^2,v_P^2)$ and define the linear map
\begin{equation} \label{eq:wv}
W_P:\RR[u,v]\to M(P) 
\end{equation}
by following the quotient homomorphism of $\RR[u_P,v_P]/(u_P^2,v_P^2)$.
This map can be viewed as a partial evaluation 
of the jet $f\mapsto J_P^{2}f$ 
(or second order Taylor series) at $P$. 
We refer to the elements in $M(P)$ as {\em $J^{1,1}$-jets}. 

For a polygonal edge $\tau$ and a coordinate system $(u_\tau,v_\tau)$ attached to it,
let $M(\tau)$ similarly 
denote the space of polynomials  in $\RR[u_\tau,v_\tau]$ that are linear in $v_\tau$. 
As a space of functions on $\RR^2$, $M(\tau)$ does not depend on the coordinate system.
Using $M(\tau)\cong \RR[u_\tau,v_\tau]/(v_\tau^2)$, we 
define the linear map
\begin{equation}  \label{eq:we}
W_\tau:\RR[u,v]\to M(\tau) 
\end{equation}
as the quotient homomorphism of $\RR[u_\tau,v_\tau]/(v_\tau^2)$.
This represents taking the linear in $v_1$ Taylor approximation of $f\in\RR[u,v]$ (around the line containing $\tau$).
If $P$ is an end vertex 
of $\tau$, let 
\begin{equation}  \label{eq:wev}
W_{P,\tau}:M(\tau) 
\to M(P), 
\end{equation}
denote the specialization of $W_P$ onto $M(\tau)$. The coordinates $(u_\tau,v_\tau)$, $(u_P,v_P)$ 
are related  by a simple linear transformation, especially when $(u_\tau,v_\tau)$ is 
a standard coordinate system attached to $\tau$ with $P$ at $u_\tau=0$. 
\begin{definition} \label{eq:thorough}
Let $\cM$ be a $G^1$ polygonal surface.
The {\em support} of a spline is the set of polygons where it specializes to non-zero functions.
A spline on $\cM$  {\em thoroughly vanishes} along an edge $\tau_1\sim\tau_2$
if it is mapped to $0$ by $W_{\tau_1}$ and  $W_{\tau_2}$. 
In other words, the spline and its first order derivatives then vanish on the edge.
It is enough to require vanishing in one of the spaces $W_{\tau_1}$ or $W_{\tau_2}$.
A spline  {\em thoroughly vanishes} at a vertex $\{P_1,\ldots,P_n\}$
if it is mapped to $0$ by all $W_{P_k}$ for $k\in\{1,\ldots,n\}$. Here it is not enough to require vanishing
in one of the spaces, since the standard mixed derivatives $\partial^2/\partial u_i\partial v_i$
might be independent. 
\end{definition}

\subsection{Splines and syzygies}
\label{eq:edgespace}

Consider $G^1$ gluing $(\lambda,\Theta_0)$ of two polygons $\sgma_1$, $\sgma_2$ 
along their edges \mbox{$\tau_{1}\subset\sgma_{1}$}, \mbox{${\tau}_{2}\subset{\sgma}_{2}$} 
as in \S \ref{sec:edgeglue}, but with the additional assumption of rational gluing data. 
We have the coordinates $(u_1,v_1)$, $(u_2,v_2)$ as $(u_{\tau_1},v_{\tau_1})$, 
$(u_{\tau_2},v_{\tau_2})$, 
and concrete rational functions $\ma(u_1),\mb(u_1)$ in (\ref{eq:tbiso}).
Let $a(u_1),b(u_1),c(u_1)$ denote polynomials that express
\begin{equation}
\ma(u_1)=\frac{b(u_1)}{a(u_1)}, 
\qquad \mb(u_1)=\frac{c(u_1)}{a(u_1)}  
\end{equation}
as in \eqref{eq:edgeabc}. We are looking for polynomial splines.
The differentiability condition (\ref{eq:dirder}) becomes
\begin{equation} \label{eq:syzygy}
a(u_1)A(u_1)+b(u_1)B(u_1)+c(u_1)C(u_1)=0,
\end{equation}
where
\begin{equation} \label{eq:syzd}
A(u_1)=-\frac{\partial g_1}{\partial v_1}(u_1,0), \quad
B(u_1)=\frac{\partial g_1}{\partial u_1}(u_1,0), \quad
C(u_1)=\frac{\partial g_2}{\partial v_2}(u_1,0).
\end{equation}
In algebraic terms, the triple $(A(u_1),B(u_1),C(u_1))$ is a {\em syzygy} 
\cite[\S 5.1]{CoxSheaUse}, \cite{Wiki}
between the polynomials $a(u_1), b(u_1), c(u_1)$. 
Let $\cZ(a,b,c)$ denote the linear space of syzygies between $(a,b,c)$. 
As we recall in Lemma \ref{lm:syzygy}, this is a free module over the ring $\RR[u_1]$. 
The {\em degree} of a syzygy $(A,B,C)$ is $\max(\deg A,\deg B,\deg C)$. 

Conversely, given a syzygy $(A,B,C)$ 
we construct  a $G^1$ spline $(g_1, g_ 2)$ on $\cM_0$ by
\begin{align} \label{eq:edgeint1}
g_1(u_1,v_1)=c_0+\int_0^{u_1} \! B(t)dt-v_1A(u_1)+v_1^2E_1(u_1,v_1),\\
 \label{eq:edgeint2}
g_2(u_2,v_2)=c_0+\int_0^{u_2} \! B(t)dt+v_2C(v_2)+v_2^2E_2(u_2,v_2),
\end{align}
where $c_0\in\RR$ is any constant, 
and $E_1,E_2$ are any polynomials. 
If $g_1\in\ker W_{\tau_1}$, $g_2\in\ker W_{\tau_2}$,
then $(g_1,g_2)$ is a spline on $\cM_0$ with the corresponding syzygy (\ref{eq:syzd}) 
being the zero vector. These splines have only the last terms in (\ref{eq:edgeint1})--(\ref{eq:edgeint2}).
The spline space decomposes 
\begin{equation} \label{eq:s1comp}
S^1(\cM_0)=\ker W_{\tau_1}\oplus \ker W_{\tau_2}\oplus M^1(\tau_1,\tau_2),
\end{equation}
where $M^1(\tau_1,\tau_2)$ denotes the space of $G^1$ splines in $M(\tau_1)\oplus M(\tau_2)$.
The latter splines can be considered to have $E_1=E_2=0$ in (\ref{eq:edgeint1})--(\ref{eq:edgeint2}).
The correspondence between the splines and syzygies is given by the linear map
\begin{equation} \label{eq:mzz}
\psi:M^1(\tau_1,\tau_2)\to \cZ(a,b,c)
\end{equation}
defined by (\ref{eq:syzd}). Its kernel are the constant splines.
The space $S^1(\cM_0)$ will be understood when $\cZ(a,b,c)$ or $M^1(\tau_1,\tau_2)$ are known.
\begin{lemma} \label{lm:syzygy}
For polynomials $a, b, c \in\RR[u_1]$ with  
$\gcd(a,b,c)=1$ we have the following facts:
\begin{enumerate}
\item $\cZ(a,b,c)$ is a free $\RR[u_1]$-module of rank $2$. 
\item Let $d=\max(\deg a,\deg b,\deg c)$. There is a syzygy in $\cZ(a,b,c)$
of degree $\mu\le d/2$. There is a pair of generators 
of degrees $\mu$ and $d-\mu$.
\item A pair of generators $(A_1,B_1,C_1)$, $(A_2,B_2,C_2)$ of $\cZ(a,b,c)$ 
satisfies 
\begin{equation} \label{eq:syzcp}
h_0 \cdot (a,b,c) = 
(B_1C_2-B_2C_1, C_1A_2-C_2A_1, A_1B_2-A_2B_1).
\end{equation}
for some $h_0\in\RR\setminus\{0\}$.
\end{enumerate}
\end{lemma}
\begin{proof}
These are standard facts \cite[\S 6.4]{CoxSheaUse} in the theory of parametrization of curves
by {\em moving lines}, their {\em $\mu$-bases}. 
We prove \refpart{iii} here, since the formulation is somewhat stronger.
Let $Z_1,Z_2$ denote the two generators, respectively. 
They are linearly independent over the field $\RR(u_1)$.
The syzygy relation (\ref{eq:syzygy}) is orthogonality
in the 3-dimensional linear space over $\RR(u_1)$. 
Hence relation (\ref{eq:syzcp}) holds with $h_0\in\RR(u_1)$.
Since $a,b,c$ are coprime, $h_0\in\RR[u_1]$. If $\deg h_0>0$,
then $Z_1,Z_2$ are linearly dependent in $\RR[u_1]/(h_0)$.
There is then an $\RR[u_1]$-linear combination $Z_3=h_1Z_1+h_2Z_2$ 
with $\deg h_1,\deg h_2<\deg h_0$ and all components of $Z_3$ divisible by $h_0$.
The syzygy $Z_3/h_0$ would not be an $\RR[u_1]$ combination $Z_1$, $Z_2$,
contradicting the assumption that they 
generate $\cZ(a,b,c)$.
\end{proof}

\subsection{Evaluations at an edge vertex}
\label{sec:vertexv}

Keeping the same context of gluing two edges $\tau_1$, $\tau_2$,
let $P_1\in\tau_1,P_2\in\tau_2$ denote their endpoints $u_1=0$, $u_2=0$.
Let $\pi_1:M^1(\tau_1,\tau_2)\to M(\tau_1)$, $\pi_2:M^1(\tau_1,\tau_2)\to M(\tau_2)$
denote the projections from 
\begin{equation} \label{eq:mtt}
M^1(\tau_1,\tau_2)\subset M^1(\tau_1)\oplus M^1(\tau_2).
\end{equation}
The evaluation maps (of the $J^{1,1}$-jets at $P_1$, $P_2$)
\begin{equation} \label{eq:wps}
W_{P_1,\tau_1}\circ\pi_1: M^1(\tau_1,\tau_2)\to M(P_1), \ 
W_{P_2,\tau_2}\circ\pi_2: M^1(\tau_1,\tau_2)\to M(P_2)
\end{equation}
are important for constructing splines on a larger rational $G^1$ polygonal surface.
The images in $M(P_1)$, $M(P_2)$ 
of splines in the similar $M^1$-spaces on adjacent edges 
will have to match 
in order to form global splines. 
The matching is most direct if $(u_1,v_1)$, $(u_2,v_2)$ are 
standard coordinates  attached to $\tau_1$, $\tau_2$.
Let 
\begin{equation} \label{eq:w0}
W_0: M^1(\tau_1,\tau_2)\to M(P_1)\oplus M(P_2)
\end{equation}
denote the combined evaluation map $(W_{P_1,\tau_1}\circ\pi_1)\oplus(W_{P_2,\tau_2}\circ\pi_2)$.
We can denote (\ref{eq:wps}) shorter:
$W_{P_1,\tau_1}\circ\pi_1=\widehat{\pi}_1\circ W_0$
and $W_{P_2,\tau_2}\circ\pi_2=\widehat{\pi}_2\circ W_0$,
where $\widehat{\pi}_1$, $\widehat{\pi}_2$ are the projections from $M(P_1)\oplus M(P_2)$.
\begin{lemma} \label{lm:edgew}
\begin{itemize}
\item The maps $\widehat{\pi}_1\circ W_0$, 
$\widehat{\pi}_2\circ W_0$ 
are surjective. 
\item The dimension of $\img W_0$ equals $4$ if the edge $\tau_1\sim\tau_2$
is joining at the vertex $\{P_1,P_2\}$, and the dimension is $5$ otherwise. 
\end{itemize}
 \end{lemma}
\begin{proof}
We assume standard coordinates $(u_1,v_1)$, $(u_2,v_2)$ attached to $\tau_1$, $\tau_2$,
with $u_1=0$, $u_2=0$ at $P_1,P_2$, respectively. 
Then (\ref{eq:edgeint1})--(\ref{eq:edgeint2}) can be followed directly to conclude 
that  a syzygy 
\[
(a'+a''u_1+\ldots,b'+b''u_1+\ldots,c'+c''u_1+\ldots)
\]
leads to a spline evaluated by $W_0$ to 
\[
(c_0+b'u_1-a'v_1-a''u_1v_1, c_0+b'u_2+c'v_2+c''u_2v_2). 
\]
If $Z_1=(A_1,B_1,C_1)$, $Z_2=(A_2,B_2,C_2)$ are generators of $\cZ(a,b,c)$, 
then \mbox{$A_1B_2-A_2B_1\neq 0$}
at $u_1=0$ by Lemma \ref{lm:syzygy} \refpart{iii}. 
Explicit computation of the $\pi_1\circ W_0$ images of the splines 
corresponding to $Z_1,Z_2,u_1Z_1$ gives then the whole $M(P_1)$,
including the constants $\RR\subset M(P_1)$.
Surjectivity of $\pi_2\circ W_0$ 
follows similarly, from $B_1C_2-B_2C_1\neq 0$ at $u_1=0$.

For any spline $(g_1,g_2)\in M^1(\tau_1,\tau_2)$, 
the first order jet $J^1_{P_2}g_2$ is determined by $J^1_{P_1}g_1$
by the jet isomorphism (\ref{eq:jetb1}) at $u_1=0$.
Only the $u_2v_2$ term might be independent from $\widehat{\pi}_1\circ W_0(g_1)$.
To establish the (in)dependency, we look for a non-zero spline with $\widehat{\pi}_1\circ W_0(g_1)=0$.
We can only use the syzygies $u_1Z_1,u_1Z_2$ and produce 
$(0,b(0)u_2v_2)\in W_0$, with $A_1C_2-A_2C_1$ proportional to $b(u_1)$.
We have $b(0)=0$ exactly in the joining edge case. 
\end{proof}

\subsection{Separation of edge vertices}
\label{sec:serverts}

In the same context of gluing two edges $\tau_1$, $\tau_2$,
let $Q_1\in\tau_1,Q_2\in\tau_2$ denote the other endpoints $u_1=1$, $u_2=1$.
We have a similar map $W_1:M^1(\tau_1,\tau_2)\to M(Q_1)\oplus M(Q_2)$ 
as $W_0$ in (\ref{eq:w0}). 
\begin{definition} \label{def:sepspline}
A {\em separating spline} in $M^1(\tau_1,\tau_2)$ has the property that 
it attains different values at the end vertices $\{P_1,P_2\}$ and $\{Q_1,Q_2\}$.
An {\em offset spline} is a separating spline that has 
the first order and the mixed derivatives $\partial^2/\partial u_iv_i$ 
equal to $0$ at both end-vertices. 
An offset spline is mapped to constant $J^{1,1}$-jets by both $W_0$, $W_1$.
For example, let $Z_0=(A,B,C)$ be a syzygy with a non-zero function $B(u_1)$. 
Then a spline corresponding by (\ref{eq:edgeint1})--(\ref{eq:edgeint2}) to the syzygy $B(u_1)\,Z_0$ is separating,
because $\int_0^1 B(t)^2dt\neq 0$. A spline corresponding to the syzygy $u_1^2(1-u_1)^2\,B(u_1)\,Z_0$
is then an offset spline.
\end{definition}
\begin{lemma}  \label{lm:separate}
The image of the combined map 
\begin{equation} \label{eq:mmmm}
W_0\oplus W_1: M^1(\tau_1,\tau_2)\to M(P_1)\oplus M(P_2) \oplus  M(Q_1)\oplus M(Q_2)
\end{equation}
is $\img W_0\oplus \img W_1$.
\end{lemma}
\begin{proof}
Let $Z'_1=(A_1,B_1,C_1)$, $Z'_2=(A_2,B_2,C_2)$ denote generators of $\cZ(a,b,c)$.
Both $B_1$, $B_2$ cannot be zero functions
because of Lemma \ref{lm:syzygy} \refpart{iii} and $a(u_1)\neq 0$.
Therefore we have a separating spline $B_1Z'_1$ or $B_2Z'_2$.
Using it and the constant $1$ spline, we can linearly reduce any target element of 
$\img W_0\oplus \img W_1$ to a pair of $J^{1,1}$-jets without the constant terms.

Let $M_O^1(\tau_1,\tau_2)\subset M^1(\tau_1,\tau_2)$ denote the subspace of splines with $c_0=0$ 
in (\ref{eq:edgeint1})--(\ref{eq:edgeint2}).
Then the restriction 
\begin{equation} \label{eq:mz}
\psi_O:M_{O}^1(\tau_1,\tau_2)\to \cZ(a,b,c)
\end{equation}
of $\psi$ onto $M_O^1(\tau_1,\tau_2)$ is an isomorphism. 
Let $F_1$ denote an offset spline in $M_O^1(\tau_1,\tau_2)$.
There is a linear combination of $\psi_O^{-1}((1-u_1)^2 Z'_1)$ and $F_1$ that 
thoroughly vanishes at $\{Q_1,Q_2\}$. 
There is a similar linear combination of $\psi_O^{-1}((1-u_1)^2 Z'_2)$ and $F_1$.
We get two splines that can achieve any values of two first order derivatives at $\{P_1,P_2\}$
and map to $0$ in $ \img W_1$. 
Similarly, there are linear combinations of  $\psi_O^{-1}(u_1(1-u_1)^2 Z'_1)$ 
and $\psi_O^{-1}(u_1(1-u_1)^2 Z'_2)$ with $F_1$ that 
realize 1 or 2 degrees of freedom for the mixed derivatives at $\{P_1,P_2\}$,
depending on whether there is joining at this vertex. This gives the full $\img W_0$.
The same argument holds for $\img W_1$, 
with $F_1$ replaced by an offset spline vanishing at $\{Q_1,Q_2\}$, 
and $\psi_O^{-1}((1-u_1)^2 Z'_1)$ replaced by $\psi_O^{-1}(u_1^2 Z'_1)$, etc.
\end{proof}

To derive dimension formulas and linear bases for $G^1$ splines 
of bounded degree on polygonal surfaces, 
we need to estimate how high the degree has to be so that Lemma \ref{lm:separate} holds
for the bounded degree subspaces of all edge spaces $M^1(\tau_1,\tau_2)$. 
\begin{definition} \label{def:sseparate}
Following \cite[Definitions 2.14, 2.24]{raimundas}, we say that a subspace 
of $M^1(\tau_1,\tau_2)$ {\em separates vertices} if it has separating splines and the constant splines.
Thereby it has splines that evaluate to any prescribed values at the two end-vertices. 
The subspace {\em strongly separates vertices}
if additionally the first order derivatives in $\img \!W_0\,\oplus\, \img \!W_1$ 
have the maximal freedom. The subspace {\em completely separates vertices}
if Lemma \ref{lm:separate} applies to the restricted $\img \!W_0\,\oplus\, \img \!W_1$. 
A necessary condition of complete separation is existence of an offset spline.
\end{definition}

\subsection{Degrees of freedom at vertices}
\label{sec:aroundv}

Consider now the setting of \S \ref{sec:g1vertex} of gluing polygons 
$\sgma_1,\ldots,\sgma_n$ around the common vertex $\cP_0=\{P_1,\ldots,P_n\}$.
We additionally assume rational $G^1$ gluing data 
along the edge pairs $\tau_k\subset\sgma_k,\tau'_{k-1}\subset\sgma_{k-1}$
for $k\in\ii=\{1,\ldots,n\}$, and look for polynomial splines.
We extend the notation $P_0=P_n$, $\tau'_0=\tau_{n}$ to
$P_{n+1}=P_1$, $\tau_{n+1}=\tau_{1}$.
Analogous to the map $W_0$ in (\ref{eq:w0}), 
we are interested in the linear natural map
\begin{equation} \label{eq:wpo0v}
W_{\cP_0}: 
S^1(\cM^*_0)\to M(P_1)\oplus\cdots\oplus M(P_n)
\end{equation}
and the dimension of its image. 
This map factors through
\begin{equation} \label{eq:mmms}
M^1(\tau_1,\tau'_{n})\oplus M^1(\tau_2,\tau'_1)\oplus\cdots\oplus M^1(\tau_n,\tau'_{n-1}).
\end{equation}
For $k\in K$, let $\pi_k$ denote the projection to $M(P_k)$ from the direct sum in (\ref{eq:wpo0v}).
Then $\pi_k\circ W_{\cP_0}$ factors through
\begin{equation} \label{eq:mms}
M^1(\tau_k,\tau'_{k-1}) \oplus M^1(\tau_{k+1},\tau'_{k}).
\end{equation}
This factorization determines the relations between $M(P_k)$ and $M(P_{k-1})$, $M(P_{k+1})$.
The kernel of $W_{\cP_0}$ consists of the splines $(g_1,\ldots,g_n)$ 
with each \mbox{$g_k\in \ker W_{\tau_k} \cap \ker W_{\tau'_k}$}.

For $k\in K$ let us denote $\crossing_k=1$ if the edge $\tau_k\sim \tau'_{k-1}$ 
is joining at $\{P_k,P_{k-1}\}$, and $\crossing_k=0$ otherwise.
\begin{lemma} \label{prop:Hg}
Let us denote $\crossing_+=1$ if $\cP_0$ 
is a crossing vertex,  and let \mbox{$\crossing_+=0$} otherwise. 
Suppose that conditions $(\ref{eq:jetbp1})$ and \refpart{B1} are satisfied.
If $\cP_0$ is a crossing vertex, suppose that $(\ref{eq:ddv1})$--$(\ref{eq:ddv2})$ are satisfied.
Then
\begin{equation} 
\dim \img W_{\cP_0} = 3+n-\sum_{k=1}^n
\crossing_k + \crossing_{+}.
\end{equation}
The projections  $\pi_k\circ W_{\cP_0}$ are surjective.
\end{lemma}
\begin{proof}
Suppose first that all edges 
are joining at $\cP_0$. Then $\cP_0$ is a crossing vertex by condition \refpart{B1}. 
Let us choose an arbitrary $J^{1,1}$-jet in $M(P_4)$. 
To get corresponding jets in $M(P_3),M(P_2),M(P_1)$, 
we apply the transformations dual to (\ref{eq:trE3}):
\begin{equation} \label{eq:trE3a}
\left( \begin{array}{c} 1 \\ u_{k-1}\\ v_{k-1} \\  u_{k-1} v_{k-1}\!\!\! \end{array} \right)
=\left( \begin{array}{cccc} 1 & 0 & 0 & 0 \\
0 & 0 & \mb_k(0) & \mb'_k(0) \\ 0 & 1 & 0 & \ma'_k(0) \\
0 & 0 & 0 & \mb_k(0) \end{array} \right) 
\left( \begin{array}{c}  1 \\ u_{k}\\  v_{k} \\ u_{k} v_{k} \end{array} \right).
\end{equation}
These transformations are consistent with the relations in the images of
the $W_0$-maps $M^1(\tau_k,\tau'_{k-1})\to M(P_k)\oplus M(P_{k-1})$ of Lemma \ref{lm:edgew}
for $k=4,3,2$. Conditions (\ref{eq:ddv1})--(\ref{eq:ddv2}) ensure 
consistency in the image of \mbox{$M^1(\tau_1,\tau'_{4})\to M(P_1)\oplus M(P_{4})$} 
as well. We have thus a spline vector in (\ref{eq:mmms}) mapping to the formed 
jets in $M(P_1),\ldots,M(P_4)$.
The lift from (\ref{eq:mmms}) to $S^1(\cM^*_0)$ is possible 
because (for each $k\in \{1,2,3,4\}$) the splines in both components in (\ref{eq:mms}) match at $M(P_k)$.
It follows that $\pi_4\circ W_{\cP_0}$ and the other projections are surjective,
and that $\dim \img W_{\cP_0} = 4$.

If there is an edge that is not joining at $\cP_0$, we may assume it to be $\tau_1\sim \tau'_n$.
Let us choose an arbitrary jet in $M(P_n)$. Its $J^{1}$ part is transformed by the linear maps 
in (\ref{eq:jetbp1}) to unique $J^1$ parts in the other $M(P_k)$ consistent with the images of
the $W_0$-maps $M^1(\tau_k,\tau'_{k-1})\to M(P_k)\oplus M(P_{k-1})$. Dependency
of the $u_{n-1}v_{n-1}$ term (of the jet in $M(P_{n-1})$) on the $u_nv_n$ term is determined
by whether $\crossing_{n}=1$ or $0$. Dependency of the next $u_kv_k,u_{k+1}v_{k+1}$ terms
is similarly determined by $\crossing_{k+1}$. Once the $u_1v_1$ term is set, there is no cyclical restriction
on the $u_nv_n$ term. Hence we have $n-\sum_{k=1}^n \crossing_k$ degrees of freedom
for the $u_kv_k$ terms. The dimension formula follows, 
together with surjectivity of $\pi_n\circ W_{\cP_0}$ and the other projections.
\end{proof}

 
\begin{lemma} \label{prop:Hgb}
Suppose we have a boundary vertex $\cP_0$ defined with 
the current specifications but without the gluing identification of $\tau_1$ and $\tau'_n$.
Consider the map $W_{\cP_0}$ as in $(\ref{eq:wpo0v})$. 
Then the projections  $\pi_k\circ W_{\cP_0}$ are surjective, and
\begin{equation} 
\dim \img W_{\cP_0} = 3+n-\sum_{k=2}^n \crossing_k.
\end{equation}
\end{lemma}
\begin{proof}
The proof is similar to the second part of the proof of Lemma \ref{prop:Hg}. 
\end{proof}

\subsection{Generators of spline spaces} 
\label{sec:generate}

Now we consider a rational $G^1$ polygonal surface $\cM$,
possibly with many interior edges and vertices,
and the space $S^1(\cM)$ of 
splines on it.
The splines defined by polynomials of bounded degree
will form linear spaces of finite dimension. 
We are interested in counting these dimensions
and constructing bases of these spaces.

For each interior edge $\cE$ 
let $\cM_\cE$ denote the polygonal surface as in \S \ref{sec:edgeglue}, \S \ref{sec:vertexv}
with a single interior edge $\tau_1\sim\tau_2$ that uses the glueing data $\Theta_k$ of $\cE$.
By the construction in \S \ref{sec:edgeglue}, the end-points of $\cM_\cE$ are distinct even
if $\cE$ connects a vertex of $\cM$ with itself.
The projection map
\begin{equation} \label{eq:edgeproj}
\pi_\cE: S^1(\cM) \to M^1(\tau_1,\tau_2)
\end{equation}
is certainly not surjective when $\cE$ connects a vertex with itself.
However, any element of \mbox{$\ker (W_0\oplus W_1)\subset M^1(\tau_1,\tau_2)$} 
as in Lemma \ref{lm:separate}
will lift to a spline in $S^1(\cM)$ of the same degree, 
with the support on $\Omega_1\cup\Omega_2$
and thoroughly vanishing on all edges other than $\tau_1,\tau_2$.
One can take $E_1=E_2=0$ in (\ref{eq:edgeint1})--(\ref{eq:edgeint2}).
Lemma \ref{lm:separate} implies that 
\begin{equation} \label{eq:edge0}
\mbox{codim}\; \ker (W_0\oplus W_1)= 
\dim \img W_0+\dim \img W_1
\end{equation}
inside $M^1(\tau_1,\tau_2)$.
By Lemma \ref{eq:wps}, $\dim \img W_k\in\{4,5\}$ for $k=1,2$.

For a boundary edge $\tau$, we have the similar projection 
\begin{equation} \label{eq:edgeprojb}
\pi_\tau: S^1(\cM) \to M(\tau)
\end{equation}
by following (\ref{eq:we}). Let $P,Q$ denote the end-points of $\tau$.
With reference to (\ref{eq:wev}), any element of $\ker W_{P,\tau}\cap \ker W_{Q,\tau}\subset M(\tau)$
similarly lifts to a spline in $S^1(\cM)$, thoroughly vanishing on all edges expect $\tau$. 
The co-dimension of these splines inside $M(\tau)$ equals $8$.

The splines that are mapped to zero by all projections (\ref{eq:edgeproj}), (\ref{eq:edgeprojb}) 
are the splines that throughly vanish on all edges.
For each polygon $\sgma$, let $L_{\sgma}$ denote the product of linear equations $l_\tau$
of all edges $\tau$ of $\sgma$. Then any polynomial multiple of $L_{\sgma}^2$ 
(as a function on $\sgma$) lifts to spline in $S^1(\cM)$ with the support on $\Omega$. 

We described all splines that thoroughly vanish at all vertices.
To combine and lift the $J^{1,1}$-jet spaces $W(P)$ at the polygonal vertices to global splines in $S^1(\cM)$,
consider an interior vertex $\cP=\{P_1,\ldots,P_n\}$ of $\cM$.
Let $\cM^*_{\cP}$ denote the polygonal surface as in \S \ref{sec:g1vertex}, \S \ref{sec:aroundv},
with a single interior vertex $\cP_0$ of valency $n$ and the same gluing data as on the edges of $\cM$
incident to $\cP$. The boundary vertices of $\cM^*_{\cP}$ are all distinct from each other
and from $\cP_0$ by the construction in \S \ref{sec:g1vertex}.  We have the projection map
\begin{equation} \label{eq:vertproj}
\pi_\cP: S^1(\cM) \to  S^1(\cM^*_{\cP_0}) 
\end{equation}
that we compose with $W_{\cP_0}$ in (\ref{eq:wpo0v}). 
Let $S^1_0(\cM^*_{\cP_0})$ denote the subspace of $S^1(\cM^*_{\cP_0})$
of splines that 
throughly vanish on the boundary edges. 
The image of $S^1_0(\cM^*_{\cP_0})$ under $W_{\cP_0}$ is the same as the whole image of $W_{\cP_0}$,
because the spline pieces $(g_k,g_{k-1})$ in each $M^1(\tau_k,\tau'_{k-1})$ can be modified 
without changing their interface image in $M(P_k)\oplus M(P_{k-1})$
to thorough vanishing at the respective boundary vertex
by multiplying the corresponding syzygies by $(1-u_1)^2$ and using offset splines.
An element $h\in \img W_{\cP_0}$ 
is lifted to a global spline in $S^1(\cM)$ as follows.
First we lift $h$ to a spline $(g_1,\ldots,g_n)$ in the subspace $S^1_0(\cM^*_{\cP_0})$ as just described.
Each polygon $\Omega$ of $\cM$ incident to $\cP$ may be matched with several polygons of $\cM^*_{\cP_0}$,
corresponding to those $P_k\in \cP$ that are vertices of $\Omega$.
We assign to this $\Omega$ a polynomial restriction that equals the sum of the corresponding components $g_k$.
In this sum, the $J^{1,1}$-jet in any $M(P_k)$ 
is non-zero only in one term, giving consistency with $h\in \img W_{\cP_0}$. 
The $J^{1,1}$-jets are properly related by the transformations like (\ref{eq:trE3a}).
If there is an edge $\tau'_1\sim\tau'_2$ of $\cM$ that connects  $\cP$ with itself, 
the spline image in $M(\tau'_1)\oplus M(\tau'_2)$ will be the sum of two terms
corresponding to the endpoints of $\tau'_1\sim\tau'_2$, both in $\cP$.
Assigning zero polynomials to the polygons of $\cM$ that are not incident to $\cP$
completes a lift of $h$ to a spline in $S^1(\cM)$.

In summary, the spline space $S^1(\cM)$ can be generated by the splines of these kinds:
\begin{itemize}
\item[\refpart{D1}] 
At each interior vertex $\cP$, one can choose finitely many splines 
realizing the degrees of freedom enumerated in Lemma \ref{prop:Hg},
and evaluating to $0$ on all polygons not incident to $\cP$.
Particularly among these splines:
\begin{enumerate}
\item There is a spline with the value 1 
and the 
other jet terms (i.e., the first order derivatives and the mixed derivatives)
equal to $0$ at $\cP$.
Its projections to the spaces $M^1(\tau_k,\tau'_{k-1})$ of incident edges
$\tau_k\sim\tau'_{k-1}$
are offset splines evaluating to 1 at $P_k\in\cP$ and to 0 at the other end-vertex.
\item There are 2 independent splines that vanish at $\cP$
and have (some) non-zero first order derivatives at $\cP$. 
Their standard mixed derivatives might be forced to $0$ if there are no edges
joining at $\cP$.
\item With reference to Lemma \ref{prop:Hg}, there are $n-\sum e_k+e_+$
independent splines that vanish with the first order derivatives at $\cP$.
They have some mixed derivatives non-zero.
\end{enumerate}
\item[\refpart{D2}] At a boundary vertex, 
one can similarly choose finitely many splines 
realizing the degrees of freedom enumerated in Lemma \ref{prop:Hgb}.
\item[\refpart{D3}] At an interior edge $\tau_1\sim \tau_2$, 
we have the splines lifted from $\ker(W_0\oplus W_1)$ as described
between the formulas (\ref{eq:edgeproj}) and (\ref{eq:edge0}).
These splines vanish on all polygons not containing $\tau_1$ or $\tau_2$,
and thoroughly vanish on all other edges.
\item[\refpart{D4}] At a boundary edge $\tau$, 
we have the space of splines similarly vanishing on the polygons not containing $\tau$, etc. 
The co-dimension of this space inside $M(\tau)$ equals $8$,
as mentioned above.
\item[\refpart{D5}] On each polygon $\Omega$, we have the splines
that are polynomial multiples of $L_{\Omega}^2$ on $\Omega$
and zero on the other polygons.
\end{itemize}
This break up offers a general strategy for counting the dimension of
spline spaces of bounded degree and constructing their bases.
Section \ref{eq:spdim} demonstrates this on the example of 
$G^1$ polygonal surfaces made up of rectangles and triangles.

The splines in \refpart{D3}--\refpart{D5} can be chosen from distinct subpaces
of the direct sum $R[u,v]$-copies over the set of polygons. 
The subspaces (one for each edge and for each polygon) are defined
by vanishing everywhere except in one $M^1$-projection or on one polygon.
A set of generating splines in \refpart{D1}--\refpart{D2} may project 
extensively to the incident $M^1$-spaces.


\section{Rectangle-triangle surfaces}
\label{sec:boxdelta}

From now on, we consider rational $G^1$ polygonal surfaces $\cM$ made up of rectangles and triangles.
The main forthcoming results are dimension formulas for the spaces of $G^1$ 
splines  of bounded degree on these surfaces. This is presented in \S \ref{sec:dformula}. 
Sections \ref{sec:asyz}, \ref{sec:sepverts} refine \S \ref{eq:edgespace}--\S \ref{sec:aroundv}
for the context of rectangles and triangles.

Here is our degree convention for $G^1$ splines on a rational $G^1$ polygonal surface $\cM$ 
made up of rectangles and triangles.
For an integer $k\ge 0$, let $S^1_k(\cM)$ denote the space of 
splines such that all restrictions to the triangles have degree $\le k$,  
and all restrictions to the rectangles have bidegree $\le (k,k)$.
We say that a spline {\em has degree $k$} if it is in $S^1_k(\cM)$
but not in $S^1_{k-1}(\cM)$. 
In the context of formulas (\ref{eq:s1comp}), (\ref{eq:mtt}),
let $M^1_k(\tau_1,\tau_2)$ denote the subspace of $M^1(\tau_1,\tau_2)$
of splines 
\begin{equation} \label{eq:splines2}
(h_0(u)+h_1(u)v_1,h_0(u)+h_2(u)v_2)
\end{equation}
such that  $\deg h_0(u)\le k$, and for $j\in\{1,2\}$ we have $\deg h_j(u)\le k$
if $\tau_j$ is an edge of a rectangle, or $\deg h_j(u)\le k-1$
if $\tau_j$ is an edge of a triangle.

\subsection{Partitioning the space of splines}
\label{eq:spdim}

Here we estimate the dimension of each subspace in the partitioning \refpart{D1}--\refpart{D5} 
of the whole spline space, and prove a preliminary dimension formula.
Following the partitioning in \S \ref{sec:generate}, we have these spline spaces: 
\begin{enumerate}
\item[\refpart{E1}] For large enough $k$, each vertex $\cP$ contributes
\begin{equation} \label{eq:E1spline}
3+n(\cP)+e_{+}(\cP)- e_{\perp}(\cP) 
\end{equation} 
independent splines by \refpart{D1}--\refpart{D2}, or Lemmas \ref{prop:Hg}, \ref{prop:Hgb}.
Here $n(\cP)$ is the valency of $\cP$; 
$e_{+}(\cP)=1$ if $\cP$ is a crossing vertex, and $e_{+}(\cP)=0$ otherwise;
$e_{\perp}(\cP)$ 
counts the number of instances  when an edge is joining at $\cP$.
Note that an edge may be joining $\cP$ at both of its end-points, contributing 2 to $e_{\perp}(\cP)$.
\item[\refpart{E2}]  For large enough $k$, each interior edge $\cE=\tau_1\sim \tau_2$ contributes
\begin{equation} \label{eq:E2spline}
\dim M^1_k(\tau_1,\tau_2)-10+e_{\perp}(\cE)
\end{equation} 
independent splines by \refpart{D3}.
Here $e_{\perp}(\cE)\in\{0,1,2\}$ counts the number of instances when $\cE$ is joining a vertex.
\item[\refpart{E3}]  By \refpart{D4}, a boundary edge on a rectangle contributes $2k-6$ independent splines for $k\ge 3$,
and a boundary edge on a triangle contributes $2k-7$ independent splines for $k\ge 4$.
\item[\refpart{E4}]  By \refpart{D5}, each rectangle contributes $(k-3)^2$ independent splines for $k\ge 3$,
and each triangle contributes $(k-4)(k-5)/2$ independent splines for $k\ge 4$.
\end{enumerate}
The degree $k$ in \refpart{E1}--\refpart{E2} is large enough if all edge spaces 
$M^1_k(\tau_1,\tau_2)$ completely separate the end-vertices in the sense of Definition \ref{def:sseparate}.
In \S \ref{sec:sepverts} we determine an explicit lower bound for such $k$.
Before that, \S \ref{sec:asyz} gives an explicit expression for $\dim M^1_k(\tau_1,\tau_2)$ for the large enough $k$.

Let $N_\Box$, $N_\Delta$, $N_0$, $N_0^+$, $N_1$, $N_1^\partial$  
denote the number of rectangles, triangles, all vertices, crossing vertices, all edges
and boundary edges of $\cM$, 
respectively.
\begin{lemma} \label{lm:dimsep}
If $k\ge 4$ and all spaces $M^1_k(\tau_1,\tau_2)$ completely separate the vertices, then
\begin{align*}
\dim S^1_k(\cM)= &\; 3\,N_0+N_0^+ 
+ \sum_{\tau_1\sim\tau_2} \dim M^1_k(\tau_1,\tau_2)-10\,N_1 \nonumber\\[1pt]
& + (2k+3) \,N_1^\partial 
+\#\{\mbox{\rm boundary edges on rectangles}\} \\ 
& + \left( k^2-6k+13 \right) N_{\Box} 
+ \frac{k^2-9k+26}{2} \, N_{\Delta}. 
\nonumber
\end{align*}
\end{lemma}
\begin{proof}
Summing up the dimension of the spaces in  \refpart{E1}--\refpart{E2}, we have
\begin{align*}
& \sum_{\mbox{\scriptsize vertices}\,\cP}  n(\cP)=  4\, N_\Box + 3 \, N_\Delta,  \qquad
\sum_{\mbox{\scriptsize vertices}\,\cP}  e_+(\cP)= N_0^+, \\
& \sum_{\mbox{\scriptsize vertices}\,\cP} e_{\perp}(\cP)= \, 
\sum_{\mbox{\scriptsize edges}\;\,\cE} e_{\perp}(\cE).
\end{align*}
In particular, the $e_{\perp}$'s cancel out. Adding the spaces in \refpart{E3}--\refpart{E4}
gives the stated dimension formula.
\end{proof}

Building up a linear basis for $S^1_k(\cM)$ following the partitioning  \refpart{E1}--\refpart{E4} is straightforward.
Even if some spaces $M^1_k(\tau_1,\tau_2)$ do not separate the vertices completely, 
the dimension formula may still hold, as exemplified in  \S \ref{sec:exbasis}.
The basis structure may need to be adjusted merely
by taking into account dependencies between the spaces $\img W_0\subset M(P_1)\oplus M(P_2)$, 
$\img W_1\subset M(Q_1)\oplus M(Q_2)$ of \S \ref{sec:serverts}.
as implicitly suggested in the proof of Lemma \ref{lm:lbound}.

\subsection{Dimension and the syzygy module}
\label{sec:asyz}

Here we go back to the setting of \S \ref{sec:vertexv} with the additional assumption 
of rectangles and triangles, and compute $\dim M^1_k(\tau_1,\tau_2)$.

For $j\in\{1,2\}$, let $r_j=1$ if $\sgma_j$ is a rectangle, and let $r_j=0$ if $\sgma_j$ is a triangle.
The splines in $M^1_k(\tau_1,\tau_2)$ correspond to the syzygies $(A,B,C)\in\cZ(a,b,c)$
such that
\begin{equation} \label{eq:syzdeg}
\deg A\le k-1+r_1,\qquad \deg B\le k-1, \qquad \deg C\le k-1+r_2.
\end{equation}
Let $\cZ_k(a,b,c)$ denote the linear subspace of $\cZ(a,b,c)$ of syzygies with these degree specifications.
\begin{lemma} \label{lm:dimmz}
$\dim M^1_k(\tau_1,\tau_2)=\dim \cZ_k(a,b,c)+1$.
\end{lemma}
\begin{proof}
The map $\psi$ in (\ref{eq:mzz}) restricts to surjective $M^1_k(\tau_1,\tau_2)\to\cZ_k(a,b,c)$.
The kernel of this map is one-dimensional.
\end{proof}
If both $\sgma_1,\sgma_2$ are triangles, then $\cZ_k(a,b,c)$ is the space of syzygies of degree $\le k-1$.
If a rectangle is involved, the $\cZ_k$-grading of $\cZ(a,b,c)$ is twisted with respect to the degree grading.
Nevertheless, Lemma \ref{lm:syzygy} can be adjusted helpfully.
The following definitions and lemma generalize to any grading of $\cZ(a,b,c)$.
The lemma can be proved as a straightforward exercise in homological algebra \cite{Wiki}. 
\begin{definition}
We say that a syzygy $Z\in \cZ(a,b,c)$ has {\em twisted degree} $k$ 
if it is in $\cZ_k(a,b,c)$ but not in $\cZ_{k-1}(a,b,c)$.
We denote the twisted degree by $\tdeg{Z}$.
The {\em leading term} of a syzygy $Z\in \cZ(a,b,c)$ with $\tdeg{Z}=k$ is the vector 
$
\big( \star u_1^{k-1+r_1}, \star \,  u_1^{k-1}, \star \,  u_1^{k-1+r_2} \big)
$
of the expected leading terms of the three components of $Z$.
A pair of generators of $\cZ(a,b,c)$ is called {\em minimal} if their leading terms are linearly independent.
\end{definition}
\begin{lemma}  \label{lm:twsyzygy}
\begin{enumerate}
\item There are non-negative integers $d_1,d_2$ such that 
any pair of minimal generators of $\cZ(a,b,c)$ is in 
$\cZ_{d_1}(a,b,c)\times\cZ_{d_2}(a,b,c)$.
\item Suppose that $d_1\le d_2$. We have
\[
\dim \cZ_k(a,b,c)=\left\{ \begin{array}{rl} 
0, & \mbox{if } k<d_1; \\
k-d_1+1, & \mbox{if } d_1\le k<d_2; \\
2k-d_1-d_2+2, & \mbox{if } k\ge d_2.
\end{array} \right.
\]
\item $d_1+d_2=\max( r_1+\deg a, \deg b, r_2+\deg c) + 2-r_1-r_2.$
\end{enumerate}
\end{lemma}
\begin{proof}
Starting with any two generators, 
we can reduce the twisted degree of (at least) one of them
until we have a minimal pair $Z_1,Z_2$.  
Let $d_1=\tdeg{Z_1}$, $d_2=\tdeg{Z_2}$.
The twisted degree of any $\RR[u_1]$ linear combination $h_1Z_1+h_2Z_2$
equals $\max(\deg h_1+\tdeg{Z_1},\deg h_2+\tdeg{Z_2})$.
The first two claims follow.

By Lemma \ref{lm:syzygy} \refpart{iii}, we have
\begin{align*}
\deg a \le & \, d_1+d_2+r_2-2,\\
\deg b \le & \, d_1+d_2+r_1+r_2-2,\\
\deg c \le & \, d_1+d_2+r_1-2.
\end{align*}
Since the leading terms of $Z_1,Z_2$ are linearly independent,
there is an equality for at least one degree.
Claim \refpart{iii} follows. 
\end{proof}
\begin{corollary}  \label{cor:m1dim}
Let $d_1,d_2$ denote the twisted degrees as in Lemma $\ref{lm:twsyzygy}$. 
Then
\[
\dim M_k^1(\tau_1,\tau_2) \ge 2k-\max( r_1+\deg a, \deg b, r_2+\deg c) + r_1 + r_2+1,
\]
The exact equality holds when $k\ge \max(d_1,d_2)$.
\end{corollary}
\begin{proof} Combine Lemmas \ref{lm:dimmz} and \ref{lm:twsyzygy}.
In particular, part \refpart{ii} of Lemma \ref{lm:twsyzygy} implies
\begin{equation}
\dim Z_k(a,b,c) \ge 2k-d_1-d_2+2.
\end{equation}
\end{proof}

\subsection{Separation of vertices}
\label{sec:sepverts}

Here we combine the settings of  \S \ref{sec:serverts}, \S \ref{sec:asyz}
and address the question of ``large enough $k$" in \refpart{E1}--\refpart{E2}.
We find sufficient degree bounds for existence of separating splines of Definition \ref{def:sepspline},
and for complete separation of vertices 
as in Definition \ref{def:sseparate}.

Let $d_1,d_2$ denote the twisted degree of the minimal generators of the syzygy module $\cZ(a,b,c)$. 
We assume $d_1\le d_2$.
Let $Z_1=(A_1,B_1,C_1)$, $Z_2=(A_2,B_2,C_2)$ denote a minimal pair of the syzygy generators,
of degree $d_1,d_2$, respectively. 

Note that $d_1=0$ only if $B_1=0$ and both polygons $\sgma_1$, $\sgma_2$ are rectangles.

Usually the space $M^1(\tau_1,\tau_2)$  has separating splines of degree $d_1$ or slightly larger.
For example, syzygies $(A,B,C)$ with constant $B\neq 0$ give these splines.
However, Example \ref{ex:lagrange} shows that the construction in Definition \ref{def:sseparate}
can be optimal in very special cases. 
\begin{lemma} \label{lm:sepse}
There are separating splines in $M^1(\tau_1,\tau_2)$ of degree 
\mbox{$\le\max(2d_1-1,d_2)$}. 
\end{lemma}
\begin{proof}
If $B_1\neq 0$, then $B_1Z_1$ gives a separating spline of degree $2d_1-1$
as demonstrated in Definition \ref{def:sseparate}. 
If $B_1=0$ then $B_2(u_1)$ does not change the sign on $u_1\in[0,1]$ because of 
Lemma \ref{lm:syzygy} \refpart{iii} and $c(u_1)/a(u_1)<0$. Then $Z_2$ gives a separating spline.
\end{proof}
\begin{example} \rm \label{ex:lagrange}
Recall {\em Legendre orthogonal polynomials} $P_n(x)$ \cite{Wiki}.
They have the property that $\int_{-1}^1 P_n(x)\,x^kdx=0$ for any degree $n$ and all $0\le k<n$. 
To switch to the integration interval $[0,1]$ as in (\ref{eq:edgeint1}), we normalize $p_n(x)=P_n(2x-1)$.
Consider the rational $G^1$ gluing data with
\begin{align*}
a(u_1)= & \, p_n(u_1)^2+p_{n-1}(u_1)^2,  & \hspace{-10pt}
b(u_1)= -p_{n-1}(u_1)^2, \\
c(u_1)= & \, p_n(u_1)p_{n-1}(u_1)-p_n(u_1)^2-p_{n-1}(u_1)^2.
\end{align*}
Since $p_n(x),p_{n-1}(x)$ do not have common roots \cite[Theorem 5.4.2]{specfaar}, we have $a(x)>0$, $c(x)<0$.
The syzygy module is generated by
\begin{align*}
\big( p_{n-1}(u_1),p_{n}(u_1),p_{n-1}(u_1) \big), \qquad
\big( p_n(u_1)-p_{n-1}(u_1),-p_{n-1}(u_1),p_n(u_1) \big).
\end{align*}
There are no separating splines of degree $<2n-1$.
\end{example}
\begin{lemma} \label{lm:csepo}
The space $M_k^1(\tau_1,\tau_2)$ completely separates the end-vertices if
it has an offset spline and $k\ge d_2+3$.  
\end{lemma}
\begin{proof}
We apply the construction with $\psi_0^{-1}$ in the proof of Lemma \ref{lm:separate}
to the splines $(1-u_1)^2Z_k$, $(1-u_1)^2u_1Z_k$, $u_1^2Z_k$, $u_1^2(1-u_1)Z_k$
with $k=1,2$.
\end{proof}
\begin{lemma} \label{lm:offss}
The space $M_k^1(\tau_1,\tau_2)$ has an offset spline if 
\mbox{$k\ge \max(2d_1+3,d_2+4)$}.
\end{lemma}
\begin{proof}
Let $L_0=u_1^2(1-u_1)^2$. 
Consider the splines 
$L_0B_1Z_1$ and $L_0Z_2$ in the same way as in the proof of Lemma \ref{lm:sepse}. 
\end{proof}
\begin{corollary} \label{eq:csep}
The space $M_k^1(\tau_1,\tau_2)$ completely separates the end-vertices if
\mbox{$k\ge \max(2d_1+3,d_2+4)$}.  
\end{corollary}
\begin{proof} Combine Lemmas \ref{lm:csepo} and \ref{lm:offss}.
\end{proof}
\begin{remark} \rm
The bound $d_2+3$ in Lemma \ref{lm:csepo} can be improved to \mbox{$d_2+2$}
only if the edge $\tau_1\sim\tau_2$ is joining at both end-vertices.
Then $\img W_0\oplus \img W_1$ might be covered by working with
$(1-u_1)^2Z_k$, $u_1(1-u_1)Z_k$, $u_1^2Z_k$, $k\in\{1,2\}$, 
as in the example case in \S \ref{sec:ef} where $d_1=1$, $d_2=2$.

The bound in Lemma \ref{lm:offss} can be reduced to $\max(2d_1+2,d_2+3)$  
if $\tau_1\sim\tau_2$ is joining at an end-vertex, say at $u_1=0$. 
Then use the factor \mbox{$L_0=u_1^2(1-u_1)$} in the proof,
and adjust the two splines linearly by $L_0Z_1$ to annihilate the mixed derivatives at $u_0=0$.
This is the case in  \S \ref{sec:ab} and  \S \ref{sec:be}.
The bound can be reduced to $\max(2d_1+1,d_2+2)$
if $\tau_1\sim\tau_2$ is joining at both end-vertices and $d_1<d_2$
by using $L_0=u_1(1-u_1)$ and adjustment by $u_1L_0Z_1$, $(1-u_1)L_0Z_1$.
This is the case in  \S \ref{sec:ef}. 

Generally, the bound $d_2+4$ in Corollary \ref{eq:csep} cannot be improved if $d_1=d_2$,
as then $\dim M_k^1(\tau_1,\tau_2)=9$ for $k=d_2+3$.
The bound $2d_2+3$ is sharply achieved by the same construction as in Example \ref{ex:lagrange},
with the Legendre polynomials replaced by {\em Gegenbauer polynomials} \cite{Wiki} $C^{(\xi)}_n(x)$
with $\xi=5/2$. 
We have \mbox{$\int_{-1}^1 C^{(\xi)}_n(x) \, x^k (1-x^2)^{\xi-1/2} dx=0$}
for all degrees $n$ and $0\le k<n$.

The bounds in Lemmas \ref{lm:sepse}, \ref{lm:offss}, Corollary \ref{eq:csep}
can be softened using $\max(2d_1,d_2)\le d_1+d_2$ 
and part \refpart{iii} of Lemma \ref{lm:twsyzygy}.
\end{remark}

\subsection{The dimension formula}
\label{sec:dformula}

Now we are ready to state the general dimension formula for splines spaces
on a rational $G^1$ polygonal surface $\cM$ made up of rectangles and triangles.
For an interior edge $\cE$ of $\cM$, 
let $d_1(\cE)$, $d_2(\cE)$ denote the twisted degree of minimal generators
of $\cZ(a,b,c)$, with $d_1(\cE)\le d_2(\cE)$. 
If $(a,b,c)$ is the polynomial gluing data for $\cE$ 
as in Definition \ref{def:g1splines}, let
\begin{equation} \label{eq:delt}
\delta(\cE)= 
\max( r_1+\deg a, \deg b, r_2+\deg c). 
\end{equation}
Here $r_1,r_2$ are as in \S \ref{sec:asyz}.
Note that $\delta(\cE)>0$ if a rectangle is involved in gluing along $\cE$.
Let $\delta(\cM)$ denote the sum of all $\delta(\cE)$ over the interior edges $\cE$.
\begin{theorem}   \label{th:dimform}
Let $k$ denote an integer such that $k\ge (2d_1(\cE)+3,d_2(\cE)+4)$ for all interior edges $\cE$. 
Then
\begin{align}  \label{eq:dimform}
\dim S^1_k(\cM)= & \; 
3\,N_0+N_0^+ -\delta(\cM) +(k-{\textstyle\frac{5}{2}}) \,N_1^\partial   \nonumber  \\[1pt]
& + \left( k^2-2k-1 \right)  N_{\Box} 
+ \frac{k^2-3k-1}{2} \, N_{\Delta}.  
\end{align}
\end{theorem}
\begin{proof}
By Corollary \ref{eq:csep}, the spaces $M^1(\tau_1,\tau_2)$ of all interior edges
completely separate their end-vertices. We combine Lemma \ref{lm:dimsep} and Corollary \ref{cor:m1dim},
observing that summing up all $r_1,r_2$ and boundary edges on rectangles counts
all edges of rectangles. We obtain
\begin{align} \label{eq:dimforme}
\dim S^1_k(\cM)=   
& \; 3\,N_0+N_0^+ +(2k-9) \,N_1 +2\,N_1^\partial-\delta(\cM) \\[1pt]
& +\left( k^2-6k+17 \right) N_{\Box} 
+ \frac{k^2-9k+26}{2} \, N_{\Delta}. 
\end{align}
We eliminate the number of edges 
using
\begin{equation} \label{eq:polygee}
2\,N_1=4\, N_{\Box}+3\,N_{\Delta}+N_1^\partial
\end{equation}
and get the claimed formula. 
Note that the numbers of triangles and boundary edges have the same parity.
\end{proof}
 
\begin{corollary}
Suppose that $\delta(\cE)=1$ for all interior edges $\cE$ of $\cM$, and $k\ge 6$. 
Then
\begin{align*}
\dim S^1_k(\cM)= & \;
3\,N_0+N_0^+ +(k-2) \,N_1^\partial \nonumber
+(k+1)\left( (k-3) \, N_{\Box} 
+ \frac{k-4}{2} \, N_{\Delta}  
\right).
\end{align*}
\end{corollary} 
\begin{proof}
We have $d_1(\cE)\le 1$, $d_2(\cE)\le 2$ on all interior edges $\cE$,
with the equalities whenever two triangles are glued. Hence the bound $k\ge 6$.
Then $\delta(\cM)$ 
counts the interior edges. This count is eliminated using (\ref{eq:polygee}).
\end{proof}
A basis of $S^1_k(\cM)$ can be constructed following the partitioning \refpart{E1}--\refpart{E4}.
The dimension formula may still hold for a few smaller $k$, as demonstrated in \S \ref{sec:exbasis}.
We prove in Lemma \ref{lm:lbound} that for $k\ge 2$ the numerical value in (\ref{eq:dimform})
is a lower bound for the dimension of $S^1_k(\cM)$. The subsequent examples illustrate sharpness
of the degree bounds.

In the case of orientable surfaces, we can use algebraic topology to eliminate
the number of vertices. 
If $\cM$ has a boundary, let $\widetilde{\cM}$
denote the topological surface constructed as follows. For each boundary component $\cB$ of $\cM$,
let $n_\cB$ denote the number of edges on $\cB$. If $n_\cB\ge 3$, we adjoin to $\cM$ a polygon 
with $n_\cB$ sides  and glue its edges consequently to the edges of $\cB$ by (say) linear homeomorphisms.
If $n_\cB=1$ we adjoin a circular disk by gluing its boundary circle to the edge of $\cB$ by
a homeomorphism. If $n_\cB=2$ we adjoin a circular disk, divide its boundary into two half-circles
and glue them consequently to the two edges of $\cB$.
After doing this to all boundary components,  we get a closed surface $\widetilde{\cM}$ without boundary.
If $\cM$ is orientable, so is $\widetilde{\cM}$. In that case, let $g(\cM)$ denote the genus of $\widetilde{\cM}$.
\begin{theorem} \label{th:gdimform}
Suppose that the underlying topological space $\cM$ is 
orientable. 
Let $k$ denote an integer such that $k\ge (2d_1(\cE)+3,d_2(\cE)+4)$ for all interior edges $\cE$. 
Then
\begin{align*}
\dim S^1_k(\cM)= & \; 6-6g\big(\cM\big)+N_0^+ -\delta(\cM) 
-3\,\#\{\mbox{\rm boundary components}\}  \\[1pt]
& +(k-1) \,N_1^\partial
+ \left( k^2-2k+2 \right) N_{\Box} + \frac{(k-1)(k-2)}{2} \, N_{\Delta}.  
\end{align*}
\end{theorem}
\begin{proof}
The Euler characteristic \cite{Wiki} for the 
surface $\widetilde{\cM}$ is
\begin{align}
 2-2g\big(\widetilde{\cM}\big) = & N_0 -N_1+
N_\Box+N_\Delta+\#\{\mbox{\rm boundary components}\}.
\end{align}
Combination with (\ref{eq:polygee}) gives
\begin{align}  \label{eq:topolov}
N_0 = & \,
2-2g\big(\widetilde{\cM}\big)+N_\Box+{\textstyle \frac12}\,N_\Delta
+{\textstyle \frac12}\,N_1^\partial-\#\{\mbox{\rm boundary components}\}.
\end{align}
The claimed formula follows.
\end{proof}
\begin{remark} \rm
For the rational $G^1$ polygonal surfaces made up of triangles only,
a formula equivalent to (\ref{eq:dimforme}) is formulated in \cite[Theorem 6.4.6]{raimundas}.
The proof in \cite{raimundas} has two errors. 
Firstly, the restrictions of Theorem \ref{cond:comp2} are not taken into account. 
The formula is thereby wrongly claimed for the surfaces with bad crossing vertices.
Secondly, the degree bounds in \cite[Lemma 6.4.5]{raimundas}
for strong and complete separation of vertices do not ensure existence of offset splines.
There is a typo $t \leftrightarrow z$ twice on line 8 in the proof of \cite[Lemma 6.4.5]{raimundas},
and the splines $g_3,h_3$, $\tilde{g}_3,\tilde{h}_3$ might thoroughly vanish 
at {\em both} end-vertices.
Theorem 6.4.7 in \cite{raimundas} matches Lemma \ref{lm:lbound} here, but misses the bound $k\ge 2$.
The numbers $d_1,d_2$ in \cite[\S 6.4]{raimundas} equal our $d_1(\cE)-1$, $d_2(\cE)-1$, respectively.
The crossing vertices are called {\em particular vertices} in \cite{raimundas}.
\end{remark}

\subsection{Polynomials on rectangles an triangles}
\label{sec:polyfaces}

To present our examples more efficiently, we recall the classical  Bernstein-B\'ezier bases 
of polynomial functions on rectangles and triangles. Besides, we prove that the formula
in (\ref{eq:dimform}) is a lower bound $\dim S^1_k(\cM)$ when $k\ge 2$.

For any three non-collinear points $X,Y,Z\in\RR^2$, let $u_{XY}^Z$ denote the linear function 
that evaluates to $0$ at $X,Y$, and evaluates to $1$ at $Z$. 
For any polygon with the edges $XY$, $XZ$, the functions $u_{XY}^Z$, $u_{XZ}^Y$ form
standard coordinates attached $XY$ or $XZ$ by Definition \ref{def:stcoor}.
The functions $u_{XY}^Z$, $u_{XZ}^Y$, $u_{YZ}^X$ are the {\em barycentric coordinates} \cite{Wiki}
on the triangle \fc{XYZ}. They are positive inside the triangle, and satisfy 
\begin{equation} \label{eq:baryc}
u_{XY}^Z+u_{XZ}^Y+u_{YZ}^X=1.
\end{equation}
For a rectangle \fc{XYZW}, we have $u_{XY}^Z+u_{ZW}^X=1$ and $u_{XW}^Y+u_{YZ}^X=1$.
We refer to the function pairs $(u_{XY}^Z,u_{ZW}^X)$, $(u_{XW}^Y,u_{YZ}^X)$
as the {\em tensor product coordinates} of the rectangle. 

Let $(u,v,w)$ denote the barycentric coordinates of a triangle.
A polynomial function of degree $\le k$ can be expressed as a homogeneous polynomial
in $u,v,w$ of degree $k$ by homogenizing $1=u+v+w$.
It is customary in CAGD to express the polynomial functions of degree $\le k$
in  the {\em Bernstein-B\'ezier basis}:
\begin{equation} \label{eq:bbb}
\frac{k!}{i!j!(k-i-j)!} \, u^i \, v^j \, w^{k-i-j}, \qquad 0\le i+j \le k.
\end{equation}
The coefficients in a linear expression in this basis are called {\em control coefficients}.
In particular, the control coefficients 
of the constant function 1 are all equal \mbox{to 1.} 
The {\em interior control coefficients} are to the terms with all $i,j,k-i-j$ positive.
The ``corner" terms (with $i\le 1, j\le 1$)  
of a Bernstein-B\'ezier expression
\[
c_0 \, w^k+c_1\,kuw^{k-1}+c_2\,kvw^{k-1}+c_3\,k(k-1)uvw^{k-2}+\ldots
\]
determine the $J^{1,1}$-jet at the vertex $u=0$, $v=0$:
\[
c_0+k(c_1-c_0)u+k(c_2-c_0)v+k(k-1)(c_3-c_2-c_1+c_0)uv+\ldots
\]
 
Let $(u,\tilde{u})$, $(v,\tilde{v})$ denote the tensor product coordinates of a rectangle.
A polynomial in $\RR[u,v]$ of degree $k$ in $u$ and of degree $\ell$ in $v$ has {\em bidegree} $(k,\ell)$.
The polynomial functions of bidegree $(k,\ell)$ can be expressed in
the {\em tensor product Bernstein-B\'ezier basis}
\begin{equation} \label{eq:tpbbb}
{k\choose i} {\ell\choose j}\, u^i \,\tilde{u}^{k-i} \, v^j \, \tilde{v}^{\ell-j}, \qquad
0\le i\le k, \quad 0\le j\le \ell.
\end{equation}
The {\em interior control coefficients} are to the terms with all $i,k-i,j,\ell-j$ positive.
Similarly, the ``corner" terms 
\[
c_0 \, \tilde{u}^k\tilde{v}^\ell+c_1\,ku\tilde{u}^{k-1}\tilde{v}^\ell
+c_2\,\ell v\tilde{u}^k\tilde{v}^{\ell-1}+c_3\,k\ell uv\tilde{u}^{k-1}\tilde{v}^{\ell-1}+\ldots
\]
determine the $J^{1,1}$-jet at the vertex $u=0$, $v=0$:
\[
c_0+k(c_1-c_0)u+\ell(c_2-c_0)v+k\ell(c_3-c_2-c_1+c_0)uv+\ldots
\]
\begin{lemma} \label{lm:lbound} 
If $k\ge 2$, the dimension of $S^1_k(\cM)$ is greater than or equal to 
the value on the right-hand side of  $(\ref{eq:dimform})$.
\end{lemma}
\begin{proof} Suppose first that $k\ge 4$. For each interior edge $\cE=\tau_1\sim\tau_2$, 
the map $W_0\oplus W_1$ of Lemma \ref{lm:separate} restricted to $M^1_k(\tau_1,\tau_2)$
might have smaller image than of the general dimension $10-e_{\perp}(\cE)$.
Let $m_\cE\ge 0$ denote this (possible) dimension deficit. 
The dimension of splines in \refpart{E1},
lifted from the direct sum of all $M(P)$ over the polynomial vertices $P$,
is at least the sum of all expressions (\ref{eq:E1spline}) minus $\sum m_\cE$.
The dimension of splines contributed by $\cE$ in \refpart{D3}, \refpart{E2}
equals (\ref{eq:E2spline}) plus $m_\cE$. Summing up the dimensions as in Lemma \ref{lm:dimsep}
with the adjustments by $\pm m_\cE$ can only underestimate the dimension of the spline space.
The adjustments $\pm m_\cE$ cancel out, like the numbers $e_\perp(\cE)$.
After applying Corollary \ref{cor:m1dim} we get the claim for $k\ge 4$.

If $k=3$, we have an overestimate by $N_\Delta$ in \refpart{E4}. 
For each triangle \fc{XYZ} we have 3 expressions for the unique interior control coefficient
in terms of the $J^{1,1}$-coefficients in $M(X)$, $M(Y)$ or $M(Z)$. 
The $\sum m_\cE$ new relations give 3 equalities relating these 3 expressions,
and only 2 of those equalities are independent. 
Hence the overestimate by $N_\Delta$ is cancelled by the dependencies among
the $\sum m_\cE$ new relations. If $k=2$, we have a similar cancelation of the overestimate
by $N_\Box$, and it remains to cancel $3N_\Delta$. For a quadratic polynomial on a triangle \fc{XYZ},
each of the control coefficients to $uv$, $uw$, $vw$ has 3 expressions in terms of
the $J^{1,1}$-coefficients in $M(X)$, $M(Y)$ or $M(Z)$. The $\sum m_\cE$ relations 
give 9 equalities between those expressions, only 6 of them are independent similarly.
Hence the overestimate by $3N_\Delta$ is cancelled as well.
\end{proof}

\begin{example} \rm
Consider a triangulation of a polygonal region $\widetilde{\sgma}$ in $\RR^2$, and a $G^1$ polygonal surface $\cR$
defined by the parametric continuity on the whole $\widetilde{\sgma}$. (See Remark \ref{rm:parametric}.)
We have thus $\delta(\cE)=0$, $d_1(\cE)=d_2(\cE)=1$ for each interior edge $\cE$.
For $k\ge 5$, Theorem \ref{th:gdimform} gives 
\begin{align}
\dim S^1_k(\cR)= 3+N_0^++(k-1) \,N_1^\partial+\frac{(k-1)(k-2)}{2} \, N_\Delta.
\end{align}
This result was proved in \cite{ms} with an explicit basis congruous to \refpart{E1}--\refpart{E4}. 
The formula holds for $k=4$ as well \cite{alf4-1}. 
Lemma \ref{lm:lbound} for this case is proved in \cite{schumaker79}.
The statement of Lemma \ref{lm:lbound} does not hold if $k=1$ and 
crossing vertices are present, since $\dim S^1_1(\cR)=3$.
See also \cite[Example 6.26]{raimundas}, \cite{bil}. 
\end{example}
\begin{example} \rm  \label{ex:torus}
The torus example in \cite[Example 6.31]{raimundas} glues two triangles using parametric continuity.
The triangles can be glued first to a quadrangle, then to a torus
(by parallel translations or the well-known identification \cite{cltopology} of the opposite edges). 
The dimension formula is $\dim S^1_k=(k-1)(k-2)$ for $k\ge 5$.
But $\dim S^1_4=7$, showing that the bound $d_2+4$ in Theorem \ref{th:dimform}   can be sharp.
\end{example}
\begin{example} \label{ex:platonic} \rm
For the octahedral example in \cite{VidunasAMS} (and Example \ref{eq:octahed} here) 
the dimension formula $\dim S^1_k=4(k-1)(k-2)$ holds for $k\ge 3$. 
Similarly, \cite[Example 6.27]{raimundas} shows that 
a tetrahedral construction with 4 triangles glued by a linear $\delta(\cE)=1$ data
into a topological sphere has $\dim S^1_k=2(k-1)(k-2)$ if \mbox{$k\ge 3$}. 
For a similar Platonic \cite{PetersPlat} cubical construction with 6 rectangles, 
the dimension formula is $\dim S^1_k=6(k-1)^2$ if \mbox{$k\ge 2$}.
In these examples, all interior control coefficients can be chosen freely,
and each choice leads to a unique spline. 
On the other hand, a Platonic construction of icosahedron with $20$ triangles
and all $30$ edges glued linearly has the dimension formula $\dim S^1_k=10k^2-30k-4$ for $k\ge 6$.
\end{example}

\section{The example of a pruned octahedron}
\label{sec:poctah}

This extensive example demonstrates the full work 
and technical details  of building a workable $G^1$ polygonal surface 
and computing a basis of 
splines of bounded degree.
This example can be broadly read without acquittance with the material
between \S \ref{sec:serverts}--\S \ref{sec:generate} and \S \ref{eq:spdim}--\S \ref{sec:dformula}.

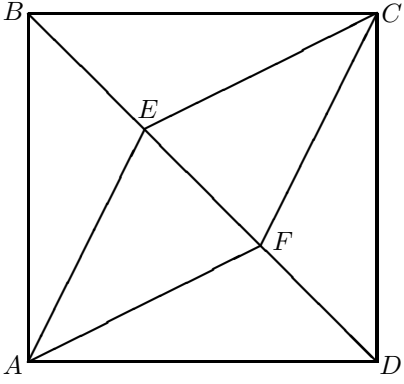
\begin{figure} \setlength{\unitlength}{0.88pt}
\[ \begin{picture}(170,155)(-10,-3)
\thicklines \put(0,150){\line(1,-1){150}} 
\put(0,0){\line(1,0){150}} \put(0,150){\line(1,0){150}}
\put(0,0){\line(0,1){150}} \put(150,0){\line(0,1){150}}
\put(0,0){\line(1,2){50}} \put(0,0){\line(2,1){100}}
\put(100,50){\line(1,2){50}} \put(50,100){\line(2,1){100}}
\put(-11,-5){$A$} \put(152,146){$C$}  \put(-11,147){$B$}  \put(151,-5){$D$}
\put(47,105){$E$} \put(105,48){$F$} 
\end{picture} \]
\caption{The Schlegel diagram of the pruned octahedron}
\label{fig:aocta}
\end{figure}

Figure \ref{fig:aocta} depicts a {Schlegel diagram} \cite{Wiki} 
of a convex polyhedron in $\RR^3$, 
with 6 triangular facets and one rectangle \fc{ABCD} ``at the back" of the picture.
We define a closed $G^1$ polygonal surface $\cH$
from the same six  triangles \fc{AEF}, \fc{CEF}, \fc{ABE}, \fc{BCE}, \fc{ADF}, \fc{CDF}
and the rectangle \fc{ABCD} by keeping the same incidence relations. 
Topologically, $\cH$ is a sphere. If we append an edge \fc{BD} and 
split the rectangle into  two triangles \fc{ABD} and \fc{BCD},
we get the octahedron (combinatorially) as in \cite{VidunasAMS} and Example \ref{eq:octahed}.

For an edge \fc{XY} from Figure \ref{fig:aocta},
let $u_X^Y$ denote the linear function on the edge
with the values $u_X^Y(X)=0$, $u_X^Y(Y)=1$. 
Following the notation of \S \ref{sec:polyfaces}, we have 
\begin{equation}
u_A^B=u_{AE}^B\big|_{AB}=u_{AD}^B\big|_{AB}, 
\end{equation}
etc. We refer to the gluing data along the edge \fc{XY}  by $(\ma_{XY}(u_X^Y),\mb_{XY}(u_X^Y))$.
With the Schlegel diagram in mind, \fc{XY} denotes both interior edge of $\cH$
and the two polygonal edges. Similarly, we have coinciding denotation for each vertex of $\cH$
and the corresponding polygonal vertices. We adjust the notation 
in (\ref{eq:wv})--(\ref{eq:wev}),  (\ref{eq:s1comp}),  (\ref{eq:w0}) accordingly.

\subsection{The gluing data}
\label{sec:gdata}

With the topological and combinatorial structure decided,
we first choose the structure of tangent sectors as in Figure \ref{fig:g1v},
as prescribed in \S \ref{rm:gconstr}. 
We select Hahn's symmetric data (\ref{eq:hsym}) of Example \ref{ex:hahnsym}.
Therefore $E,F,A,C$ 
are crossing vertices 
with these relations between directional derivatives at them:
\begin{align} \label{eq:horder4}
\partial_{EA}+\partial_{EC}=0, \qquad \partial_{EB}+\partial_{EF}=0, \nonumber \\
\partial_{FA}+\partial_{FC}=0, \qquad \partial_{FD}+\partial_{FE}=0, \nonumber \\
\partial_{AB}+\partial_{AF}=0, \qquad \partial_{AE}+\partial_{AD}=0,\\
\partial_{CB}+\partial_{CF}=0, \qquad \partial_{CE}+\partial_{CD}=0. \nonumber
\end{align}
At the vertices of valency 3 we have
\begin{equation} \label{eq:horder3}
\partial_{BA}+\partial_{BC}+\partial_{BE}=0, \qquad 
\partial_{DA}+\partial_{DC}+\partial_{DF}=0.
\end{equation}

Next we choose the glueing data  
along the edges of $\cH$. 
To glue the triangles \fc{EFA} and \fc{EFC} along $EF$,
we interpolate the following relations between the derivatives:
\begin{equation} \label{eq:dinterp1}
\partial_{EA}+\partial_{EC}=0,  \qquad \partial_{EA}+\partial_{EC}=2\,\partial_{EF}.
\end{equation}
The second expression here is actually 
$\partial_{FA}+\partial_{FC}=0$  rewritten following (\ref{eq:reders}):
\begin{equation} \label{eq:dinterp2}
\partial_{FA}=\partial_{EA}-\partial_{EF}, \qquad 
\partial_{FC}=\partial_{EC}-\partial_{EF}.
\end{equation}
We choose the linear interpolation of the tangent relations in (\ref{eq:dinterp1}):
\begin{equation} \label{eq:croscros}
\partial_{EA}+\partial_{EC}=2u_E^F\,\partial_{EF}.
\end{equation}
Thereby we set up $\ma_{EF}(u_E^F)=2u_E^F$, $\mb_{EF}(u_E^F)=-1$,
as in Example \ref{ex:joining}.

The edges $EA$, $EC$, $FA$, $FC$ connect crossing vertices just as $EF$.
We take the linear glueing data
\begin{align}
\partial_{EB}+\partial_{EF}=2u_E^A\,\partial_{EA}, \qquad
\partial_{EB}+\partial_{EF}=2u_E^C\,\partial_{EA},  \\
\partial_{FD}+\partial_{FE}=2u_F^A\,\partial_{FA}, \qquad
\partial_{FD}+\partial_{FE}=2u_F^C\,\partial_{FA}. \nonumber
\end{align}
The new conditions (\ref{eq:ddv1})--(\ref{eq:ddv2}) are satisfied across \fc{AE}, \fc{EC} 
and  \fc{AF}, \fc{FC} with the derivatives $\ma'=2$ and all constant $\mb=-1$. 
As only triangles are glued, the balancing condition of Definition \ref{def:balance} holds, 
just as in the linear Hahn 
gluing of rectangles in \cite{Peters2010}.

Remarkably, the linear glueing data along $AB$, $AD$, $CB$, $CD$ looks the same:
\begin{align} \label{eq:glue34}
\partial_{AE}+\partial_{AD}=2u_A^B\,\partial_{AB}, \qquad
\partial_{AF}+\partial_{AB}=2u_A^D\,\partial_{AD}, \nonumber \\
\partial_{CE}+\partial_{CD}=2u_C^B\,\partial_{CB}, \qquad
\partial_{CF}+\partial_{CB}=2u_C^D\,\partial_{CD},
\end{align}
because $\partial_{BA}+\partial_{BC}+\partial_{BE}=0$ is rewritten to
$\partial_{AE}+\partial_{AD}=2\,\partial_{AB}$ by following (\ref{eq:reders}) 
and (\ref{eq:rectuv}):
\begin{equation}
\partial_{BA}=-\partial_{AB},  \quad \partial_{BC}=\partial_{AD}, \quad
\partial_{BE}=\partial_{AE}-\partial_{AB},
\end{equation}
etc. The linear glueing data satisfies (\ref{eq:ddv1})--(\ref{eq:ddv2}),
but the balancing condition does not apply across $A$ and $C$.

For 
glueing the triangles \fc{EBA} and \fc{EBC} along the edge \fc{EB},
we interpolate these relations between directional derivatives:
\begin{equation}
\partial_{EA}+\partial_{EC}=0,  \qquad \partial_{EA}+\partial_{EC}=3\,\partial_{EB}.
\end{equation}
The latter relation is $\partial_{BA}+\partial_{BC}+\partial_{BE}=0$ rewritten using
\begin{equation}
\partial_{BA}=\partial_{EA}-\partial_{EB}, \quad 
\partial_{BC}=\partial_{EC}-\partial_{EB}, \quad
\partial_{BE}=-\partial_{EB}.
\end{equation}
The balancing condition across \fc{BE}, \fc{EF} 
is not satisfied, though only triangles are involved. 
To satisfy (\ref{eq:ddv1})--(\ref{eq:ddv2}) we take the quadratic interpolation 
\begin{equation} \label{eq:crossing3}
\partial_{EA}+\partial_{EC}=\left( 2u_E^B+(u_E^B)^2\right)\,\partial_{EB}.
\end{equation}
Similarly, for glueing along \fc{FD} we choose
\begin{equation} \label{eq:crossing3a}
\partial_{FA}+\partial_{FC}=\left( 2u_F^D+(u_F^D)^2\right)\,\partial_{FD}.
\end{equation}
Alternative glueing data for the edges \fc{EB}, \fc{FD} is briefly discussed in \S \ref{sec:abe}.

We constructed full $G^1$ gluing data on $\cH$ without reference to coordinate systems,
only using directional derivatives and the functions $u_X^Y$ on edges.
If barycentric and tensor product coordinates are locally used,
and polynomial functions are presented in the Bernstein-B\'ezier bases 
(\ref{eq:bbb}), (\ref{eq:tpbbb}), then the directional derivatives are expressed as
\begin{equation}
\partial_{AB}=\frac{\partial}{\partial u_{AE}^B}-\frac{\partial}{\partial u_{BE}^A}
=\frac{\partial}{\partial u_{AD}^B}-\frac{\partial}{\partial u_{BC}^A}, \qquad \mbox{etc.}
\end{equation}

\subsection{The splines around $EF$, $EA$, $EC$, $F\!A$, $FC$}
\label{sec:ef}

We seek a spline basis for $S^1_k(\cH)$ with $k=4$ or slightly larger.
The degree convention is given in \S \ref{sec:boxdelta}.
We start by constructing the $M^1$-spaces of \S \ref{eq:edgespace}
for the 5 edges connecting the crossing vertices. 
The formulas are stated for the edge $EF$.
To have them for the edges $EA$, $EC$, $F\!A$, $FC$, replace the labels
\begin{align*}
(E,F,A,C)\mapsto (E,A,B,F), (E,C,B,F), (F,A,E,D) \mbox{ or } (F,C,E,D).
\end{align*}
Let $M^1(EF)$ denote the corresponding space $M^1(\tau_1,\tau_2)$ in (\ref{eq:s1comp}). 
It consists of the polynomial pairs (\ref{eq:splines2}) 
with $u=u_{EA}^F$ or $u=u_{EC}^F$, 
$v_1=u_{EF}^{A}$, $v_2=u_{EF}^{C}$, 
and such that $(h_1(u),h'_0(u),h_2(u))$ is a syzygy between $(1,-2u,1)$. 
The $\RR[u]$-module of syzygies is freely generated by $(1,0,-1)$, $(2u,1,0)$. 
The space of syzygies of degree $\le k$ has dimension $2k+1$.
It corresponds to the splines in $M^1(EF)$ of degree $\le k+1$.
The dimension of this spline space is $2k+2$, as we include the constant splines.
Therefore $\dim M^1_k(EF)=2k$.

We keep the notation $M(E)$, $M(F)$ of (\ref{eq:wv}) for the spaces of $J^{1,1}$-jets 
of polynomial functions on \fc{EFA}, and let $M'(E)$, $M'(F)$ denote the similar spaces
for polynomials on \fc{EFC}. Since $E$, $F$ are crossing vertices,
the spline jets in $M'(E)$ are determined by $M(E)$ by Lemma \ref{prop:Hg},
and the jets in $M'(F)$ are determined by $M(F)$.
Consider the $W_0$ maps in (\ref{eq:w0}), denoting them
\[
W_{\!EF}:M^1(EF)\to M(E)\oplus M'(E),\quad
W_{\!F\!E}:M^1(EF)\to M(F)\oplus M'(F).
\]
We use the variables $u,v$ for the jet spaces, with $v$ identified with $u_{EF}^{A}$ or $u_{EF}^{C}$.
The explicit action of $W_{\!EF}$ on the spline defined by $(h_0(u),h_1(u),h_2(u))$ in (\ref{eq:splines2}) is then
\begin{align} \label{eq:wef}
\big( & h_0(0)+h'_0(0)\,u+h_1(0)\,v+h'_1(0)\,uv,  \\
& h_0(0)+h'_0(0)\,u-h_1(0)\,v+(-fh_1(0)+2h_0'(0))\,uv\big). \nonumber
\end{align}
Following the derivative transformation (\ref{eq:dinterp2}), 
$W_{\!F\!E}$ sends the same spline to
\begin{align} \label{eq:wef2}
\big( & h_0(1)-h'_0(1)\,u+(h_1(1)-h'_0(1))\,v+(-h'_1(1)+h''_0(1))\,uv, \\
& h_0(1)-h'_0(1)\,u+(h'_0(1)-h_1(1))\,v+(h'_1(1)-h''_0(1)-2h'_0(1))\,uv\big). 
\quad \nonumber
\end{align}
The dual variable transformation for the endpoint symmetry 
is more cumbersome than jet relations: we should transform $u\mapsto 1-u-v$ and then reduce mod $v^2$. 

The map $W_{\!EF}\oplus W_{\!F\!E}$ as in Lemma \ref{lm:separate} has the image of dimension $8$.
Since \mbox{$\dim M^1_4(EF)=8$}, 
there is a chance that the splines in $M_4^1(EF)$ 
map isomorphically to this image or to \mbox{$M(E)\oplus M(F)$}. This turns out to be the case.
Here are the splines of degree 4 that have exactly one non-zero coefficient 
in $M(E)\oplus M(F)$. They are presented in terms of the polynomial triples 
$(h_0(u),h_1(u),h_2(u))$ in (\ref{eq:splines2}), with the common factor brought forward:
\begin{align} \label{eq:us}
U_1=& \,(1-u)\cdot \left( (1-u)(1+2u), -6u^2, -6u^2 \right), 
&  U_5=& \, u^2 \cdot \big( 3-2u,6(1-u), 6(1-u) \big), \qquad \nonumber \\
U_2=& \, u \, (1-u)\cdot \big( 1-u,-2u, 2(1-2u) \big), 
&  U_6=& \, u^2 \cdot \left( 1-u, 1-2u, 3-4u \right),\nonumber  \\
U_3=& \, (1-u)^2\,(1+2u) \cdot \left( 0,1, -1 \right),
& U_7=& \, u^2\,(3-2u) \cdot \left( 0, 1, -1 \right),  \\ 
U_4=& \, u \, (1-u)^2 \cdot \left( 0,1, -1 \right),
&  U_8=& \, u^2\,(1-u) \cdot \left( 0, 1, -1 \right). \nonumber 
\end{align}
The splines $U_1,U_5$ evaluate to $1$ at $E$ or $F$ (respectively),
and their all relevant derivatives $\partial/\partial u$, $\partial/\partial v$,
$\partial^2/\partial u\partial v$ at both $E$, $F$ are zero.
Similarly, the other splines evaluate exactly one of those derivatives at $E$ or $F$ 
(of the restriction to \fc{EFA}) to $1$.
The polynomials $h_0,h_1,h_2$ have degree $\le 3$, 
but $h_1,h_2$ are multiplied by $u_{EF}^A$ or $u_{EF}^C$
giving the degree $4$.

Note that the splines $U_2,U_6$ do not reflect the symmetry between the triangles \fc{EFA} and \fc{EFC}.
By the construction, they have a derivative $\partial/\partial u$ along $EF$ equal to 1,
and their $uv$ terms (representing a mixed derivative $\partial^2/\partial u\partial v$) 
in $M(E)$ or $M(F)$ are zero.
But their $uv$ terms in $M'(E)$ or $M'(F)$ are non-zero, 
reflecting the specific relation of $J^{1,1}$-jets around crossing vertices.
The splines $U_2+U_4$ and $U_6+U_8$ have symmetric specializations to \fc{EFA} and \fc{EFC},
but their $uv$ term in $M(E)$ or $M(F)$ is non-zero as well. 

\begin{figure}
\begin{align*}
\begin{picture}(100,24)(-4,-3)
\put(2,0){\line(1,0){96}}  \put(2,0){\vector(1,1){16}} \put(2,0){\vector(1,-1){16}}
\put(98,0){\vector(-1,1){16}} \put(98,0){\vector(-1,-1){16}}
\put(100,-3){$F$} \put(21,13){$C$}   \put(72,13){$C$} 
\put(-8,-3){$E$} \put(20,-19){$A$}  \put(72,-19){$A$}  
\put(143,22){$U_1:$} \put(282,22){$U_2:$} 
\put(2,-45){$U_3:$} \put(143,-45){$U_4:$} \put(282,-45){$U_5:$} 
\put(2,-113){$U_6:$} \put(143,-113){$U_7:$} \put(282,-113){$U_8:$} 
\end{picture}  
\hspace{54pt} &  \;
\begin{array}{c@{\;}c@{\;}c@{\;}c@{\;}c@{\;}c@{\;}c@{\;}c@{\;}c}
&  1 && 1 && 0 &&  0 \\
\bf 1 && \bf 1 && \bf \frac12 && \bf 0 && \bf 0 \\
& 1 && 1 && 0 &&  0 
\end{array} 
&  \hspace{32pt}
\begin{array}{c@{\;}c@{\;}c@{\,}c@{\,}c@{\;}c@{\;}c@{\;}c@{\;}c}
& 0 && \frac{5}{12} && 0 && 0 \\
\bf 0 && \bf \frac14 && \bf \frac16 && \bf 0 && \bf 0 \\
& 0 && \frac14 && 0 && 0 
\end{array} \\[28pt]
\begin{array}{c@{}c@{}c@{}c@{\,}c@{\;}c@{\;}c@{\;}c@{\;}c}
& -\frac14 && -\frac14 && 0 && 0 \\
\bf 0 && \bf 0 && \bf 0 && \bf 0 && \bf 0 \\
& \frac14 && \frac14 && 0 && 0 
\end{array} 
\hspace{54pt} & \,
\begin{array}{c@{\;}c@{\;}c@{}c@{}c@{\;}c@{\;}c@{\;}c@{\;}c}
& 0 && \!-\frac{1}{12} && 0 &&  0 \\
\bf 0 && \bf 0 && \bf 0 && \bf 0 && \bf 0 \\
&  0 && \frac1{12} && 0 &&  0 
\end{array}
&  \hspace{32pt}
\begin{array}{c@{\;}c@{\;}c@{\;}c@{\;}c@{\;}c@{\;}c@{\;}c@{\;}c}
& 0 && 0 && 1 && 1 \\
\bf 0 && \bf 0 && \bf \frac12 && \bf 1 && \bf 1 \\
& 0 && 0 && 1 && 1 
\end{array} \; \\[28pt]
\begin{array}{c@{\;}c@{\;}c@{\;}c@{\;}c@{\,}c@{\,}c@{\;}c@{\;}c}
& 0 && 0 && \frac5{12} && 0 \\
\bf 0 && \bf 0 && \bf \frac16 && \bf \frac14 && \bf 0 \\
& 0 && 0 && \frac14 && 0 
\end{array} 
\hspace{54pt} &  
\begin{array}{c@{\;}c@{\;}c@{\;}c@{\;}c@{}c@{}c@{}c@{\,}c}
& 0 && 0 && -\frac14 && -\frac14 \\
\bf 0 && \bf 0 && \bf 0 && \bf 0 && \bf 0 \\
&  0 && 0 && \frac14 &&  \frac14 
\end{array}
&  \hspace{32pt}
\begin{array}{c@{\;}c@{\;}c@{\;}c@{\;}c@{}c@{}c@{\;}c@{\;}c}
& 0 && 0 && \!-\frac{1}{12} && 0 \\
\bf 0 && \bf 0 && \bf 0 && \bf 0 && \bf 0 \\
& 0 && 0 && \frac{1}{12} && 0 
\end{array} 
\end{align*}
\caption{Splines around the edge $EF$} 
\label{fig:splinef}
\end{figure}
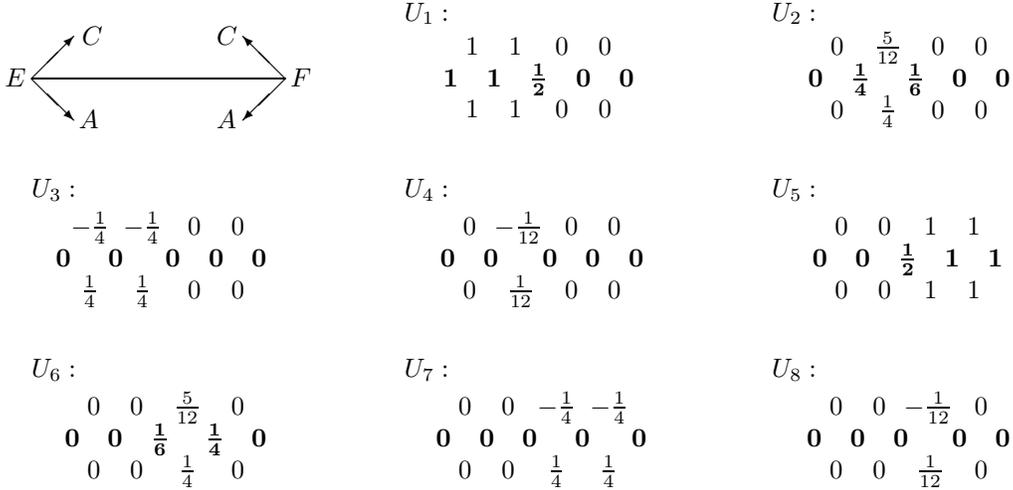

Figure \ref{fig:splinef} represents the 8 splines in $M^1(EF)$ in the Bernstein-B\'ezier basis.
The middle row (in the bold font) of each array displays the control coefficients of the restriction to 
the common edge. The restriction to \fc{EFA} is defined by the two bottom rows, while the restriction
\fc{EFC} is defined by the two upper rows. The transition to the Bernstein-B\'ezier basis involves
homogenization with (\ref{eq:baryc}) 
and ignoring the 
terms divisible by $(u_{EF}^A)^2$ or $(u_{EF}^C)^2$. 

\subsection{The splines around $AB$, $AD$, $CB$, $CD$}
\label{sec:ab}

The gluing data (\ref{eq:glue34}) on \fc{AB}, \fc{AD}, \fc{CB}, \fc{CD} 
leads to the same syzygy module $\cZ(1,-2u,1)$ as in \S \ref{sec:ef}, 
generated by the same syzygies  $(1,0,-1)$, $(2u,1,0)$.
However, these 4 edges join a triangle with a rectangle. 
The splines of degree $\le k$ correspond to the syzygies $(A,B,C)$ with
$\deg A\le k-1$, $\deg B\le k-1$, $\deg C\le k$. 
For example, the syzygy $(0,u^3,2u^4)$ gives a spline of degree $4$.
Besides, the vertices $B$, $D$ are not crossing vertices,
so there are 5 degrees of freedom around them by Lemma \ref{lm:edgew}.

Like in the previous section, we give explicit formulas for one edge \fc{AB}.
For the other edges, replace the labels
\begin{align*}
(A,B,E,D)\mapsto (A,D,F,B), \ (C,B,E,D) \mbox{ or } (C,D,F,B).
\end{align*}
Similarly to \S \ref{sec:ef}, let $M^1(AB)$ denote the corresponding space 
in (\ref{eq:s1comp}). The splines in $M^1(AB)$ correspond to the polynomial pairs 
(\ref{eq:splines2})  with $u=u_{AE}^B$ or $u=u_{AD}^B$, 
$v_1=u_{AB}^{E}$, $v_2=u_{AB}^{D}$, 
such that  $(h_1(u),h'_0(u),h_2(u))\in \cZ(1,-2u,1)$. 
The spaces $M^1(AB)$ and $M^1(EF)$ coincide 
in terms of the polynomial triples $(h_0(u),h_1(u),h_2(u))$,
but the degree grading is different.

\begin{figure}
\begin{align*} \\[-12pt]
\begin{picture}(108,25)(-7,-3)
\put(4,0){\line(1,0){92}}  \put(4,0){\vector(0,1){22}} \put(4,0){\vector(1,-1){16}}
\put(96,0){\vector(0,1){22}} \put(96,0){\vector(-1,-1){16}}
\put(98,-3){$B$} \put(8,17){$D$}   \put(85,17){$C$} 
\put(-7,-3){$A$} \put(22,-19){$E$}  \put(70,-19){$E$}  
\put(149,23){$\widetilde{U}_0:$} \put(298,23){$\widetilde{U}_1:$} 
\put(0,-48){$\widetilde{U}_2:$} \put(143,-48){$\widetilde{U}_3:$} 
\put(292,-48){$\widetilde{U}_4:$} 
\put(-12,-118){$\widetilde{U}_5:$} \put(87,-118){$\widetilde{U}_6:$} 
\put(198,-118){$\widetilde{U}_7:$} \put(312,-118){$\widetilde{U}_8:$} 
\end{picture} 
 \hspace{49pt} & \hspace{6pt}
\begin{array}{c@{\;}c@{\;}c@{\;}c@{}c@{}c@{}c@{\;}c@{\;}c}
0 && 0 &&  -\frac{5}{24} && \frac1{16} && 0 \\[2pt]
\bf 0 && \bf 0 && \bf \!-\frac{1}{12} && \bf 0 && \bf 0 \\
& 0 && 0 && 0 && 0 
\end{array}  &  \hspace{20pt}
\begin{array}{c@{\;}c@{\;}c@{\;}c@{}c@{\,}c@{\;}c@{\;}c@{\;}c}
1 && 1 && -1 &&  0 && 0\\
\bf 1 && \bf 1 && \bf 0 && \bf 0 && \bf 0 \\
& 1 && 1 && 0 &&  0 
\end{array} \hspace{6pt} \\[28pt]
\begin{array}{c@{\;}c@{\;}c@{\,}c@{}c@{\,}c@{\;}c@{\;}c@{\;}c}
0 && \frac{3}{8} && -\frac14 && 0  && 0 \\[2pt]
\bf 0 && \bf \frac14 && \bf 0 && \bf 0 && \bf 0 \\
& 0 && \frac14 && 0 && 0 
\end{array}
 \hspace{58pt} & 
\begin{array}{c@{\,}c@{}c@{\,}c@{}c@{\,}c@{\;}c@{\;}c@{\;}c}
\!-\frac14 && -\frac14 && -\frac18 && 0 && 0 \\
\bf 0 && \bf 0 && \bf 0 && \bf 0 && \bf 0 \\
& \frac14 && \frac14 && 0 && 0 
\end{array} & \hspace{20pt}
\begin{array}{c@{\;}c@{}c@{}c@{}c@{\,}c@{\;}c@{\;}c@{\;}c}
0 && \!-\frac{1}{16} && \!\!-\frac{1}{24}  &&  0 &&  0 \\
\bf 0 && \bf 0 && \bf 0 && \bf 0 && \bf 0 \\
&  0 && \frac1{12} && 0 &&  0 
\end{array} \\[29pt] \hspace{-4pt}
\begin{array}{c@{\;}c@{\;}c@{\;}c@{\;}c@{\;}c@{\;}c@{\;}c@{\;}c}
0 && 0 && 2 && 1 && 1 \\
\bf 0 && \bf 0 && \bf 1 && \bf 1 && \bf 1 \\
& 0 && 0 && 1 && 1 
\end{array} \hspace{80pt} & \hspace{-59pt}
\begin{array}{c@{\;}c@{\;}c@{\;}c@{\,}c@{}c@{\;}c@{\;}c@{}c}
0 && 0 && \frac{37}{24} && 0 && \!-\frac14 \\[2pt]
\bf 0 && \bf 0 && \bf \frac23 && \bf \frac14 && \bf 0 \\
& 0 && 0 && \frac14 && 0 
\end{array}  \hspace{23pt}
\begin{array}{c@{\;}c@{\;}c@{\;}c@{}c@{\,}c@{}c@{\,}c@{}c}
0 && 0 && \!-\frac18 && \!-\frac14 && \!-\frac14 \\
\bf 0 && \bf 0 && \bf 0 && \bf 0 && \bf 0 \\
&  0 && 0 && \frac14 &&  \frac14 
\end{array} \hspace{-28pt}
& 
\begin{array}{c@{\;}c@{\;}c@{\;}c@{}c@{}c@{}c@{\;}c@{\;}c}
0 && 0 &&  -\frac{1}{4} && 0 && 0 \\[2pt]
\bf 0 && \bf 0 && \bf \!-\frac{1}{12}\! && \bf 0 && \bf 0 \\
& 0 && 0 && \frac{1}{12} && 0 
\end{array} \hspace{-10pt}
\end{align*}
\caption{Splines around the edge $AB$} 
\label{fig:splinab}
\end{figure}

We keep the notation $M(A)$, $M(B)$ for the spaces of $J^{1,1}$-jets at the triangle \fc{ABE} corners,
and let $M'(A)$, $M'(B)$ denote the spaces of $J^{1,1}$-jets at the rectangle corners.
Following (\ref{eq:w0}), the $W_0$-map $W_{\!AB}:M^1(AB)\to M(A)\oplus M'(A)$ acts as in (\ref{eq:wef}),
but $W_{\!BA}:M^1(AB)\to M(B)\oplus M'(B)$ acts by
\begin{align}
\big( & h_0(1)-h'_0(1)\,u+(h_1(1)-h'_0(1))\,v+(-h'_1(1)+h''_0(1))\,uv,   \quad \nonumber \\
& h_0(1)-h'_0(1)\,u+(2h'_0(1)-h_1(1))\,v-h'_2(1)\,uv\big). 
\end{align}
The dimension of $\img W_{\!BA}$ equals 5, with $h'_2(1)$ as an extra degree of freedom.
Note that the coefficients to the terms $-h'_0(1)u$, $(h_1(1)-h'_0(1))v$, \mbox{$(2h'_0(1)-h_1(1))v$} 
sum up to zero, reflecting (\ref{eq:horder3}).

The splines in $M_k^1(AB)$ are generated by the polynomial multiples of degree $\le k-1$
of the syzygies $(1,0,-1)$, $(0,1,2u)$. Together with the constant splines, $\dim M_k^1(AB)=2k+1$.
In particular, $\dim M_4^1(AB)=9$ just as the combined dimension of $\img W_{AB}$
and $\img W_{BA}$. The expressions in (\ref{eq:us}) define splines in 
$M_4^1(AB)$ just as in $M_4^1(EF)$. Additionally, we have the spline
\begin{align} \label{eq:f9} 
\widetilde{U}_0 
=  u^2(1-u) \cdot \left( -{\textstyle\frac12\,}(1-u),1,-3+4u  \right)
\end{align} 
in terms of (\ref{eq:splines2}).
This spline in $M^1(EF)$ has degree $5$ and is annihilated by $W_{\!EF}$, $W_{\!F\!E}$,
but $W_{\!B\!A\,}\widetilde{U}_0=(0,uv)$. The splines in (\ref{eq:us}) have to be adjusted 
by $\widetilde{U}_0$ to have $f'_2(1)=0$. Here is the ``diagonal" spline basis for $\dim M_4^1(AB)$:
\begin{align} \label{eq:sbasis2}
&\widetilde{U}_0, \quad  \widetilde{U}_1=U_1+6\widetilde{U}_0, 
\quad \widetilde{U}_2=U_2+2\widetilde{U}_0, 
\quad \widetilde{U}_3=U_3, \quad \widetilde{U}_4=U_4,  \\
& \widetilde{U}_5=U_5-6\widetilde{U}_0, \quad \widetilde{U}_6=U_6-U_8-6\widetilde{U}_0, 
\quad \widetilde{U}_7=U_7,\quad \widetilde{U}_8=U_8+\widetilde{U}_0. \qquad \nonumber
\end{align}
Figure \ref{fig:splinab} presents these splines in $M^1(AB)$ as arrays of 
relevant Bernstein-B\'ezier coefficients.

\begin{figure}
\begin{align*} \\[-10pt]
\begin{picture}(113,20)(-7,-3)
\put(2,0){\line(1,0){94}}  \put(2,0){\vector(1,1){16}} \put(2,0){\vector(1,-1){16}}
\put(96,0){\vector(-1,1){16}} \put(96,0){\vector(-1,-1){16}}
\put(95,-9){$B$} \put(21,13){$C$}   \put(70,13){$C$} 
\put(-7,-3){$E$} \put(20,-19){$A$}  \put(70,-19){$A$}  
\put(130,21){$V_1:$}  \put(273,21){$V_2:$} 
\end{picture}   &  \hspace{27pt}
\begin{array}{c@{\;}c@{\;}c@{\;}c@{\;}c@{\;}c@{\;}c@{}c@{\,}c@{\;}c@{\;}c@{\;}c@{\;}c}
&  1 && 1 && \frac15 && \!-\frac25 && 0 && 0 \\
\bf 1 && \bf 1 && \bf \frac35 && \bf  \frac15 && \bf 0 && \bf 0 && \bf 0 \\
& 1 && 1 && \frac35 &&  \frac15 && 0 && 0 
\end{array} & \hspace{14pt}
\begin{array}{c@{\;}c@{\;}c@{\,}c@{}c@{}c@{}c@{}c@{}c@{\;}c@{\;}c@{\;}c@{\;}c}
& 0 && \frac{7}{30} && \frac1{20} && \!-\!\frac2{15} && 0 && 0 \\
\bf 0 && \bf \frac16 && \bf \frac2{15} && \bf \frac1{20} && \bf 0 && \bf 0 && \bf 0 \\
& 0 && \frac16 && \frac1{6} && \frac1{12} && 0 && 0
\end{array} \\[28pt]
\begin{picture}(113,20)(-7,-3)
\put(2,0){\line(1,0){94}}  \put(2,0){\vector(1,1){16}} \put(2,0){\vector(1,-1){16}}
\put(96,0){\vector(-1,1){16}} \put(96,0){\vector(-1,-1){16}}
\put(95,-9){$D$} \put(21,13){$C$}   \put(70,13){$C$} 
\put(-7,-2){$F$} \put(20,-19){$A$}  \put(70,-19){$A$}  
\put(130,21){$V_5:$}  \put(273,21){$V_6:$} 
\put(36,-46){$V_8:$}  \put(215,-46){$V_9:$} 
\end{picture}   &  \hspace{29pt}
\begin{array}{c@{\;}c@{\;}c@{\;}c@{\;}c@{\;}c@{\;}c@{\;}c@{\;}c@{\;}c@{\;}c@{\;}c@{\;}c}
& 0 && 0 && \frac45 && \frac75 && 1 && 1 \\
\bf 0 && \bf 0 && \bf \frac25 && \bf \frac45 && \bf 1  && \bf 1 && \bf 1 \\
& 0 && 0 && \frac25 && \frac45 && 1 && 1 
\end{array} &  \hspace{14pt}
\begin{array}{c@{\;}c@{\;}c@{\;}c@{\;}c@{\,}c@{}c@{}c@{\,}c@{\,}c@{\;}c@{}c@{\,}c}
& 0 && 0 && \frac{11}{12} && \,\frac{6}{5} && 0 && \!-\frac16 \\
\bf 0 && \bf 0 && \bf \frac13 && \bf \frac{11}{20} && \bf \frac{7}{15} && \bf \frac16 && \bf 0 \\
& 0 && 0 && \frac1{12} && \,\frac{1}{6} &&\frac16 && 0 
\end{array}  \\[28pt]
\begin{array}{c@{\;}c@{\;}c@{\;}c@{}c@{}c@{}c@{}c@{}c@{\,}c@{\,}c@{\;}c@{\;}c}
& 0 && 0 && \!\!-\frac{1}{10} && \!\!-\frac{2}{15} && 0 && 0 \\
\bf 0 && \bf 0 && \bf \!-\frac{1}{30}\! && \bf \!\!-\frac{1}{20}\! && \bf \!\!-\frac{1}{30} && \bf 0 && \bf 0 \\
& 0 && 0 && 0 && \frac{1}{60} && \frac{1}{30} && 0 
\end{array} 
\hspace{-86pt}  & \hspace{115pt}
\begin{array}{c@{\;}c@{\;}c@{\;}c@{}c@{}c@{}c@{}c@{}c@{}c@{\,}c@{\;}c@{\;}c}
& 0 && 0 && \!\!-\frac{1}{12} && \!\!-\frac{1}{10} && \frac{1}{30} && 0 \\
\bf 0 && \bf 0 && \bf \!-\frac{1}{30} && \bf \!\!-\frac{1}{20} && \bf \!\!-\frac{1}{30} && \bf 0 && \bf 0 \\
& 0 && 0 && \!\!-\frac{1}{60} && \!\!-\frac{1}{60} && 0 && 0 
\end{array}  \hspace{-197pt}
\end{align*}
\caption{Splines around the edges \fc{EB}, \fc{FD}} 
\label{fig:splinfd}
\end{figure}
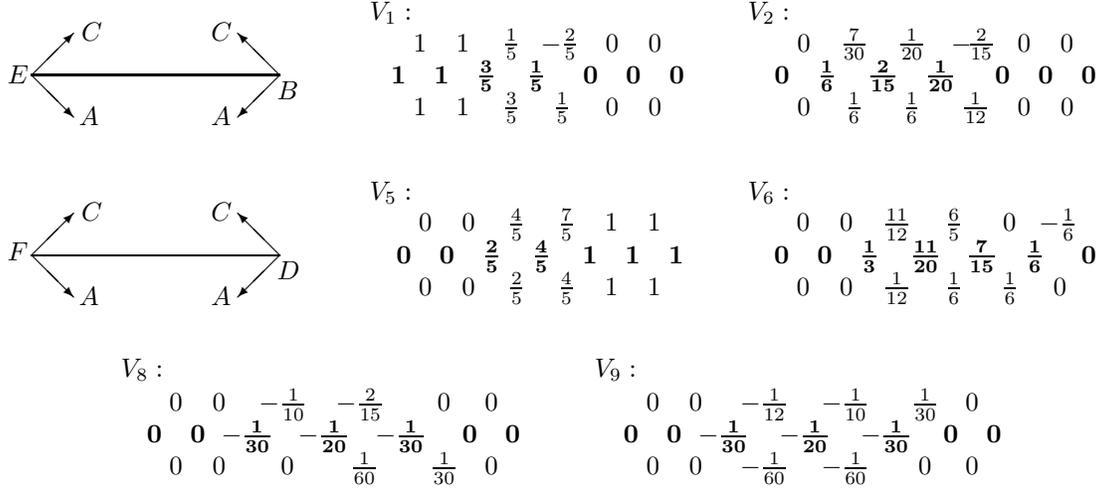

\subsection{The splines around $EB$ and $FD$}
\label{sec:be}

The gluing data (\ref{eq:crossing3}) on the edges \fc{EB},  \fc{FD}
leads to the syzygy module \mbox{$\cZ(1,-2u-u^2,1)$}.
The spline spaces $M^1(EB)$, $M^1(FD)$ are defined as in (\ref{eq:mzz}).
The vertex evaluation maps $W_{EB}\oplus W_{BE}$ and $W_{FD}\oplus W_{DF}$ 
are defined as in (\ref{eq:wef}), (\ref{eq:wef2}) in terms of (\ref{eq:splines2}). 
The dimension of their images equals $9$.

The syzygy module is generated by $(1,0,-1)$, $(2u+u^2,1,0)$. 
The dimension of splines of degree $\le 4$ splines is 7, 
not enough to cover the 9 degrees of freedom. 
In particular, there are these relations for the $J^{1,1}$-jets of $M_4^1(EB)$ or $M_4^1(FD)$: 
\begin{align} \label{eq:quarel}
2h_0(0)+h_0'(0) & = \, 2h_0(1)-h_0'(1), \\  \label{eq:quarel2}
2h_0'(0)-6h_0'(1) & = \, (h_0''(1)-h_1'(1))+(h_0''(1)-h_2'(1)).
\end{align}
The first relation is a restriction on the edge control coefficients.
The dimension of degree $\le 5$ splines is 9, but the spline 
\begin{equation} \label{eq:trivzero}
(h_0(u),h_1(u),h_2(u)) = u^2(1-u)^2\cdot (0,1,-1).
\end{equation}
thoroughly vanishes at the end-vertices.
There is still a restriction on the jets of $M_5^1(EB)$ or $M_5^1(FD)$,
which is actually $6\times(\ref{eq:quarel})$ minus $(\ref{eq:quarel2})$.

Splines of degree 6 are needed to completely separate the vertices.
Among 9 splines of degree 6 evaluating the 9 degrees of freedom on the triangles \fc{AEB},  \fc{AFD} 
``diagonally" like in (\ref{eq:us}), three splines can be of actual degree 4
and have the same Bernstein-B\'ezier representations as $U_3$, $U_4$, $U_7$
in Figure \ref{fig:splinef}. Let us call these splines $V_3$, $V_4$, $V_7$, respectively.
As an example, here is the representation of $V_3$ as an array 
of degree 6 Bernstein-B\'ezier coefficients:
\begin{equation}
\begin{array}{c@{}c@{\,}c@{}c@{\,}c@{}c@{}c@{}c@{}c@{\;}c@{\;}c@{\;}c@{\;}c}
& -\frac16 && -\frac16 && \!-\frac7{60} && \!-\frac1{20} && 0 && 0 \\
\bf 0 && \bf 0 && \bf 0 && \bf 0 && \bf 0  && \bf 0 && \bf 0 \\
& \frac16 && \frac16 && \frac7{60} && \frac1{20} && 0 && 0 
\end{array}
\end{equation}
The middle coefficients $\pm 7/60$, $\pm 1/20$ 
can be nullified after combination with (\ref{eq:trivzero})  
and other thoroughly vanishing \mbox{$u^3(1-u)^2\cdot (0,1,-1)$}. 
Further, 6 splines evaluating the other 6 degrees of freedom
individually to $1$ are presented in Figure \ref{fig:splinfd},
as the similar arrays of Bernstein-B\'ezier coefficients of degree 6.
They can be modified by the same two polynomial multiples of $(0,1,-1)$. 
The presented splines have the property that the restriction to \fc{AEB} or \fc{AFD}
has degree $4$. 



\subsection{Splines around vertices}
\label{sec:exvertexsp}

From the $M^1$-splines presented in Figures \ref{fig:splinef}, \ref{fig:splinab}, \ref{fig:splinfd}
we can build the splines on $\cH$ characterized in \refpart{E1}.
The first step is to determine the local degrees of freedom 
in the (4 or 3) associated polygonal vertex spaces $M(P)$ in (\ref{eq:wv}) around each vertex of $\cH$.

The degrees of freedom around $A$, $C$ are presented in Figure \ref{fig:freeda},
each in the form of 4 blended $2\times 2$ arrays of corner control coefficients 
(in the Bernstein-B\'ezier expressions) of degree 4 polynomial functions
on the incident polygons. The arrays \refpart{b} and \refpart{c} have a derivative along an edge equal to $1$,
but their standard mixed derivatives are non-zero as well. These two arrays are adjusted by \refpart{d} for the maximal
symmetry on the triangles. 

These degrees of freedom are lifted to splines on $\cH$
in Figure \ref{fig:splina}, presented by blended arrays of control coefficients on the incident polygons.
The missing spline \refpart{c} is a mirror image of the spline \refpart{b}.
The support of these splines consists of the polygons incident to $A$ or $C$,
that is, three triangles and the quadrangle.
Only the control coefficients for these polygons are shown.
The spline \refpart{a} is recognizably built up from copies of the $M^1$-splines $U_1$, $\widetilde{U}_1$.
Similarly, the spline \refpart{b} is build up from $U_2+U_4$, $U_3+U_4$,
$-\widetilde{U}_2-\widetilde{U}_4$, $\widetilde{U}_3+\widetilde{U}_4$,
and the spline \refpart{d} is build up from copies of $U_4$, $\widetilde{U}_4$.
In spline \refpart{d}, we can modify the $0$ entry next to the two $\frac1{24}$ entries to $\frac1{36}$,
so to lower the degree of the specialization to the rectangle.

\begin{figure}
\begin{align*}
\begin{picture}(64,24)(-30,-4)
\put(-32,0){\vector(1,0){64}}  \put(-32,0){\vector(-1,0){0}}  
\put(0,-32){\vector(0,1){64}}  \put(0,-32){\vector(0,-1){0}}  
\put(25,4){$F$} \put(3,26){$E$}  
\put(-32,4){$B$} \put(3,-32){$D$}  
\put(99,26){\refpart{a}} \put(255,26){\refpart{b}}  
\put(25,-56){\refpart{c}} \put(174,-56){\refpart{d}}  
\end{picture}  && \hspace{-32pt}
\begin{array}{ccccc}
&& \vdots \\[1pt]
& 1 & \bf 1 & 1 & \\[1pt]
\cdots & \bf 1 & \bf 1 & \bf 1 & \cdots \\[1pt]
& 1 & \bf 1 & 1 & \\[-2pt]
&& \vdots \\
\end{array}
& & \hspace{-32pt}
\begin{array}{ccccc}
&& \vdots \\[1pt]
& \frac13 & \bf \frac14 & \frac13 &  \\[2pt]
\cdots  & \bf 0 & \bf 0 & \bf 0 & \cdots  \\[2pt]
& \!\!\!-\frac5{16}\!\! & \bf \!\!-\frac14\!\! & \!\!-\frac13\!\! & \\[0pt]
&& \vdots \\
\end{array}  \\[2pt]
& \hspace{-12pt} \begin{array}{ccccc}
&& \vdots \\
& \!\!-\frac13\!\! & \bf 0 & \frac13 & \\[2pt]
\cdots\! & \bf \!\!-\frac14\!\! & \bf 0 & \bf \frac14 & \cdots \\[2pt]
& \!\!\!-\frac5{16}\!\! & \bf 0 & \frac13 & \\
&& \vdots \\
\end{array}
& & \hspace{-18pt}
\begin{array}{ccccc}
&& \vdots \\
&  \!\!\!-\frac1{12}\!\!  & \bf 0 & \frac1{12} &  \\[1pt]
\cdots  & \bf 0 & \bf 0 & \bf 0 & \cdots  \\[1pt]
&  \frac1{16} & \bf 0& \!\!\!-\frac1{12}\!\!  & \\[2pt]
&& \vdots \\
\end{array}
\end{align*}
\caption{Degrees of freedom around the vertices $A,C$} 
\label{fig:freeda}
\end{figure}
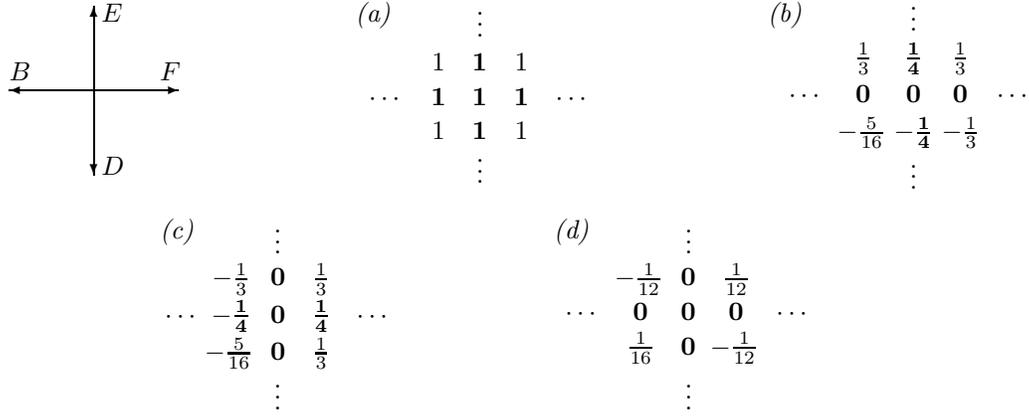

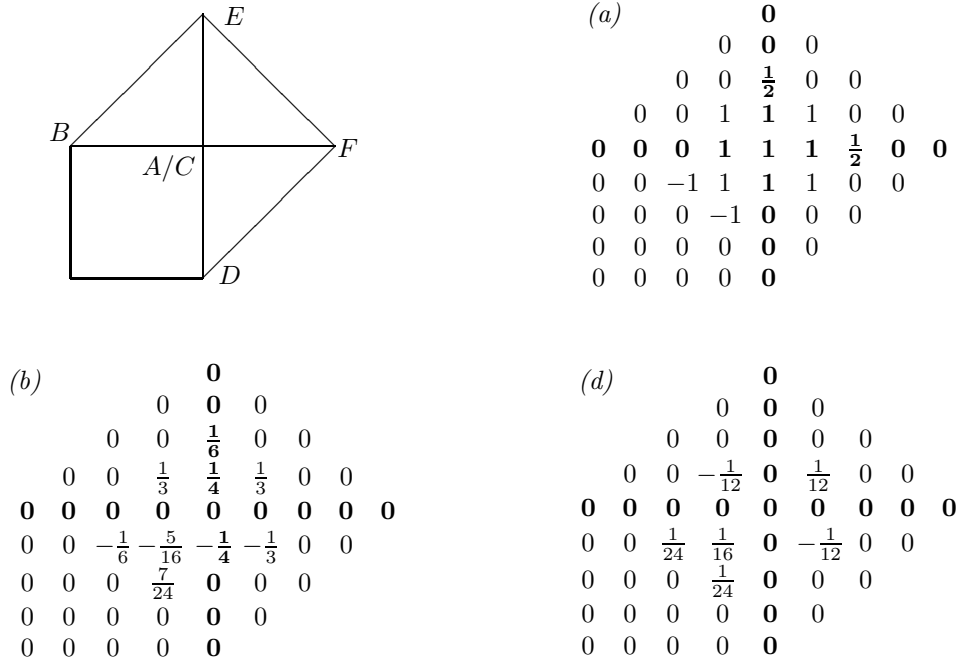
\begin{figure}
\begin{align*} \\[-6pt]
\begin{picture}(158,50)(-30,-3)
\put(0,0){\line(1,0){100}}  \put(0,0){\line(1,1){50}} 
\put(0,0){\line(0,-1){50}} \put(0,-50){\line(1,0){50}}
\put(50,-50){\line(0,1){100}} \put(50,-50){\line(1,1){50}} 
\put(50,50){\line(1,-1){50}}  \put(27,-11){$A/C$} 
\put(101,-4){$F$} \put(58,46){$E$}  
\put(-8,2){$B$} \put(56,-52){$D$}  
\put(195,47){\refpart{a}} 
\put(-24,-92){\refpart{b}}  \put(192,-92){\refpart{d}}  
\end{picture}  
&  \hspace{64pt}
\begin{array}{ccccccccc}
&&&& \bf 0 \\
&&&  0 & \bf 0 & 0 \\[1pt]
&&  0 & 0 & \bf \frac12 & 0 & 0 \\[1pt]
& 0 & 0 & 1 & \bf 1 & 1 & 0 & 0 \\[1pt]
\bf 0 & \bf 0& \bf 0 & \bf 1 & \bf 1 & \bf 1 & \bf \frac12 & \bf 0 &  \bf 0 \\[1pt]
0 & 0 & \!\!-1\!\! & 1 & \bf 1 & 1 & 0 & 0 \\
0 & 0 & 0 & \!\!-1\!\! & \bf 0 & 0 & 0  \\
0 & 0 & 0 & 0 & \bf 0 & 0 \\
0 & 0 & 0 &0 & \bf 0 \\
\end{array} \\[20pt]
\begin{array}{ccccccccc}
&&&& \bf 0 \\
&&&  0 & \bf 0 & 0 \\[1pt]
&& 0 & 0 & \bf \frac16 & 0 & 0 \\[2pt]
& 0 & 0 & \frac13 & \bf \frac14 & \frac13 & 0 & 0 \\[1pt]
\bf 0 & \bf 0& \bf 0 & \bf 0 & \bf 0 & \bf 0 & \bf 0 & \bf 0 & \bf 0 \\[1pt]
0 & 0 & \!\!-\frac16\!\!  & \!\!\!-\frac5{16}\!\! & \bf \!\!-\frac14\!\! & \!\!-\frac13\!\! & 0 & 0 \\[2pt]
0 & 0 & 0 & \frac{7}{24} & \bf 0 & 0 & 0  \\[1pt]
0 & 0 & 0 & 0 & \bf 0 & 0 \\
0 & 0 & 0 &0 & \bf 0 \\
\end{array}  
& \hspace{60pt}
\begin{array}{ccccccccc}
&&&& \bf 0 \\
&&&  0 & \bf 0 & 0 \\
&& 0& 0 & \bf 0 & 0 & 0 \\[1pt]
& 0 & 0 & \!\!\!-\frac1{12}\!\! & \bf 0 & \frac1{12} & 0 & 0 \\[1pt]
\bf 0 & \bf 0 & \bf 0& \bf 0 & \bf 0 & \bf 0 & \bf 0 & \bf 0 & \bf 0  \\[1pt]
0 & 0 & \!\frac{1}{24}\! & \frac1{16} & \bf 0 & \!\!-\frac1{12}\!\!\! & 0 & 0 \\[2pt]
0 & 0 & 0 & \frac1{24} & \bf 0 & 0 & 0 \\[1pt]
0 & 0 & 0 & 0 & \bf 0 & 0 \\
0 & 0 & 0 &0 & \bf 0 \\
\end{array} 
\end{align*}
\caption{Splines around the vertices $A,C$} 
\label{fig:splina}
\end{figure}

The degrees of freedom around $E$, $F$ are similar. 
Their lifts to splines on $\cH$ are presented in Figure \ref{fig:splinf}.
The lifted splines have degree 4 or 6; corresponding arrays of Bernstein-B\'ezier
coefficients are depicted.  The splines are build from copies of
the $M^1$-splines in Figures \ref{fig:splinef} and \ref{fig:splinfd},
with the degree adjusted and the interior coefficients around edges
symmetrized using $M^1$-splines like $u^3(1-u)^2\cdot(0,1,-1)$.
In particular, the spline \refpart{d} has more zero control coefficients around the vertices $A$, $C$.

The degrees of freedom around $B$, $D$ are presented in Figure \ref{fig:freedb},
each in the form of 3 blended $2\times 2$ arrays of corner control coefficients
of degree 6 polynomial functions on the incident polygons.
The 3 standard mixed derivatives are totally free. 
Lifts to splines on $\cH$ of degree 4 or 6 are presented in Figure \ref{fig:splinb}.
An independent spline corresponding to Figure \ref{fig:freedb}\refpart{f} is not depicted; 
one can take a mirror image of Figure \ref{fig:splinb}\refpart{d} 
for it. The variation of the Bernstein-B\'ezier 
coefficients of these splines to negative values is remarkable.
The splines in Figure \ref{fig:splinb} are build from copies of 
the $M^1$-splines in Figures \ref{fig:splinab} and \ref{fig:splinfd}.
The adjusting spline $10u^3(1-u)^2\cdot(0,1,-1)$ along a rectangle edge
looks as follows in the Bernstein-B\'ezier form: 
\begin{equation}
\begin{array}{c@\hg c@\hg c@\hg c@\hg c@\hg c@{\;}c@\hg c@{\;}c@\hg c@\hg c@\hg c@\hg c}
0 && 0 && 0 && \! -\frac12 \! && \! -\frac23 \! && 0 && 0 \\
 \bf 0 &&  \bf 0 && \bf 0 && \bf 0 && \bf 0 &&  \bf 0 && \bf 0 \\
 & 0 && 0 && 0 && 1 && 0 && 0  
\end{array}
\end{equation}

\begin{figure}
\begin{align*} \\[-7pt]
\begin{picture}(86,60)(-4,-40)
\put(2,-36){\line(1,1){72}}  \put(2,-36){\line(1,0){72}}  \put(2,-36){\line(0,1){72}}
\put(2,36){\line(1,-1){72}}   \put(2,36){\line(1,0){72}}  \put(74,-36){\line(0,1){72}}
\put(45,-3){$E$}  \put(77,-37){$F$} \put(77,31){$C$}  
\put(-9,31){$B$} \put(-8,-39){$A$}  
\put(142,25){\refpart{a}} \put(283,25){\refpart{b}}
\end{picture}  
\hspace{63pt} 
\begin{array}{c@\hh c@{\;}c@{\,}c@{\,}c@{\;} c@\hh c@\hh c@\hh c}
\bf 0 && 0 && 0 && 0 && \bf 0 \\
& \bf 0 && 0 && 0 && \bf 0 \\
0 && \bf  0 &&  \frac13 && \bf \frac16 && 0  \\
& 0 && \bf 0 && \bf \frac14 && 0  \\
0 && -\frac13 && \bf 0 && \frac13 && 0  \\
& 0 && \bf \!-\frac14 && \bf 0 && 0  \\
0 && \bf \!-\frac16 &&  \!-\frac13 && \bf 0  && 0  \\
& \bf 0 && 0 && 0 && \bf 0 \\
\bf 0 && 0 && 0 && 0 && \bf 0 \\
\end{array}  \hspace{-78pt} 
&  \hspace{96pt} 
\begin{array}{c@\hh c@\hh c@\hh c@{}c@{\,}c@\hh c@\hh c@\hh c}
\bf 0 && 0 && 0 && 0 && \bf 0 \\
& \bf 0 && 0 && 0 && \bf 0 \\
0 && \bf  0 &&  -\frac1{12} && \bf 0 && 0  \\
& 0 && \bf  0 && \bf 0 && 0  \\
0 && \frac1{12}\!\! && \bf 0 && \frac1{12}\!\! && 0  \\
& 0 && \bf 0 && \bf 0 && 0  \\
0 && \bf 0 &&  \!-\frac1{12} && \bf 0  && 0  \\
& \bf 0 && 0 && 0 && \bf 0 \\
\bf 0 && 0 && 0 && 0 && \bf 0 \\
\end{array}  \\[-74pt]  
\begin{picture}(86,60)(-4,0)
\put(2,-36){\line(1,1){72}}  \put(2,-36){\line(1,0){72}}  \put(2,-36){\line(0,1){72}}
\put(2,36){\line(1,-1){72}}   \put(2,36){\line(1,0){72}}  \put(74,-36){\line(0,1){72}}
\put(45,-3){$F$}  \put(77,-37){$E$} \put(76,33){$C$}  
\put(-9,30){$D$} \put(-9,-37){$A$}   
\put(-49,-46){\refpart{c}} \put(175,-46){\refpart{d}}
\end{picture}  \hspace{68pt} \\[46pt] \hspace{6pt}
\begin{array}{c@\hh c@{\,}c@{\,}c@{\,}c@{\,}c@\hh c@\hh c@\hh c@\hh c@\hh c@\hh c@\hh c}
\bf 0 && 0 && 0 && 0 && 0 && 0 && \bf 0 \\
& \bf 0 && 0 && 0 && 0 && 0 && \bf 0 \\
0 && \bf  0 &&  -\frac1{10} && 0 && \frac3{10} && \bf \frac15 && 0  \\
& 0 && \bf \frac15 && \frac25 && \frac7{10} && \bf \frac12 && 0  \\
0 && -\frac1{10} && \bf \frac35 && 1 && \bf \frac45 && \frac3{10} && 0  \\
& 0 && \frac25 && \bf 1 && \bf 1 && \frac7{10} && 0  \\
0 && 0 && 1 && \bf 1 && 1 && 0 && 0 \\
& 0 && \frac7{10} && \bf 1 && \bf 1 && \frac7{10} && 0  \\
0 && \frac3{10} && \bf \frac45 && 1 && \bf \frac45 && \frac3{10} && 0  \\
& 0 && \bf \frac12 && \frac7{10} && \frac7{10} && \bf \frac12 && 0 \\
0 && \bf \frac15 && \frac3{10} && 0 && \frac3{10} && \bf \frac15 && 0  \\
& \bf 0 && 0 && 0 && 0 && 0 && \bf 0 \\
\bf 0 && 0 && 0 && 0 && 0 && 0 && \bf 0
\end{array} \hspace{-6pt} & \hspace{28pt}
\begin{array}{c@\hh c@{\,}c@{\,}c@{}c@{\,}c@{}c@{}c@{}c@{}c@{}c@{\,}c@\hh c}
\bf 0 && 0 && 0 && 0 && 0 && 0 && \bf 0 \\
& \bf 0 && 0 && 0 && 0 && 0 && \bf 0 \\
0 && \bf  0 && \!-\frac1{40}\! && 0 && 0 
&& \!\bf 0\; && 0  \\
& 0 && \bf \frac1{20} && \!\frac{13}{120}\! && 0 
&& \,\bf 0 && 0  \\
0 && \!-\frac1{40} && \bf \frac2{15} && \frac15 && \bf\;0 && 0 
&& 0  \\
& 0 && \!\frac{13}{120} && \bf \frac16 && \bf \,0 && 0 
&& 0  \\
0 && 0 && \frac15 && \bf 0 && \!-\frac15 && 0 && 0 \\
& 0 && 0 
&& \bf 0\; && \bf \!\!-\frac16\! && -\frac15 && 0  \\
0 && 0 
&& \bf \,0 && \!-\frac15 && \bf \!-\frac15 && -\frac1{10} && 0  \\
& 0 && \bf 0 && 0 
&& -\frac1{5} && \bf \!\!-\frac3{20}\! && 0 \\
0 && \bf \,0 && 0 
&& 0 && -\frac1{10} && \bf \!\!-\frac1{15} && 0  \\
& \bf 0 && 0 && 0 && 0 && 0 && \bf 0 \\
\bf 0 && 0 && 0 && 0 && 0 && 0 && \bf 0
\end{array} \hspace{-2pt}
\end{align*}
\caption{Splines around the vertices $E,F$} 
\label{fig:splinf}
\end{figure}
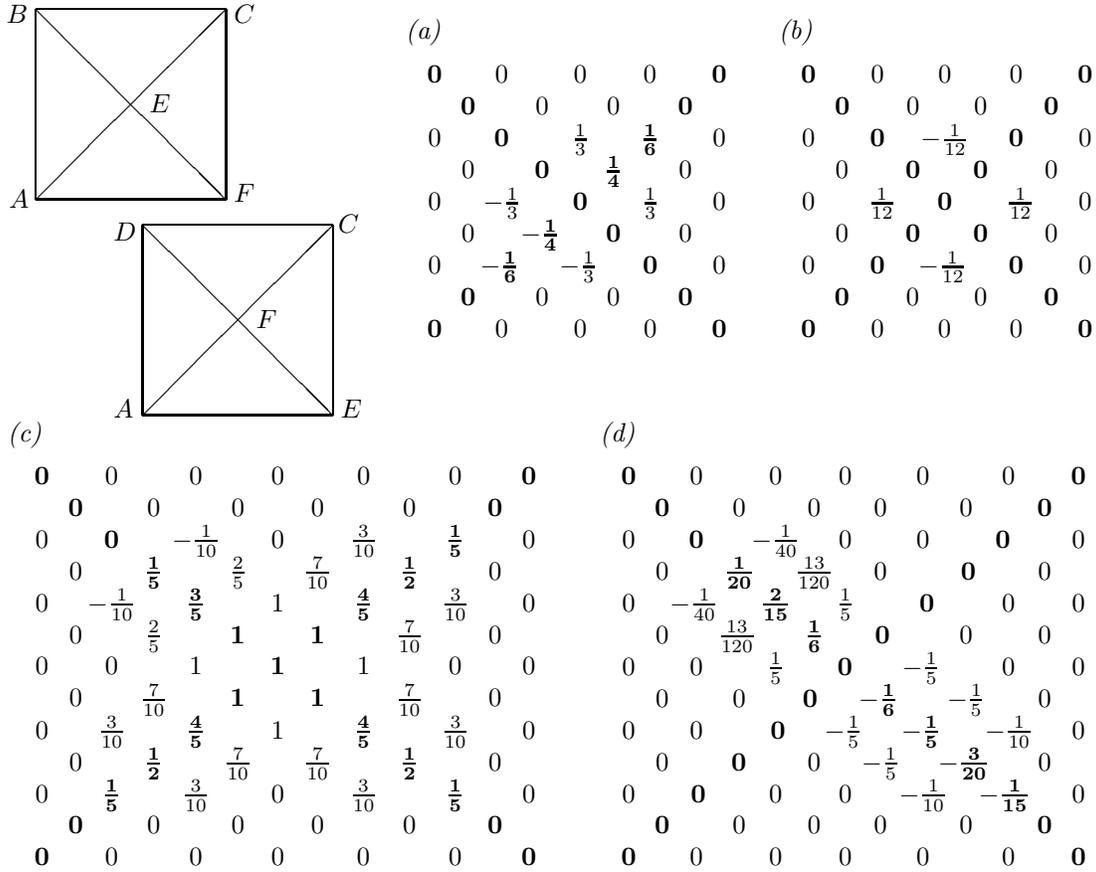

\begin{figure}
\begin{align*}
\begin{picture}(44,38)(-40,-4)
\put(0,0){\vector(-1,0){40}}  \put(0,0){\vector(1,2){18}}  
\put(0,0){\vector(1,-2){18}} 
\put(4,29){$A$}   \put(-40,4){$E/F$} \put(3,-36){$C$}  
\put(61,24){\refpart{a}} \put(167,24){\refpart{b}}  \put(279,24){\refpart{c}}  
\put(-5,-63){\refpart{d}} \put(108,-63){\refpart{e}}  \put(217,-63){\refpart{f}}  
\end{picture}  && \hspace{-38pt}
\begin{array}{c@{\;}c@{\;}c@{\;}c@{\;}c@{\,}c}
&&&&& \!\iddots \\
&& 1 && \bf 1 \\[1pt]
\cdots & \bf 1 && \bf 1 && 1 \\[1pt]
&& 1 && \bf 1 \\[-3pt]
&&&&& \ddots 
\end{array}
& & \hspace{-32pt}
\begin{array}{c@{\;}c@{}c@{\,}c@{}c@{}c}
&&&&& \!\iddots \\
&& \frac16 && \bf \frac16 \\[1pt]
\cdots\; & \bf 0 && \bf 0 && \! 0 \\[1pt]
&& -\frac16 && \bf -\frac16 \\[-2pt]
&&&&& \ddots 
\end{array}
& & \hspace{-28pt}
\begin{array}{c@{\;}c@{\;}c@{\;}c@{\,}c@{\,}c}
&&&&& \!\iddots \\
&& 0 && \bf 0 \\[1pt]
\cdots\; & \bf 0 && \bf 0 && \! \frac1{36} \\[1pt]
&& 0 && \bf 0 \\[-3pt]
&&&&& \ddots 
\end{array}  \\[10pt]
& \hspace{-12pt} 
\begin{array}{c@{\;}c@{\,}c@{\,}c@{}c@{}c}
&&&&& \!\iddots \\
&& \frac1{12} && \bf -\frac1{12} \\[1pt]
\cdots\; & \bf \frac16 && \bf 0 && \! -\frac16 \\[1pt]
&& \frac1{12} && \bf -\frac1{12} \\[-2pt]
&&&&& \ddots
\end{array}
& & \hspace{-24pt}
\begin{array}{c@{\;}c@{\,}c@{\,}c@{\,}c@{\,}c}
&&&&& \!\iddots \\
&& \frac1{30} && \bf 0 \\[1pt]
\cdots\; & \bf 0 && \bf 0 && \! 0 \\[1pt]
&& 0 && \bf 0 \\[-4pt]
&&&&& \ddots
\end{array}
& & \hspace{-30pt}
\begin{array}{c@{\;}c@{\,}c@{\,}c@{\,}c@{\,}c}
&&&&& \!\iddots \\
&& 0 && \bf 0 \\[1pt]
\cdots\; & \bf 0 && \bf 0 && \! 0 \\[1pt]
&& \frac1{30} && \bf 0 \\[-4pt]
&&&&& \ddots \\
\end{array}
\end{align*}
\caption{Degrees of freedom around the vertices $B$, $D$} 
\label{fig:freedb}
\end{figure}

\begin{figure}
\begin{align*}
\begin{picture}(156,20)(0,-32)
\put(0,0){\line(1,0){40}}  \put(40,0){\line(1,2){20}} \put(40,0){\line(1,-2){20}}
\put(0,0){\line(3,2){60}} \put(0,0){\line(3,-2){60}}
\put(60,-40){\line(1,2){20}} \put(60,40){\line(1,-2){20}}
\put(44,-2){$B$}  \put(63,-43){$C$} \put(64,36){$A$}  
\put(-8,3){$E$} \put(82,-3){$D$}  \put(153,30){\refpart{a}}
\end{picture}  
&  \hspace{-32pt}
\begin{array}{c@\hg c@{}c@\hg c@{}c@\hg c@{\,} c@\hg c@{\;} c@\hg c@\hg c@\hg c@\hg c@\hg c@\hg c@\hg c@\hg c@\hg c@\hg c@\hg c}
&&&&&&&&&&&&&&&&& \bf 0 \\
&&&&&&&&&&&&&& 0 && \bf 0 && 0 \\[-6pt]
&&&&&&&&&&& 0 && 0 && \bf \frac25 && 0 && \ddots \\
&&&&&&&&  0 && 0 && \frac25 && \bf \frac45 && \frac23 && 0  \\[-6pt] 
&&&&& 0 && 0 && 0 && \frac45 && \bf 1 && \frac65 && 0 && \ddots \\
&&0 && 0 && \frac35 && \frac{11}{10} &&  1 && \bf 1 && \frac{19}{15}  && 0 && 0 \\
\bf0&\bf 0 &&\bf\frac25 && \bf\frac45 && \bf 1 && \bf 1 && \bf 1 && 1 && 0 && 0 && 0 \\
&& 0 && 0 && \frac35 && \frac{11}{10} &&  1 && \bf 1 && \frac{19}{15}  && 0 && 0 \\[-6pt]
&&&&& 0 && 0 && 0 && \frac45 && \bf 1 && \frac65 && 0 && \!\iddots \\
&&&&&&&&  0 && 0 && \frac25 && \bf \frac45 && \frac23 && 0 \\[-6pt]
&&&&&&&&&&& 0 && 0 && \bf \frac25 && 0 &&  \!\iddots \\
&&&&&&&&&&&&&& 0 && \bf 0 && 0 \\
&&&&&&&&&&&&&&&&& \bf 0
\end{array}  \\[22pt] 
\begin{picture}(156,20)(0,-32)
\put(0,0){\line(1,0){40}}  \put(40,0){\line(1,2){20}} \put(40,0){\line(1,-2){20}}
\put(0,0){\line(3,2){60}} \put(0,0){\line(3,-2){60}}
\put(60,-40){\line(1,2){20}} \put(60,40){\line(1,-2){20}}
\put(44,-3){$D$}  \put(63,-43){$C$} \put(64,36){$A$}  
\put(-8,2){$F$} \put(82,-3){$B$}   \put(153,-14){\refpart{b}}
\put(-10,-116){\refpart{c}}  \put(181,-190){\refpart{d}}  \put(-10,-287){\refpart{e}}
\end{picture}  \\[-46pt]  & \hspace{-42pt}
\begin{array}{c@\hg c@{}c@\hg c@{}c@\hg c@{\,} c@\hg c@{\;} c@\hg c@\hg c@{}c@{}c@{\,}c@{}c@{}c@{}c@\hg c@\hg c@\hg c}
&&&&&&&&&&&&&&&&& \bf 0 \\
&&&&&&&&&&&&&& 0 && \bf 0 && 0 \\[-6pt]
&&&&&&&&&&& 0 && 0 && \bf -\frac2{15} && 0 && \ddots \\
&&&&&&&&  0 && 0 && -\frac1{24} && \bf -\frac9{40} && -\frac{17}{60} && 0  \\[-6pt] 
&&&&& 0 && 0 && 0 && -\frac1{60} && \bf -\frac15 && -\frac7{15} && 0 && \ddots \\
&&0 && 0 && \frac12 && \frac{41}{60} && \frac1{12} && \bf -\frac1{12} && -\frac25  && 0 && 0 \\
\bf0&\bf 0 &&\bf\frac13 && \bf\frac{11}{20} && \bf \frac{7}{15} && \bf \frac16 && \bf 0 && -\frac16 && 0 && 0 && 0 \\
&&0 && 0 && \frac12 && \frac{41}{60} && \frac1{12} && \bf -\frac1{12} && -\frac25  && 0 && 0 \\[-6pt]
&&&&& 0 && 0 && 0 && -\frac1{60} && \bf -\frac15 && -\frac7{15} && 0 && \!\iddots \\
&&&&&&&&  0 && 0 && -\frac1{24} && \bf -\frac9{40} && -\frac{17}{60} && 0  \\[-6pt]
&&&&&&&&&&& 0 && 0 && \bf -\frac2{15} && 0 &&  \!\iddots \\
&&&&&&&&&&&&&& 0 && \bf 0 && 0 \\
&&&&&&&&&&&&&&&&& \bf 0
\end{array}  \\[-50pt] 
\begin{array}{c@\hh c@{}c@{\,}c@{\,}c@{\;}c@{\,}c@{\,}c@{\;}c@{\!}c@{\!}c@{\,}c@\hg c@\hg c@\hg c@\hg c}
&&&&&&&&&&& \bf 0 \\
&&&&&&&& 0 && \bf 0 && 0 \\
&&&&& 0 && 0 && \bf \frac23 && 0 && 0 \\
&& 0 && 0 && \frac14 && \bf \frac14 && \frac{37}{24} && 0 && 0  \\
\bf 0 & \bf 0 && \bf 0 && \bf 0 && \bf 0 && 0 && 0 && 0 && 0 \\
&& 0 && 0 && \!\!-\frac14 && \bf \!-\frac14 && \!-\frac{37}{24} && 0 && 0 \\
&&&&& 0 &&  0 && \bf \!-\frac23  && 0 && 0 \\
&&&&&&&& 0 && \bf 0 && 0 \\
&&&&&&&&&&& \bf 0 \\
\end{array}  \;  \\[-50pt]
& \hspace{-54pt}
\begin{array}{c@\hg c@{}c@{}c@{}c@{}c@{\!}c@{\!}c@{}c@{\;}c@\hg c@\hg c@{\,}c@{}c@{}c@{}c@{}c@{\;}c@\hg c@\hg c}
&&&&&&&&&&&&&&&&& \bf 0 \\
&&&&&&&&&&&&&& 0 && \bf 0 && 0 \\[-6pt]
&&&&&&&&&&& 0 && 0 && \bf \!-\frac1{30} && 0 && \ddots \\
&&&&&&&&  0 && 0 && -\frac1{15} && \bf \!-\frac1{20} && -\frac{1}{30} && 0  \\[-6pt] 
&&&&& 0 && 0 && 0 && 0 && \bf \!\!-\frac1{30} && -\frac3{40} && 0 && \ddots \\
&&0 && 0 && 0 && \!-\frac{1}{20} && \frac1{30} && \bf 0 && -\frac1{15}  && 0 && 0 \\
\bf0&\bf 0 &&\bf -\frac1{30} && \bf-\frac{1}{20} && \bf \!-\frac{1}{30} && \bf 0 && \bf 0 && 0 && 0 && 0 && 0 \\
&&0 && 0 && \!-\frac1{10} && \!-\frac{1}{15} && 0 && \bf 0 && 0  && 0 && 0 \\[-6pt]
&&&&& 0 && 0 && 0 && 0 && \bf 0 && 0 && 0 && \!\iddots \\
&&&&&&&&  0 && 0 && 0 && \bf 0 && 0 && 0  \\[-6pt]
&&&&&&&&&&& 0 && 0 && \bf 0 && 0 &&  \!\iddots \\
&&&&&&&&&&&&&& 0 && \bf 0 && 0 \\
&&&&&&&&&&&&&&&&& \bf 0
\end{array}  \\[-53pt]  
\begin{array}{c@\hh c@{}c@{\,}c@{\,}c@{\;}c@{\,}c@{\,}c@{\;}c@{\!}c@{\!}c@{\,}c@\hg c@\hg c@\hg c@\hg c}
&&&&&&&&&&& \bf 0 \\
&&&&&&&& \,0\! && \bf 0 && 0 \\
&&&&& 0 && 0 && \bf \!-\frac1{12} && 0 && 0 \\
&& 0 && 0 && 0 && \bf 0 && \!-\frac{5}{24} && 0 && 0  \\
\bf 0\! & \bf 0 && \bf 0 && \bf 0 && \bf 0 && \frac1{16} && 0 && 0 && 0 \\
&& 0 && 0 && 0 && \bf 0 && \!-\frac{5}{24} && 0 && 0 \\
&&&&& 0 &&  0 && \bf \!-\frac1{12}  && 0 && 0 \\
&&&&&&&& \,0\! && \bf 0 && 0 \\
&&&&&&&&&&& \bf 0 \\
\end{array} \quad
\end{align*}
\caption{Splines around the vertices $B,D$} 
\label{fig:splinb}
\end{figure}

\subsection{The spline basis}
\label{sec:exbasis}

Now we can count the dimension of $S^1_k(\cH)$ and describe bases of the spline spaces.
We have 4 degrees of freedom around each crossing vertex as in Figure \ref{fig:freeda},
and 6 degrees of freedom around $B$ and $D$ as in Figure \ref{fig:freedb}.
The total of $4\cdot 4+2\times 6=28$ degrees of freedom is realized by splines of degree $\le 6$.
Up to the symmetries of the polygonal surface, these splines are depicted 
in Figures \ref{fig:splina},  \ref{fig:splinf},  \ref{fig:splinb}.
These are the splines characterized in \refpart{E1}. 

The dimension of $M_k^1$ spaces in \S \ref{sec:ef}, \S \ref{sec:ab}, \S \ref{sec:be} equals
$2k,2k+1,2k-1$, respectively. The kernels of $W_0\oplus W_1$-maps of Lemma \ref{lm:separate} 
are subspaces of splines thoroughly vanishing 
at the vertices. They have the dimension $2k-8,2k-8,2k-10$, respectively. 
These splines lift to $S^1_k(\cH)$  by assigning zero specializations 
in all other $M^1$-spaces and on the non-incident polygons,
giving us $22(k-4)-4$ independent splines in $S^1_k(\cH)$ for $k\ge 6$.
These are the splines characterized in \refpart{E2}.

In particular, for each edge and any $\ell\ge 2$ there is a spline of degree $\ell+3$ that evaluates 
to $(h_0,h_1,h_2)=u^\ell(1-u)^2\cdot (0,1,-1)$ in the $M^1$-space of that edge,  in terms of (\ref{eq:splines2}).
For $k\ge 5$, this gives $11(k-4)$ independent splines supported on just one edge. 
The splines corresponding to similar multiples of the syzygies $(2u,1,0)$, $(2u+u^2,1,0)$ 
can be linearly combined with the offset splines $U_1$, $\widetilde{U}_1$ or $V_1$ 
(and constant splines) to produce similarly vanishing splines supported only on one edge.
Alternatively, one can linearly combine ``adjacent" ($\ell\to\ell+1$) syzygy multiples to produce a spline 
thoroughly vanishing at both vertices. 
On the edges of \S \ref{sec:ef} and \S \ref{sec:ab}, 
polynomial multiples of the syzygy $(2u,1,0)$ give the splines
\begin{equation} \label{eq:esplina}
(h_0,h_1,h_2)=u^\ell(1-u)^{m-1}\cdot(1-u,2\ell-2(\ell+m)u,0)
\end{equation}
with $\ell\ge 2$, $m\ge 3$. 
Similarly, on \fc{EB} and \fc{FD} we have the splines
\begin{equation}
(h_0,h_1,h_2)=u^\ell(1-u)^{m-1}\cdot(1-u,(2+u)(\ell-\ell u-mu),0).
\end{equation}
For the edges of  \S \ref{sec:ef},
we can adjust (\ref{eq:esplina}) to 
\begin{equation}
(h_0,h_1,h_2)=u^\ell(1-u)^{m-1}\cdot(1-u,2\ell-2\ell u-mu,-mu)
\end{equation}
and allow $m\ge 2$.

Finally, we have $(k-3)^2+6\cdot (k-4)(k-5)/2$ independent splines that 
have only one non-zero control point, namely an interior control point 
of a restriction to the rectangle or to some triangle.
These splines are characterized in \refpart{E4}. 
There are no splines characterized in \refpart{E3} as $\cH$ has no boundary.

In total, we have 
\begin{align} \label{eq:gnedim}
\dim S^1_k(\cH)= & \, 28+22(k-4)-4+(k-3)^2+3(k-4)(k-5) \nonumber \\
= & \, (2k-3)^2+k-4
\end{align}
for $k\ge 6$. The dimension formulas in \S \ref{sec:dformula} give the same answer.
This formula 
holds for $k=4$ and $k=5$ as well, giving the dimensions $25$ and $50$, respectively. 
This can be foreseen in the context of the proof of Lemma \ref{lm:lbound}.
The additional  $\sum m_\cE$ relations (4 or 2, respectively)  
between the $J^{1,1}$-jets at the vertices $B$, $E$, $F$, $D$ are surely unrelated
because the edges \fc{EB} and \fc{FD} are not connected to each other.


\begin{figure}
\begin{align*} \\[-9pt] 
\begin{array}{c@{\;}c@{\;}c@{\;}c@{}c@{\!}c@{}c@{\;}c@{\;}c@{\;}c@{\;}c}
&  1 && 1 && -\frac2{13} && 0 && 0 \\
\bf 1 && \bf 1 && \bf \frac4{13} && \bf 0 && \bf 0 && \bf 0 \\
& 1 && 1 && -\frac2{13} && 0 && 0 
\end{array} \hspace{-4pt} & \hspace{37pt}
\begin{array}{c@{\;}c@{\;}c@{\;}c@{}c@{\!\!}c@{}c@{\;}c@{\;}c@{\;}c@{\;}c}
& 0 && \frac14 && -\frac7{195} && 0 && 0 \\
\bf 0 && \bf \frac15 && \bf \frac{14}{195} && \bf 0 && \bf 0 && \bf 0 \\
& 0 && \frac14 && -\frac7{195} && 0 && 0
\end{array} 
& \hspace{9pt}
\begin{array}{c@{\;}c@{\;}c@{\;}c@{\,}c@{}c@{\,}c@{\;}c@{\;}c@{\;}c@{\;}c}
& 0 && 0 && \frac{15}{13} && 1 && 1 \\
\bf 0 && \bf 0 && \bf \frac9{13} && \bf 1  && \bf 1 && \bf 1 \\
& 0 && 0 && \frac{15}{13} && 1 && 1 
\end{array} \, \\[30pt]
\begin{picture}(0,20)(-10,-3)
\put(-10,90){$\widetilde{V}_1:$}  \put(130,90){$\widetilde{V}_2:$} \put(280,90){$\widetilde{V}_5:$} 
\put(-10,23){$\widetilde{V}_6:$}  \put(130,23){$\widetilde{V}_8:$}  \put(280,23){$\widetilde{V}_9:$} 
\end{picture}  
\begin{array}{c@{\;}c@{\;}c@{\;}c@{}c@{\;}c@{\;}c@{\;}c@{\;}c@{\;}c@{\;}c@{\!}c@{\!}c}
& 0 && 0 && 1 && 0 && \!\!-\frac15 \\
\bf 0 && \bf 0 && \bf \frac{27}{40} && \bf \frac{5}{8} && \bf \frac15 && \bf 0 \\
& 0 && 0 && \frac{5}{4} &&\frac15 && 0 
\end{array}  & \hspace{34pt} 
\begin{array}{c@{\;}c@{\;}c@{\;}c@{\!}c@{\!}c@{}c@{}c@{\,}c@{\;}c@{\;}c}
& 0 && 0 && \!\!-\frac{2}{13} && 0 && 0  \\
\bf 0 && \bf 0 && \bf -\frac{9}{130} \! && \bf \!\!-\frac{1}{20} && \bf 0 && \bf 0 \\
& 0 && 0 && \!\!-\frac{1}{13} && \frac{1}{20} && 0 
\end{array} & \hspace{18pt}
\begin{array}{c@{\;}c@{\;}c@{\;}c@{\!}c@{\!}c@{}c@{}c@{\,}c@{\;}c@{\;}c}
& 0 && 0 && \!\!-\frac{1}{13} && \frac{1}{20} && 0  \\
\bf 0 && \bf 0 && \bf -\frac{9}{130} \! && \bf \!\!-\frac{1}{20} && \bf 0 && \bf 0 \\
& 0 && 0 && \!\!-\frac{2}{13} && 0 && 0 
\end{array} \hspace{-8pt}
\end{align*}
\caption{Alternative splines around \fc{EB} and \fc{FD}} 
\label{fig:splinfda}
\end{figure}

\subsection{Alternative splines around $EB$ and $FD$}
\label{sec:abe}

In \S \ref{sec:gdata} we chose the quadratic gluing data (\ref{eq:crossing3})--(\ref{eq:crossing3a})
to interpolate between the relations (\ref{eq:horder4}) and (\ref{eq:horder3}) of gluing around $E,F$
and $B,D$ in such a way that Theorem  \ref{cond:comp2} holds at the vertices $E$, $F$.
An alternative is to replace (\ref{eq:crossing3})--(\ref{eq:crossing3a}) by the fractional-linear gluing data 
\begin{equation} \label{eq:crossing3b}
\partial_{EA}+\partial_{EC}=\frac{6u_E^B}{3-u_E^B}\,\partial_{EB}, \qquad
\partial_{FA}+\partial_{FC}=\frac{6u_F^D}{3-u_F^D}\,\partial_{FD}.
\end{equation}
This leads to the syzygy module $\cZ(3-u,6u,3-u)$,
with the generators $(1,0,-1)$, \mbox{$(6u,u-3,0)$}.
With this modification, the dimension of $M_k^1(EB)$, $M_k^1(FD)$ becomes $2k$ for $k\le 1$.
This gives complete separation of vertices in degree 5 already.
The splines 
in Figure \ref{fig:splinfd} 
should be replaced by the splines in Figure \ref{fig:splinfda}.
The splines $V_3=U_3$, $V_4=U_4$, $V_7=U_7$ of degree 4 can be kept as 
$\widetilde{V}_3$, $\widetilde{V}_4$, $\widetilde{V}_7$.
The spline (\ref{eq:trivzero}) is in $M_5^1(EB)$ and $M_5^1(FD)$ as well.
The degrees of freedom around $E,F$ and in Figure \ref{fig:freedb} are not changed.
Modification of the degree 6 splines in Figures \ref{fig:splinf}, \ref{fig:splinb}
is left as an exercise. Dimension formula (\ref{eq:gnedim}) is adjusted to $(2k-3)^2+k-2$.

\section{Conclusions and perspectives}
\label{sec:practical}

In \S 
\ref{sec:gsplines} we gave a general definition of $G^1$ {\em polygonal surfaces} 
and $G^1$ functions on them. The $G^1$ gluing data is allowed to be specified 
in terms of transitions maps, jet bundles, tangent bundles or transversal vector fields.
The new conditions in Theorem \ref{cond:comp2} generalizing the balancing condition in \cite{Peters2010} 
are incorporated in the definition of $G^1$ polygonal surfaces.

We analyze particularly {\em rational} $G^1$ polygonal surfaces and $G^1$ polynomial splines on them.
The relation of these splines to syzygies and interdependence of the local $M^1$-splines defined on edges
are explored in \S \ref{sec:freedom}. A structured set of generators for the space $S^1(\cM)$ of splines 
is presented in \S \ref{sec:generate}. Section \ref{sec:boxdelta} defines the generators more explicitly
in the context of $G^1$ surfaces made of rectangles and triangles, and gives the dimension formula
for the spaces of splines of bounded degree on these surfaces.
The example of \S \ref{sec:poctah} derives spline generators explicitly,
and illustrates the technical details handsomely. 
The splines are constructed not by solving equations of $G^1$ continuity constraints,
but my matching the degrees of freedom at the vertices and local $M^1$-splines along the edges.

The practical value of a constructed $G^1$ 
spline space $\cS$ is measured by smoothness and fairness \cite{Peters2003}
of realizations in $\RR^3$  by the functions in $\cS$. Lemmas \ref{prop:Hg} and \ref{prop:Hgb} imply
that for any vertex of a rational $G^1$ polygonal surface $\cM$ there is a realization of $\cM$
by 
splines with a non-zero Jacobian (thus locally smooth) at that vertex.
Sections \ref{sec:serverts} and \ref{sec:generate} imply that the vertices
can be {\em completely separated} by the spline spaces. 
Part \refpart{iii} of Lemma \ref {lm:syzygy} implies that any edge and its neighborhood 
on both polygons have a locally smooth realization in $\RR^2$ by splines. 
Globally, the realizations can have self-intersections 
(unavoidable for some non-orientable surfaces \cite{cltopology})
that are difficult to control analytically. 
Locally smooth realization of the whole surface $\cM$ could be derived by approximation
properties of spline spaces, in the context of approximating a corresponding 
differential surface realized in $\RR^3$. 
The conclusion is that spaces of polynomial splines on rational $G^1$ surfaces
are adequate for smooth realization in $\RR^3$. 
If a corresponding differential surface can be topologically realized in $\RR^3$ 
without a self-intersection, the globally smooth realization should be possible with splines. 

Curvature fairness is an issue for low degree 
splines \cite{Peters2003}, \cite{Greiner94}.
Realized surfaces tend to be relatively flat along the edges and especially around the vertices.
In the context of rational $G^1$ surfaces,
the restrictions to edges are often of lower degree than the splines themselves.
For example, this is the case when triangles are glued and $\deg \ma>\deg \mb$.
Gluing data with constant $\mb(u_1)$ is easy to define and interpolate,
but seemingly it generates splines of low quality. 
In particular, syzygy generators like $(1,0,-1)$ should be avoided
because they do not influence edge parametrization and postpone richer spline spaces
to larger degree given fixed $\delta(\cE)$ in (\ref{eq:delt}). 
The syzygy generators of the same degree could be more desirable.
Common roots of $a(u_1)$, $c(u_1)$ in (\ref{eq:edgeabc}) close to the interval $[0,1]$
likely mean more unwanted variation of the splines.

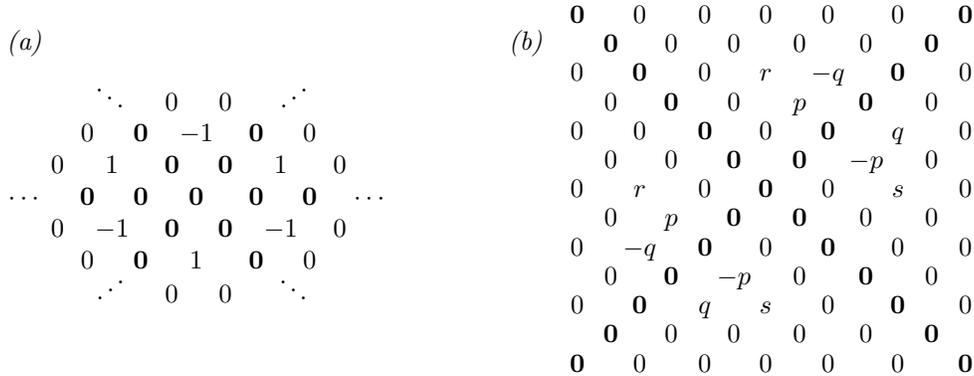
\begin{figure}
\[
\begin{array}{c@{\;}c@\hg c@{}c@{\,}c@\hg c@{}c@{\,}c@\hg c@{}c@{\,}c@\hg c@{\;}c}
&&& \bf \!\ddots && 0 && 0 && \bf \;\,\iddots\!\!\! \\
&& 0 && \bf 0 && -1 && \bf 0 && 0 \\ 
& 0 && 1 && \bf 0 && \bf 0 && 1 && 0 \\ 
\bf \cdots && \bf 0 && \bf 0 && \bf 0 && \bf 0 && \bf 0 && \bf \cdots \\
& 0 && -1 && \bf 0 && \bf 0 && -1 && 0 \\ 
&& 0 && \bf 0 && 1 && \bf 0 && 0 \\[-6pt] 
&&& \bf \!\iddots && 0 && 0 && \bf \;\,\ddots\!\!\! \\
\end{array} 
\begin{picture}(60,20)(-10,-3)
\put(-158,53){\refpart{a}} \put(32,53){\refpart{b}}  
\end{picture}  
\begin{array}{c@\hh c@{\,}c@{\;}c@\hh c@{\,}c@{\;}c@\hh c@{\,}c@{\,}c@{\;}c@\hh c@\hh c}
\bf  0 && 0 && 0 && 0 && 0 && 0 && \bf 0 \\[-1pt]
& \bf 0 && 0 && 0 && 0 && 0 && \bf 0 \\[-1pt]
0 && \bf  0 &&  0 && r && -q && \bf 0&& 0  \\[-1pt]
& 0 && \bf 0 &&0 && p && \bf 0 && 0  \\[-1pt]
0 && 0 && \bf 0 && 0 && \bf 0 && q && 0  \\[-1pt]
& 0 && 0 && \bf 0 && \bf 0 && -p && 0  \\[-1pt]
0 && r && 0 && \bf 0 && 0 && s && 0 \\[-1pt]
& 0 && p && \bf 0 && \bf 0 && 0 && 0  \\[-1pt]
0 && -q && \bf 0 && 0 && \bf 0 && 0 && 0  \\[-1pt]
& 0 && \bf 0 && -p && 0 && \bf 0 && 0 \\[-1pt]
0 && \bf 0 && q && s && 0 && \bf 0 && 0  \\[-1pt]
& \bf 0 && 0 && 0 && 0 && 0 && \bf 0 \\[-1pt]
\bf 0 && 0 && 0 && 0 && 0 && 0 && \bf 0
\end{array}
\]
\caption{A waving spline, and an adjusting spline} 
\label{fig:waving}
\end{figure}

A particular defect of $G^1$ surfaces is a ``waving" behavior around vertices of high valency
as in \cite[Figure 5]{Peters2010}. These saddle neighborhoods of high order are easily 
generated by multiple degrees of freedom in part \refpart{iii} of \refpart{D1}.
For example, at a vertex of even valency with all surrounding $\mb_k(u)=-1$
we have a spline whose only non-zero control points are the mixed derivatives alternating
between $1$ and $-1$. Figure \ref{fig:waving}\refpart{a} depicts such a spline around a vertex of valency 6.
To allow only simple saddle or convex neighborhoods of $C^2$ continuity at the vertices,
one may introduce a special class of polygonal surfaces with mixed geometric continuity,
by requiring $G^1$ continuity along the edges and additionally imposing $G^2$ continuity around the vertices. 
The conditions of \S \ref{sec:g1vertex} and \S \ref{sec:crossing} would then be strengthened 
to compatibility conditions of the $G^2$ gluing data, and Definition \ref{def:g1splines}
would be appended by $G^2$ continuity restrictions on the splines at the end-points of vertices.
The spaces $M(P)$ in (\ref{eq:wv}) will have to be extended to 6-dimensional $J^2$-jets,
while the spaces $M(\tau)$  in (\ref{eq:we}) would be appended by two utmost $O(v^2)$ terms.
The new $G^2$ restrictions on splines will be linear constraints on the image $J^2$-jets 
of adjusted $W_0$-maps in (\ref{eq:w0}). The dimension of $M^1$-spaces will not change generically,
as new restrictions on the crossing derivatives will be balanced by new $O(v^2)$ degrees of freedom.
Following Definition \ref{def:sepspline}, an enhanced kind of vertex separability would have to be introduced.
The spline space partitioning \refpart{D1}--\refpart{D5} would have to be adjusted then straightforwardly.

Expectedly, fairness issues can be addressed by higher degree splines. But it is harder to control the shape
with substantially more degrees of freedom. Shape optimization by energy minimization \cite{Greiner94}
is usually applied to whole surfaces in $\RR^3$. But this optimization could be applied to
filter spline spaces to manageable dimensions. For example, the splines in \refpart{D1} could be
optimally adjusted by contiguous splines in \refpart{D3}, \refpart{D5} while preserving the vanishing
properties at non-adjacent edges. 
For example, optimizing the spline in Figure \ref{fig:splinf}\refpart{c}
by minimizing the sum of four $L_2$-norms of the standard gradient 
$(\partial/\partial u,\partial/\partial v)$ on the four triangles leads to the linear adjustment of that spline
by Figure \ref{fig:waving}\refpart{b} with
\begin{equation}
(p,q,r,s)=\frac{1}{4630} \textstyle \left( 324+\frac14,  514+\frac14, 4865+\frac18, 3044+\frac{11}{24} \right).
\end{equation}
Subsequently, short sequences of nearest splines in \refpart{D3} 
could be optimally combined and adjusted by splines in \refpart{D5}, etc.
A subspace generated by the optimized splines should provide high-quality realizations
and manageable control. If shape optimization in $\RR^3$ is still needed, 
it should be then more numerically stable.

It is not necessary to build whole $G^1$ polygonal surfaces at once.
A practical approach is to consider {\em macro-patches} of polygons abstractly glued
along a portion of edges with $G^1$ continuity. The macro-patches fit Definition \ref{def:g1complex}
as $G^1$ polygonal surfaces with boundary, but the splines on them 
are supposed to thoroughly vanish 
along the boundary. The whole $G^1$ surface $\cM$ is built from overlaying macro-patches,
so that each interior vertex of $\cM$ is an interior vertex of some macro-patch.
A spline space on $\cM$ is defined by aggregating the spline spaces on the macro-patches.
For high enough degree, the aggregated spline space should concur with Definition \ref{def:g1splines}.
This approach allows to generalize the partitioning \refpart{D1}--\refpart{D5}.
For example, splines of sharp lower degree may be supported on a macro-patch rather
than on a set of polygons with a single interior vertex or an interior edge.
Or positive splines may need a larger support. 
Or performing an optimizing filtering could be more effective across macro-patches.
\begin{example} \rm
Here we consider the construction in \cite{flex} as a rational $G^1$ polygonal surface $\cQ$.
As in \cite{flex}, consider a topological mesh $\widehat{\cQ}$ of quadrangles.
Each quadrangle is represented by the rectangle $[0,1]^2$ which is subdivided by the middle
horizontal and vertical lines into 4 smaller rectangles. The quadrangles
are supposed to be parametrized by $C^1$ splines on the subdivided rectangle $[0,1]^2$, 
and the quadrangle parametrizations should meet with a specific $G^1$ continuity. 
Correspondingly, the polygonal surface $\cQ$ has 3 kinds of vertices:
\begin{enumerate}
\item The vertices of the original mesh $\widehat{\cQ}$, each of some valency $\ge 3$.
The interior vertices of valency 4 are {\em regular}.
The other interior vertices are {\em irregular}.
\item Crossing vertices represented by the midpoints of the edges of $\widehat{\cQ}$.
\item Crossing vertices represented the center of the subdivided rectangle $[0,1]^2$
for each quadrangle of $\widehat{\cQ}$. 
\end{enumerate}
At all vertices, the symmetric vertex gluing 
data as in (\ref{eq:hsym}) is assumed.
There are edges connecting vertices of the kinds \refpart{ii} and \refpart{iii},
representing the subdivision of $[0,1]^2$. The standard $C^1$ gluing data is assumed on them
(see Remark \ref{rm:parametric}). The other edges are half-edges of the original mesh $\widehat{\cQ}$.
They connect vertices of the kinds \refpart{ii} and \refpart{iii}, 
and the quadratic gluing data is assumed on them:
\begin{equation}
\ma(u)=q\,u^2, \qquad  \mb(u)=-1, 
\qquad \mbox{with } q=2\cos\frac{2\pi}n.
\end{equation}
Here we use a standard coordinate $u$ with $u=0$ being a vertex of type \refpart{ii},
and $u=1$ being a vertex of type \refpart{i} of valency $n$. 
This gluing data satisfies Theorem \ref{cond:comp2}. 
The main assertion in  \cite{flex} is that quartic splines on this surface provide
sufficiently many degrees of freedom for high quality surfaces.

The quadrangles are considered as macro-patches in \cite{flex}.
Sets of rectangles centered at the vertices of types \refpart{i}, \refpart{ii}
can be considered as macro-patches as well. 
Following our partitioning \refpart{E1}--\refpart{E4} of splines, we have the following quartic splines:
\begin{itemize}
\item For each quadrangle, we have a center vertex of  type \refpart{iii}
and corresponding $4$ splines around its center vertex as in \refpart{E1}, 
$8$ splines along the four $C^1$ edges as in \refpart{E2}, and $4$ splines on the smaller rectangles
as in \refpart{E4}. These splines thoroughly vanish on the boundary of 
$[0,1]^2$. 
\item Around each regular vertex of $\widehat{\cQ}$, we have the standard $C^1$ continuity
and similar $4+8$ vertex and edge splines.
\item A half-edge of $\widehat{\cQ}$ connecting an irregular vertex $\cP$, 
the corresponding $M^1_4$-space has dimension $9$. 
However, the 9 degrees of freedom as in (\ref{eq:mmmm}) are not met
because there is a spline as in (\ref{eq:trivzero}). 
The vertices are therefore not completely separated.
There is a relation between corner control coefficients like (\ref{eq:quarel2}):
\begin{equation} \label{eq:qrel}
12q\,h_0(1)-12q\,h'_0(1)=12q\,h_0(0)+4q\,h'_0(0)-h'_1(1)-h'_2(1).
\end{equation}
If the half-edge leads from $\cP$ to a regular vertex of $\widehat{\cQ}$, 
this additional restriction is localized to a macro-patch centered at $\cP$ and 
at the adjacent vertices of type \refpart{ii}. 
If $\cP$ is connected to other irregular vertices of $\widehat{\cQ}$,
some generating splines may need to be globally defined.
\end{itemize}
\end{example}
A dimension formula can be derived using the results of \S \ref{sec:dformula}.
If the mesh $\widehat{\cQ}$ defines an orientable surface, 
we use Theorem \ref{th:gdimform}  and apply formulas (\ref{eq:polygee}), (\ref{eq:topolov}) 
to the topology of $\widehat{\cQ}$ to get this dimension expression for $k\ge 5$:
\begin{align}
\dim S^1_k(\cQ)= & \; 16-16g\big(\widehat{\cQ}\big) + \left( k-1 \right)^2 \#\{\mbox{\rm quadrangles}\}
-5\,\#\{\mbox{\rm irregular vertices}\}  \nonumber  \\[1pt]
& +(2k-3) \, \#\{\mbox{\rm boundary edges of } \widehat{\cQ}\} -8\,\#\{\mbox{\rm boundary components}\} .
\end{align}
If $k=4$ and there are dependancies between the relations (\ref{eq:qrel}),
this formula is an underestimate. This is possible only if there are cycles of edges 
connecting only irregular vertices.

\small

\bibliographystyle{plain}
\bibliography{cg-references}

\end{document}